\documentclass[reqno,12pts]{amsart}
\usepackage{amssymb,amsmath,amsthm,amstext,amsfonts}
\usepackage{amsmath,amstext,amsthm,amsfonts}
\usepackage[dvips]{graphicx}
\usepackage[dvips]{graphicx}
\usepackage{color}

\makeatletter \@addtoreset{equation}{section} \makeatother

\renewcommand\thetable{\thesection.\@arabic\c@table}

\theoremstyle{plain}
\newtheorem{maintheorem}{Theorem}

\newtheorem{maincorollary}{Corollary}

\newtheorem{theorem}{Theorem }[section]
\newtheorem{proposition}[theorem]{Proposition}
\newtheorem{lemma}[theorem]{Lemma}
\newtheorem{corollary}[theorem]{Corollary}

\theoremstyle{definition} \theoremstyle{remark}
\newtheorem{remark}[theorem]{Remark}

\newtheorem{definition}[theorem]{Definition}

\newcommand{\field}[1]{\mathbb{#1}}
\newcommand{\real}{\field{R}}

\renewcommand{\natural}{\field{N}}

\newcommand{\torus}{\field{T}}

\newcommand{\al} {\alpha}       
        
\newcommand{\ga} {\gamma}    
\newcommand{\de} {\delta}       \newcommand{\De}{\Delta}

\newcommand{\vep}{\varepsilon}

\newcommand{\la} {\lambda}      \newcommand{\La}{\Lambda}

\newcommand{\si} {\sigma}

\newcommand{\R}{\mathbb{R}}
\newcommand{\supp}{\operatorname{supp}}
\newcommand{\diam}{\operatorname{diam}}

\newcommand{\topp}{\operatorname{top}}

\newcommand{\ov}{\overline}

\renewcommand{\field}[1]{\mathbb{#1}}

\newcommand{\re}{\field{R}}

\renewcommand{\natural}{\field{N}}
\newcommand{\vr}{\varphi}

\newcommand{\Ptop}{P_{\topp}}

\newcommand{\cL}{\mathcal{L}}
\newcommand{\cE}{\mathcal{E}}
\newcommand{\cM}{\mathcal{M}}

\newcommand{\cF}{\mathcal{F}}

\newcommand{\cW}{\mathcal{W}}
\newcommand{\cU}{\mathcal{U}}

\newcommand{\cA}{\mathcal{A}}

\def\ds{\displaystyle}

\begin{document}

\title{Differentiability of thermodynamical quantities in non-uniformly expanding dynamics}
\author{ T. Bomfim and A. Castro and P. Varandas}

\address{Thiago Bomfim, Departamento de Matem\'atica, Universidade Federal da Bahia\\
Av. Ademar de Barros s/n, 40170-110 Salvador, Brazil.}
\email{tbnunes@ufba.br}
\urladdr{https://sites.google.com/site/homepageofthiagobomfim/}

\address{Armando Castro, Departamento de Matem\'atica, Universidade Federal da Bahia\\
Av. Ademar de Barros s/n, 40170-110 Salvador, Brazil.}
\email{armando@impa.br}

\address{Paulo Varandas, Departamento de Matem\'atica, Universidade Federal da Bahia\\
Av. Ademar de Barros s/n, 40170-110 Salvador, Brazil.}
\email{paulo.varandas@ufba.br }
\urladdr{http://www.pgmat.ufba.br/varandas}

\date{\today}

\begin{abstract}
In this paper we study the ergodic theory of a robust non-uniformly expanding maps where no Markov
assumption is required. We prove that the topological pressure is differentiable as a function of the dynamics
and analytic with respect to the potential. Moreover we not only prove the continuity of the equilibrium states
and their metric entropy as well as the differentiability of the maximal entropy measure and extremal Lyapunov exponents with respect to the dynamics. We also prove a local large deviations principle and central limit
theorem and show that the rate function, mean and variance vary continuously with respect to observables, potentials
and dynamics. Finally, we show that the correlation function associated to the maximal entropy measure is differentiable with respect to the dynamics and it is $C^1$-convergent to zero.
In addition, precise formulas for the derivatives of thermodynamical quantities are given.
\end{abstract}

\subjclass[2000]{37A35, 37C30, 37C40, 37D25, 60F}
\keywords{Non-uniform hyperbolicity, regularity of topological pressure, linear response formula, regularity of limit theorems}

\maketitle

\section{Introduction}

The thermodynamical formalism was brought from statistical mechanics to dynamical systems
by the pioneering works of Sinai, Ruelle and Bowen \cite{Si72,Bo75,BR75} in the mid seventies.
Indeed, the correspondence between one-dimensional lattices and uniformly hyperbolic maps, via
Markov partitions, allowed to translate and introduce several notions of Gibbs measures and
equilibrium states in the realm of dynamical systems.
Nevertheless, although uniformly hyperbolic dynamics arise in physical systems
(see e.g. \cite{HM03}) they do not include some relevant classes of
systems including the Manneville-Pomeau transformation (phenomena of
intermittency), H\'enon maps and billiards with convex scatterers.
We note that all the previous systems present some non-uniformly
hyperbolic behavior and its relevant measure exhibits a weak
Gibbs property.
More recently there have been established
many evidences that non-uniformly hyperbolic dynamical systems admit
countable and generating Markov partitions. This is now parallel to the development
of a thermodynamical formalism of gases with infinitely many states.
We refer the reader to~\cite{Sa99,Pin11} for some recent progress in this direction.

A cornerstone of the theory that has driven the recent attention of many authors both in the physics and
mathematics literature concerns the differentiability of thermodynamical quantities as the topological pressure,
SRB measures or equilibrium states with respect to the underlying dynamical system.
Results on the differentiability of the topological entropy and pressure
include some important contributions by Walters~\cite{Wa92} and
Katok, Knieper, Pollicott and Weiss ~\cite{KKPW89} on the differentiability of the topological
entropy for Anosov and geodesic flows.
The differentiability of the SRB measure or equilibrium state with respect to the dynamical system
has been referred, for natural reasons, as linear response formulas (see e.g. \cite{Rue09}).
This has proved to be a hard subject not yet completely understood. In fact, this has been studied mostly for uniformly hyperbolic diffeomorphisms and flows in~\cite{KKPW89, Rue97, BL07, Ji12}, for the SRB measure of some partially hyperbolic difeomorphisms in~\cite{Dol04} and for one-dimensional piecewise expanding and quadratic maps in~\cite{Rue05, BS08, BS09, BS12, BBS13} and a general picture is still far from complete. In this paper we address these questions for robust classes of non-uniformly expanding maps.

If, on the one hand, the study of finer statistical properties of thermodynamical quantities as equilibrium states,
mixing properties, large deviation and limit theorems,  stability under deterministic perturbations or regularity of the
topological pressure is usually associated to good spectral properties of the Ruelle-Perron-Frobenius operator, on the other hand neither the stability of the equilibrium states or differentiability results for thermodynamical
quantities could follow directly from the spectral gap property. This is due to the fact that transfer operators acting on the space of H\"older continuous potentials may not even vary continuously with the dynamical system (see e.g. Subsection~\ref{ex:discontinuous} for an example), which makes the classical perturbation theory hard to apply.
Revealing its fundamental importance,  the functional analytic approach to thermodynamical formalism
has gained special interest in the last few years and produced new and interesting results even in both uniformly and non-uniformly hyperbolic setting (see e.g. \cite{KL99, BKL01, Cas02,Cas04, GL06,BT07,DL08,BG09,BG10,Ru10,CV13, DZ13, CN15}).

In this article we address the study of linear response formulas, continuity and differentiability of several thermodynamical quantities and limit theorems for the robustly nonuniformly expanding maps of \cite{VV10, CV13}. These are local
diffeomorphisms that admit the coexistence of expanding and contracting behavior and need not admit any Markov partition.
Such classes of maps include important classes of examples as bifurcation of expanding homeomorphisms,
subshifts of finite type or intermittency phenomena in the Maneville-Pommeau maps (see Section~\ref{s.examples}).
One of the difficulties is that these dynamical systems are not topologically conjugate.

Our strategy builds on the Birkhoff's method of projective cones.
In \cite{CV13}  this method was applied to prove that the Ruelle-Perron-Frobenius operator acting on both the Banach spaces of H\"older continuous and smooth observables has a spectral gap and to prove continuous dependence of the topological pressure with respect to the dynamical system and the potential.

Our starting point here is to prove that the Ruelle-Perron-Frobenius operator $\cL_{f,\phi}$ associated to local
diffeomorphisms  is always differentiable with respect to the dynamics and potential as an a linear operator in
$L(C^{r+\al}(M,\mathbb R), C^{r-1}(M,\mathbb R))$ for $r\ge 1$. Recall the transfer operator may be discontinuous
as an element in $L(C^{r}(M,\mathbb R), C^{r}(M,\mathbb R))$. Then we deduce a chain-rule like formula for the
derivative of the transfer operators $\cL^n_{f,\phi}$ as operators in $L(C^{r+\al}(M,\mathbb R), C^{r-1}(M,\mathbb R))$.
The program to prove differentiability of many thermodynamical quantities is to prove that these quantities
are limit of differentiable objects with uniform derivatives. In the case of the topological pressure and the
equilibrium states we prove that topological pressure $\Ptop(f,\phi)$ and equilibrium states $\mu_{f,\phi}$ can be
obtained as limits involving $\cL^n_{f,\phi} (1)$ of the $(\cL_{f,\phi}^{*})^n \eta$, where $\cL_{f,\phi}^{*}$ stands for the dual
of the transfer operator and $\eta$ is a fixed probability measure. This has been carried out with success and we are
able to provide uniform bounds for the derivatives of this expressions in terms of derivatives involving the transfer operator.
To the best of our knowledge these are the first differentiability formulas for the topological pressure
and equilibrium states for multidimensional non-uniformly expanding maps.

The second main goal of this article is to prove stability of the limit laws with respect to the dynamics and potential.
On the one hand the spectral gap property and the the differentiability of the topological pressure is well known to
imply a central limit theorem and a local large deviations principle. On the other hand, we prove that the correlation
function associated to the maximal entropy measure is smooth with respect to the dynamical system $f$ and
convergent to zero. In consequence, we deduce that the mean and variance in the central limit theorem vary
smoothly with respect to $f$. Continuity results are obtained for more general equilibrium states.
In addition, we use a continuous inverse mapping theorem for fibered maps to deduce that the large deviations
rate function vary continuously with respect to the dynamical system and potential. To the best of our knowledge
the previous results are new even in the uniformly hyperbolic setting.

This paper is organized as follows. In Section~\ref{s.statements}, we describe the setting of our results and
state our main results on the regularity of the thermodynamical quantities and the applications to the stability
of the limit laws. Some preliminary results are given in Section~\ref{s.preliminaries}, while our main results concerning the differentiability of topological pressure, conformal measures and equilibrium states are proven in Section~\ref{sec:diff}. In Sections~\ref{subsec:smooth.cor} and~\ref{sec:CLT} we prove that the correlation function
is convergent to zero in the $C^1$-topology and obtain the differentiability of mean and variance in the central limit theorem. A local large deviations principle and the regularity of the rate function is discussed in Section~\ref{sec:deviations} . Finally, some applications and examples are discussed in Section~\ref{s.examples}.

\section{Statement of the main results}\label{s.statements}

\subsection{Setting}

In this section we introduce some definitions and establish the setting.
Let $M$ be compact and connected Riemannian manifold of dimension $m$ with distance $d$.
Let $f:M \to M$ be a \emph{local homeomorphism} and assume that there exists a  function
$x\mapsto L(x)$ such that, for every $x\in M$ there is a
neighborhood $U_x$ of $x$ so that $f_x : U_x \to f(U_x)$ is
invertible and
$$
d(f_x^{-1}(y),f_x^{-1}(z))
    \leq L(x) \;d(y,z), \quad \forall y,z\in f(U_x).
$$
In particular every point has the same finite number of preimages $\deg(f)$ which coincides with the
degree of $f$.

For all our results we assume that $f$ satisfies conditions
(H1) and  (H2) below. Assume there are constants $\si>1$
and $L\ge 1$, and an open region $\cA\subset M$
such that \vspace{.1cm}
\begin{itemize}
\item[(H1)] $L(x)\leq L$ for every $x \in \cA$ and
$L(x)< \sigma^{-1}$ for all $x\notin \cA$, and $L$ is
close to $1$: the precise condition is given in \eqref{eq.vep} and \eqref{eq.vepp}.
\item[(H2)] There exists a finite covering $\cU$ of $M$ by open domains of injectivity for $f$
such that $\cA$ can be covered by $q<\deg(f)$ elements of $\mathcal U$. \vspace{.1cm}
\end{itemize}
The first condition means that we allow expanding and contracting
behavior to coexist in $M$: $f$ is uniformly expanding outside $\cA$
and not too contracting inside $\cA$. In the case that  $\cA$ is empty then
$f$ is uniformly expanding. The second condition requires
that every point has at least one preimage in the expanding region.

An observable $g: M\to \mathbb R$ is $\al$-H\"older continuous if the H\"older constant
$$
|g|_\al
	=\sup_{x\neq y} \frac{|g(x)-g(y)|}{d(x,y)^\al}
$$
is finite. As usual, given $r\in\mathbb N_0$ and $\al\in (0,1)$ we endow the space $C^{r+\al}(M, \mathbb R)$ of
$C^r$ observables $g$ such that $D^r g$ is $\al$-H\"older continuous with the norm
$\|\cdot\|_{r,\al}=\|\cdot \|_r+|\cdot|_\al$. We write for simplicity $\|\cdot\|_\al$ for the case that $r=0$.
Throughout, we let $\phi:M \to \mathbb R$ denote a potential  at least H\"older continuous and satisfying
either
\begin{itemize}
\item[(P)] $\sup\phi-\inf\phi<\vep_\phi$ \quad  and \quad $|e^\phi|_{\alpha}  <\vep_\phi \;e^{\inf \phi}$
\end{itemize}
provided that $\phi$ is $\al$-H\"older continuous, or
\begin{itemize}
\item[(P')] $\sup\phi-\inf\phi<\vep_\phi$ \quad  and \quad $\max_{s\le r}\|D^s \phi\|_0 <\vep_\phi$
\end{itemize}
if $\phi$ is $C^r$, where $\vep_\phi>0$ depends only on $L$, $\si$, $q$, $\deg(f)$, $r$, a positive integer $m$
and small $\de>0$ stated precisely in \cite{CV13} (see  equations \eqref{eq.vep} and \eqref{eq.vepp} below).
These are open conditions on the set of potentials, satisfied by constant potentials.  In particular we
can consider measures of maximal entropy and equilibrium states associated to potentials $\beta \phi$ with
$\phi$ at least H\"older continuous and $\beta$ small, which in the physics literature is known as the high temperature setting.

Throughout the paper we shall denote by $\cF$ an open set of local homeomorphisms with Lipschitz inverse and $\cW$ be some family of H\"older continuous potentials satisfying $(H1), (H2)$ and $(P)$ with uniform constants.
Moreover, we shall denote by $\cF^{r+\al}$ an open set of $C^{r+\al}$ local diffeomorphisms
such that (H1) and (H2) hold with uniform constants and their inverse branches are $C^{r+\alpha}$,
and $\cW^{r+\al}$ to denote an open set of $C^{r+\al}$ potentials such that (P) or (P') holds. We notice that the higher regularity of the dynamics is used to deduce the regularity of the inverse branches,
which are related with the Perron-Frobenius operator.
We shall always use the term {\em differentiable} to mean
{\em $C^1$-differentiable}.

\subsection{Strong statistical properties of equilibrium states}\label{existence eq. states}

Let us first introduce the necessary definitions and collect from \cite{VV10,CV13} some results
on the existence and statistical properties of equilibrium states for this robust class
of transformations.
Given a continuous map $f:M\to M$ and a potential $\phi:M \to \mathbb R$,
the variational principle for the pressure asserts that
\begin{equation*}
\label{variational principle} \Ptop(f,\phi)=\sup \left\{
h_\mu(f)+\int \phi \;d\mu : \mu \;\text{is}\; f\text{-invariant}
\right\}
\end{equation*}
where $\Ptop(f,\phi)$ denotes the topological pressure of $f$ with
respect to $\phi$ and $h_\mu(f)$ denotes the metric entropy. An
\textit{equilibrium state} for $f$ with respect to $\phi$ is an
invariant measure that attains the supremum in the right hand side
above.

In our setting equilibrium states arise as invariant measures absolutely continuous
with respect to an expanding, conformal and non-lacunary Gibbs measure $\nu$.
Since we will not use these notions here we shall refer the reader to \cite{VV10} for precise
definitions and details.
Many important properties arise from the study of transfer operators. We consider the Ruelle-Perron-Fr\"obenius
transfer operator $\cL_{f,\phi}$ associated to  $f:M\to M$ and $\phi:M\to\real$ as the linear operator defined
on a Banach space $X \subset C^0(M,\mathbb R)$ of continuous functions $g:M\to\real$ and given by
$$
\cL_{f,\phi} (g)(x) = \sum_{f(y)=x} e^{\phi(y)} g(y).
$$
Since $f$ is a local homeomorphism it is clear that $\cL_{f,\phi} g$ is continuous for every continuous $g$
and, furthermore, $\cL_{f,\phi}$  is indeed a bounded operator relative to the norm of uniform convergence in
$C^0(M,\mathbb R)$ because
$
\|\cL_{f,\phi}\| \le  \deg(f) \; e^{\sup|\phi|}.
$
Analogously, $\cL_{f,\phi}$ preserves the Banach space $C^{r+\alpha}(M,\mathbb R)$, with $r+\al>0$, provided that $\phi$ is $C^{r+\al}$. Moreover, it is not hard to check that $\cL_{f,\phi}$ is a bounded linear operator in the Banach space $C^{r}(M,\mathbb R)\subset C^0(M,\mathbb R)$ ($r\ge 1$) endowed with the norm $\|\cdot\|_r$
whenever $f$ is a $C^r$-local diffeomorphism and $\phi\in C^{r}(M,\mathbb R)$.

We say that the Ruelle-Perron-Frobenius operator $\cL_{f,\phi}$ acting on a Banach space $X$ has the
\emph{spectral gap property} if there exists a decomposition of its spectrum
$\sigma(\cL_{f,\phi})\subset \mathbb C$ as follows: $\sigma(\cL_{f,\phi})=\{\lambda_1\}\cup \Sigma_1$ where
$\lambda_1$ is a leading eigenvalue for $\cL_{f,\phi}$ with one-dimensional associated eigenspace
and there exists $0 < \lambda_{0} < \lambda_{1}$ such that $\Sigma_{1} \subset \{ z\in \mathbb C : |z|<\lambda_0 \}$.
When no confusion is possible, for notational simplicity  we omit the dependence of the Perron-Frobenius
operator on $f$ or $\phi$.
We build over the following result which is a consequence of the results in \cite{VV10,CV13}.

\begin{theorem}\label{thm.oldstuff}
Let $f:M \to M$ be a local homeomorphism with
Lipschitz continuous inverse satisfying (H1) and (H2), and let $\phi:M \to \R$
be a H\"older continuous potential  such that (P) holds. Then
\begin{enumerate}
\item  there exists a unique equilibrium state $\mu_{f,\phi}$ for $f$ with respect to $\phi$,
it is expanding, exact and absolutely continuous with respect to some conformal,  non-lacunary
Gibbs measure $\nu_{f,\phi}$;
\item  the Ruelle-Perron-Frobenius has a  spectral gap property in the space of H\"older
continuous observables and the density $d\mu_{f,\phi}/d\nu_{f,\phi}$ is H\"older;
\item $\Ptop(f,\phi) = \log \lambda_{f , \phi}$, where $\lambda_{f , \phi}$ is the spectral radius of the Ruelle-Perron-Frobenius;
\item the topological pressure function $\cF \times \cW  \ni (f,\phi)  \to \Ptop(f,\phi)$ is continuous;
\item  the invariant density function $\cF \times \cW  \to  C^\alpha(M,\mathbb R)$ given by
 $(f, \phi)  \mapsto  \frac{d\mu_{f,\phi} }{d\nu_{f,\phi}}$
is continuous, where $C^\al(M,\mathbb R)$ is endowed with the $C^0$ topology.
\end{enumerate}
If, in addition, the potential $\phi: M \to \mathbb R$ is $C^r$-differentiable and satisfies (P')
then
\begin{enumerate}
\item[(5)]  the Ruelle-Perron-Frobenius has a  spectral gap property in the space of
$C^r$-observables and the density $d\mu_{f,\phi}/d\nu_{f,\phi}$ belongs to $C^r(M, \mathbb R)$;
\item[(6)] the topological pressure $\cF^r \times \cW^r  \ni (f,\phi)  \to \Ptop(f,\phi)$ and the invariant density
function $\cF^r \times \cW^r  \to  C^{r}(M,\mathbb R)$ given by $(f, \phi) \mapsto \frac{d\mu_{f,\phi} }{d\nu_{f,\phi}}$
vary continuously in the $C^r$ topology;
\item[(7)] the conformal measure function $\cF^r \times \cW^r  \to  \cM(M)$ given
by   $(f, \phi) \mapsto \nu_{f,\phi}$
is continuous in the weak$^*$ topology. In consequence, the equilibrium measure
$\mu_{f,\phi}$ varies continuously in the weak$^*$ topology;
\end{enumerate}
\end{theorem}

Let us mention that condition (1) above holds more generally for all H\"older continuous
potentials such that $\sup\phi-\inf\phi<\log\deg(f)-\log q$ (see \cite[Theorem~A]{VV10}.
The aforementioned results lead to the natural questions about the regularity of some thermodynamical
quantities as the topological pressure, equilibrium states, Lyapunov exponents, entropy, the central limit theorem, the large deviation rate function or the correlation function when one perturbs the dynamics $f$ or the potential $\phi$.
Our purpose in the present paper is to address these questions in this non-uniformly expanding context.

\subsection{Statement of the main results}

\subsubsection{Spectral theory of transfer operators} Our first result addresses the problem of the regularity
of the transfer operators with respect to the dynamical system given by a local diffeomorphism. As discussed
before,
in general the Koopman operator when acting in the space of $C^{r+\alpha}$-observables is not differentiable with respect to the dynamics in the operator norm topology. This implies also on the lack of differentiability for the transfer operators. Nevertheless, we get the following:

\begin{maintheorem}{(Differentiability of transfer operator)}\label{thm:general.differentiability}
Let  $M$ be a compact connected Riemannian manifold and 
$\phi\in C^{r+ \alpha}(M,\mathbb R)$
be any fixed potential, with $r \geq 1$ and $\al>0$. Then the map
$$
\begin{array}{ccc}
{Diff}_{loc}^{r+\al}(M) & \to & L(C^{r+\al}(M, \mathbb R), C^{r-1}(M,\mathbb R))\\
f &  \mapsto & \cL_{f, \phi}
\end{array}
$$
is $C^1$-differentiable.
\end{maintheorem}

We observe that a pointwise differentiability version of the previous result holds in the more general context
where the potential $\phi$ belongs to $C^{r}(M,\mathbb R)$ (see~Subsection~\ref{sec:Banach-diff} for the definition
of pointwise differentiability and the proof of Proposition~\ref{propdif}). We will not use this fact here.

In general we can only expect pointwise continuity of the transfer operators acting on the space of $C^r$-observables. More precisely, given a fixed observable $g\in C^r(M,\mathbb R)$ the map $f\mapsto \cL_{f,\phi}(g)$
is continuous. However,  we refer the reader to Section~\ref{s.examples} for an explicit example where the
transfer operator in not even pointwise continuous when acting on the space of H\"older continuous observables.

The situation is rather different when we consider the dependence on the potential. For that reason, for the time being
let us focus on the regularity of the transfer operators as
$\mathcal{L}_{f,\phi}: C^{r+\alpha}(M, \R) \to C^{r+\alpha}(M, \R)$ on the potential $\phi$ and deduce
the analiticity of spectral radius, leading eigenfunction and eigenmeasure, and the equilibrium state when the dynamics $f$ is fixed. The precise definition of analyticity of functions acting on Banach spaces is postponed to Subsection~\ref{sec:analytic}.

\begin{maintheorem}\label{thm:B}
Assume $r+\al>0$. Let $f:M \to M$ be a local homeomorphism with $C^{r+\al}$ inverse branches satisfying
(H1) and (H2)  and let $\mathcal{W}^{r+\al}\subset C^{r+\alpha}(M, \R)$ be an open subset of H\"older continuous potentials $\phi:M \to \R$ such that either (P) holds (in the case $r=0$) or (P') holds (in the case $r>0$) with uniform constants.
Then the following functions are analytic:
\begin{enumerate}
\item[(i)] The Ruelle-Perron-Frobenius operator $C^{r+\al}(M, \mathbb R) \!\ni \!\phi \!\mapsto \!\mathcal{L}_{\phi} \!\in \!
L(C^{r+\alpha}(M, \R))$;
\item[(ii)] The spectral radius function $\mathcal{W}^{r+\al} \ni  \phi \mapsto \lambda_{ \phi}=\exp (\Ptop(f,\phi))$;
\item[(iii)] The invariant density function $\mathcal{W}^{r+\al} \ni \phi \mapsto h_{ \phi} \in C^{r+\alpha}(M, \R)$;
\item[(iv)] The conformal measure function $\mathcal{W}^{r+\al} \ni \phi \mapsto \nu_{f,\phi} \in (C^{r+\alpha})^*$. In particular,  for any fixed $g\in C^{r+\al}(M,\R)$ the map $\phi \mapsto \int g\;  d\nu_{f,\phi}$ is analytic;
\item[(v)] The equilibrium state function $\mathcal{W}^{r+\al} \ni \phi \mapsto \mu_{f,\phi}=h_\phi \nu_{f,\phi} \in (C^{r+\alpha})^*$. In particular,
for any fixed $g\in C^{r+\al}(M,\R)$ the map $\phi \mapsto \int g\;  d\mu_{f,\phi}$ is analytic.
\end{enumerate}
\end{maintheorem}

The previous result has some precursors, among which we mention the complex analyticity of the pressure function
for expanding maps (see e.g.~\cite{PP90}).  In fact, such differentiability results hold by standard operator
perturbation theory when the transfer operator has a spectral gap and it varies smoothly with respect to
the potential $\phi$. In addition, we obtain precise formula for the first derivative of the topological pressure in the previous theorem in Proposition~\ref{raio}.
Since the topological pressure is given by the logarithm of the spectral radius obtain
that  for all $ H\in C^{r+\alpha}(M, \R)$
$$
D_{\phi}P_{top}(f, \phi)_{ \;|\phi_{0}} \cdot H
	= \int h_{f, \phi_{0}} \cdot H \; d\nu_{f, \phi_{0}}
	= \int H \; d\mu_{f, \phi_{0}}.
$$

Now we focus on the regularity of the spectral objects associated to the transfer operator when one perturbs
the dynamical system. The following theorem asserts that both the topological pressure
and the maximal entropy measure are differentiable in a strong way, as functionals. More precisely,

\begin{maintheorem}\label{thm:C}
Assume $r\ge 2$. Let $\phi$ be a fixed $C^{r}$ potential on $M$ satisfying (P')
and $\cF^{r} $ be as above.
The following properties hold:
\begin{itemize}
\item[(i)] The pressure function $\Ptop(\cdot, \phi): \cF^{2} \to \re$ given by $f\mapsto \Ptop(f, \phi)$ is differentiable;
\item[(ii)] If $\phi\equiv 0$ then the maximal entropy measure function $\cF^{2} \ni f \mapsto \mu_{f} \in (C^{2}(M,\mathbb R))^{\ast}$ is differentiable. In particular, the map $ \cF^{2} \ni f \mapsto \int g\;  d\mu_{f}$  is differentiable for any fixed
$g\in C^{2}(M,\R)$.
\item[(iii)] If $\phi\equiv 0$ and $\mathcal{F}^{2} \ni f \mapsto g_{f} \in C^{2}(M , \R)$ is differentiable at ${f_{0}}$ then the map $\mathcal{F}^{2} \ni f \mapsto \int g_{f} d\mu_{f}$ is differentiable at $f_{0}$.
\end{itemize}
\end{maintheorem}

Our proof of the differentiability of the topological pressure with respect to the dynamics involves the analysis
of the iterations of the transfer operator at the constant function one. For that reason it was also necessary
to obtain precise formulas for the first derivatives
of the expressions above. Given $g\in C^1(M,\mathbb R)$ fixed, we obtain an expression for the first derivative
of $f\mapsto \cL_{f,\phi}(g)$, prove the chain rule
\begin{equation*}
D_{f}\mathcal{L}_{f,\phi}^{n}(g)_{|f_{0}} \cdot H
	= \sum_{i = 1}^{n}\mathcal{L}_{f_{0},\phi}^{i - 1} (D_{f}\mathcal{L}_{f,\phi}(\mathcal{L}_{f_{0},\phi}^{ n - i}(g))_{|f_{0}} \cdot H),
\end{equation*}
even without the differentiability of the transfer operator in the strong norm topology, and deduce the
expression for the derivative of the functional $\mu_f \colon \mathcal{F}^{2} \ni f \mapsto \int g \,d\mu_{f}$
given by
\begin{equation*}
D_{f}\mu_{f}(g)_{|f_{0}} \cdot H
	= \sum_{i = 0}^{\infty}\int D_{f}\tilde{\mathcal{L}}_{f}(\tilde{\mathcal{L}}_{f_{0}}^{i}(P_{0}(g)))_{|f_{0}} 	 
		\cdot H d\mu_{f_{0}}.
\end{equation*}
We omit the potential $\phi\equiv 0$ above for notational simplicity. Furthermore, since partial derivatives are continuous then the function $(f,\phi)\mapsto\Ptop(f,\phi)$ is differentiable (cf. Subsection~\ref{sec:Banach-manifold}). We refer the reader to Proposition~\ref{propdif}, Corollary~\ref{cor:jointdiff} and Theorem~\ref{thm:diff.max} for more details.

To finish this section one should comment on higher order differentiability results. In fact, on the one hand
using operator perturbation theory methods as in \cite{GL06} it seems most likely that one can actually prove higher
order differentiability of the spectral components. E.g. if $f \in C^{r+ \alpha}(M,M)$ then the conformal
measure map
$\mathcal{F}^{r+\alpha} \times \mathcal{W}^r \ni (f,\phi)  \mapsto \nu_{f,\phi} \in \mathcal({C}^{r +\alpha- 1})^{\ast}$
is $C^{r+ \alpha- 1}$-differentiable. The novelty of our approach and method used here is that it has the
advantage of providing very useful asymptotic formulas for the derivatives of the topological pressure and
equilibrium states. A priori it is not
clear how these can be obtained by means of the classical operator perturbation theory. In fact, the
classical operator perturbation theory requires the family of operators $\cL_{f,\phi}$ acting on the same Banach space
to vary continuously with $(f,\phi)$, something that may not occur even for expanding maps
(cf. Subsection~\ref{ex:discontinuous}). 
We give a wide range of applications in the following section.

\vspace{.3cm}
\subsubsection{Applications: Stability and differentiability in dynamical systems}

In this subsection we derive some interesting consequences on the stability of the robust class of non-uniformly
expanding maps considered.
The following is a consequence of Theorem~\ref{thm:B} and Theorem~\ref{thm:C}.

\begin{maincorollary}\label{cor:diff}
Given $r\ge 2$,  let $\cF^{r} $
 and  $ \mathcal W^{r}$
 be open sets of local diffeomorphisms and potentials as above.
If $f \mapsto \phi_f \in \mathcal W^{ 2}$ is  differentiable then the pressure function
$\mathcal F^{2} \ni f \mapsto \Ptop(f,\phi_f)$ is differentiable.
In particular, if $\delta>0$ is small then the pressure functions $\cF^{3} \times (-\delta,\delta)
\ni (f , t) \mapsto \Ptop(f,-t \log\|Df^{\pm1}\|)$ are differentiable.
\end{maincorollary}

As a byproduct of Theorems~\ref{thm:B} and \ref{thm:C}, we also obtain the regularity of the measure theoretical entropy, extremal Lyapunov exponents and sum of the positive Lyapunov exponents associated to the equilibrium states.

\begin{maincorollary}\label{cor:diff2}
Assume that $r\ge 1$ and $\al>0$.
Then
 $$
\cF^{r+\al} \times \mathcal W^{1+\al} \ni (f,\phi) \mapsto h_{\mu_{f,\phi}}(f) = \Ptop(f,\phi)-\int \phi \, d\mu_{f,\phi}
 $$
and  the Lyapunov exponent functions
 $$
\cF^{r+\al} \ni f \mapsto \int \log \|Df(x)\| \, d\mu_{f,\phi}
 	\quad\text{and}\quad
\cF^{r+\al} \ni f \mapsto \int \log \|Df(x)^{-1}\|^{-1} \, d\mu_{f,\phi} 	 
 $$
and
$$
\cF^{r+\al} \ni f  \mapsto \int \log |\det Df(x)| \, d\mu_{f,\phi} 	
$$
 are continuous. Furthermore, if $\phi \equiv 0$ and $r \geq 3$ and $\alpha\ge 0$ then the previous functions vary differentially with respect to the dynamics $f$.
\end{maincorollary}

Other application of our results include a strong stability of the statistical laws.  In \cite{CV13} we deduced that
this class of maps has exponential decay of correlations, which is well known to imply a Central Limit Theorem.
To prove the stability of this limit theorem our first step is to prove that time-$n$ correlation function with respect to the maximal entropy measure is differentiable with respect to $f$ and its derivative converges  to zero in the $C^1$-topology. More precisely,

\begin{maincorollary}\label{cor:smooth.correlation}
Given $\cF^{2}$ an open set of local diffeomorphisms and $\cW^{2}$ an open set of potentials as above, consider the correlation function
$$
C_{\varphi , \psi}(f, \phi, n) = \int(\varphi\circ f^{n})\psi \;d\mu_{f , \phi} - \int\varphi \;d\mu_{f , \phi}\int\psi \;d\mu_{f , \phi}
$$
defined for $f\in \cF^{2}$,  $\phi\in \cW^{2}$, observables $\varphi,\psi\in C^\al(M,\mathbb R)$ and
$n\in \mathbb N$.
Then
\begin{itemize}
\item[i)] The map $(f,\phi) \mapsto C_{\varphi, \psi} (f , \phi , n)$  is analytic in $\phi$ and continuous in $f$, and
\item[ii)] The map $f \mapsto C_{\varphi, \psi} (f , 0 , n)$ is differentiable, $\partial_f C_{\varphi, \psi}(f , 0, n)$ is convergent to zero as $n\to\infty$, and the convergence can be taken uniform in a small neighborhood of $f$.
\end{itemize}
 \end{maincorollary}

One should mention that property (i) above holds more generally, namely when we consider
$f\in \cF^{1+\al}$ and $\phi\in \cW^{1+\al}$. The regularity of the correlation function also allow us to establish the regularity of the quantities involved in the central limit theorem with respect to the dynamics and potential. More precisely,

\begin{maintheorem}\label{thm:CLT}
Let $\phi \in \cW^{2}$ and $f\in \cF^{2}$ be given. If $\psi\in C^{\al}(M,\mathbb R)$ then:
\begin{enumerate}
\item[i.] either $\psi = u\circ f - u+\int \psi \,d\mu_{f,\phi}$ for some $u \in L^{2}(M , \mathcal{F} , \mu_{f,\phi})$
	(we say $\psi$ is cohomologous to a constant)
\item[ii.] or the convergence in distribution
	$$
	\frac1{\sqrt{n}} \sum_{j = 0}^{n - 1} \psi \circ f^{j}
		\xrightarrow[n \to +\infty]{\mathcal D}  \mathcal{N}(m , \sigma^2)
	$$
	holds with mean $m=m_{f,\phi}(\psi)=\int\psi \;d\mu_{f, \phi}$ and variance $\sigma^2$ given by
	$$
	\sigma^2=\sigma_{f,\phi}^{2}(\psi)
		= \int \tilde \psi^{2}\;d\mu_{f,\phi} + 2\sum_{n = 1}^{\infty}\int \tilde \psi(\tilde \psi\circ f^{n}) \;d\mu_{f,\phi}>0,
	$$
	where $\tilde \psi=\psi-\int\psi \;d\mu_{f, \phi}$ is a mean zero function depending on $(f,\phi)$.
\end{enumerate}
Moreover, both functions $(f,\phi,\psi) \mapsto m_{f,\phi,\psi}$ and $(f,\phi,\psi) \mapsto \sigma^2_{f,\phi}(\psi)$
are analytic on $\phi, \psi$ and continuous on $f$.
Finally, if $\psi \in C^{2}(M,\mathbb R)$ and $\phi\equiv 0$ then $(f,\psi) \mapsto m_{f}(\psi)$ and
$(f , \psi) \mapsto \sigma^2_f(\psi)$ are differentiable.
\end{maintheorem}

Let us make some comments on the last result. Using the continuity of the variance $\sigma^2_{f,\phi}(\psi)$ with
respect to the dynamics $f$, potential $\phi$ and observable $\psi$, and since the first case in the theorem above
corresponds to the case that $\sigma^2_{f,\phi}(\psi)=0$ then we obtain the following consequences
for the cohomological equation.

\begin{maincorollary}\label{cor:cohomo}
Let $\cF^{2}$ and $\cW^2$ be as above. Then, if $\psi$ is not cohomologous
to a constant for $(f,\phi)$ then the same property holds for all close $(\tilde f,\tilde\phi)$. In consequence, the sets
$
\{
	(f,\phi)\in \cF^{2} \times \cW^{2} : \psi \text{ is cohomologous to } \int \psi \;d\mu_{f,\phi}
\}
$
and
$
\{
	\psi\in C^{2}(M,\mathbb R) :  \psi \text{ is cohomologous to } \int \psi \;d\mu_{f,\phi}
\}
$
are closed.
\end{maincorollary}

Therefore, a particularly interesting open question is to understand if the sets defined above have empty interior,
meaning that open and densely on the dynamical system and the potential any H\"older continuous observable
would not be cohomologous to a constant.

\vspace{.3cm}
Other consequence of the differentiability of the pressure function is related to a local large deviations
principle. First we recall some notions. Given an observable $\psi: M\to \mathbb R$ and $t\in \R$ the  \emph{free energy} $\cE_{f,\phi,\psi} $ is given by
\begin{equation*}
\cE_{f,\phi,\psi}(t)
	= \limsup_{n\to\infty} \frac1n \log \int e^{t S_n  \psi} \,d\mu_{f,\phi},
\end{equation*}
where $S_n  \psi=\sum_{j=0}^{n-1} \psi\circ f^j$ is the usual Birkhoff sum. In our setting we will prove
that the limit above does exist for all H\"older continuous $\psi$ and $|t|\le t_{\phi,\psi}$,
for some small $t_{\phi,\psi}>0$. Moreover we study its regularity in the parameters $t,\phi,\psi$ and $f$.

\begin{maintheorem}\label{thm:FreeEnergy}
Let $\al>0$, $f\in \mathcal{F}^{1+\alpha}$ and $\phi\in C^{\al}(M,\mathbb R)$ satisfy (H1), (H2) and (P). Then for any H\"older continuous observable $\psi:M\to \R$ there exists $t_{\phi,\psi}>0$ such that
for all $|t|\le t_{\phi,\psi}$ the following limit exists
$$
\cE_{f,\phi, \psi}(t)
	:=\lim_{n\to\infty} \frac1n \log \int e^{t S_n \psi} \; d\mu_{f,\phi}
	= \Ptop(f, \phi +t\psi) -\Ptop(f, \phi).
$$
In consequence, $\mathcal E_{f,\phi}(t)$ is analytic in $t,\phi$ and $\psi$.
Moreover, if $\psi$ is cohomologous to a constant then $t \mapsto \cE_{f,\phi,\psi}(t)$ is affine and, otherwise,
$t \mapsto \cE_{f,\phi,\psi}(t)$ is real analytic and strictly convex in $[-t_{\phi,\psi}, t_{\phi,\psi}]$.
Furthermore, if $\psi\in C^{2}(M,\mathbb R)$ then for every fixed $|t|\le t_{\phi,\psi}$ the function $\mathcal{F}^{2} \ni f\mapsto \cE_{f,\phi,\psi}'(t)$ is continuous and if $\phi\in \mathcal{W}^{2}$ we have that $\mathcal{F}^{2} \ni f\mapsto \cE_{f,\phi,\psi}(t)$
is differentiable.
\end{maintheorem}

So, if $\psi$ is not cohomologous to a constant then the function  $[-t_{\phi,\psi},t_{\phi,\psi}] \ni t\to \cE_{f,\phi,\psi}(t)$
is strictly convex  it is well defined the ``local"   Legendre transform $I_{f,\phi,\psi}$ given by
\begin{equation*}
I_{f,\phi,\psi}(s)
	= \sup_{-t_{\phi,\psi}\le t \le t_{\phi,\psi}} \; \big\{ s\, t-\cE_{f,\phi,\psi}(t) \big\}.
\end{equation*}
Let us mention that local rate functions have also been used in \cite{RY08} and let us refer the reader to Section~\ref{sec:deviations} for more details.

In fact,  using differentiability of the pressure function we obtain a level-1 large deviation
principle and deduce that stability of the rate function with the dynamical system. More precisely,

\begin{maintheorem}\label{thm:differentiability.LDP}
Let $V$ be a compact metric space and $(f_v)_{v\in V}$ be a parametrized and injective family of maps in $\mathcal{F}^{2}$
and let $\phi\in C^{2}(M,\mathbb R)$ be a potential so that (P') holds. If the observable $\psi\in C^{2}(M,\mathbb R)$ is not cohomologous to a constant then there exists an interval
$J\subset \mathbb R$ such that: for all $v\in V$ and $[a,b]\subset J$
$$
\limsup_{n\to\infty} \frac1n \log \mu_{f_v,\phi}
	\left(x\in M : \frac1n S_n\psi(x) \in [a,b] \right)
	\le-\inf_{s\in[a,b]} I_{f_v,\phi,\psi}(s)
$$
and
$$
\liminf_{n\to\infty} \frac1n \log \mu_{f_v,\phi}
	\left(x\in M : \frac1n S_n\psi(x) \in (a,b) \right)
	\ge-\inf_{s\in(a,b)} I_{f_v,\phi_v,\psi}(s)
$$
If in addition $\psi\in C^{2}(M,\mathbb R)$ then the rate function $(s,v) \mapsto I_{f_v,\phi,\psi}(s)$
is continuous on $J\times V$ in the $C^0$-topology.
\end{maintheorem}

\begin{figure}[htb]
\includegraphics[width=9cm]{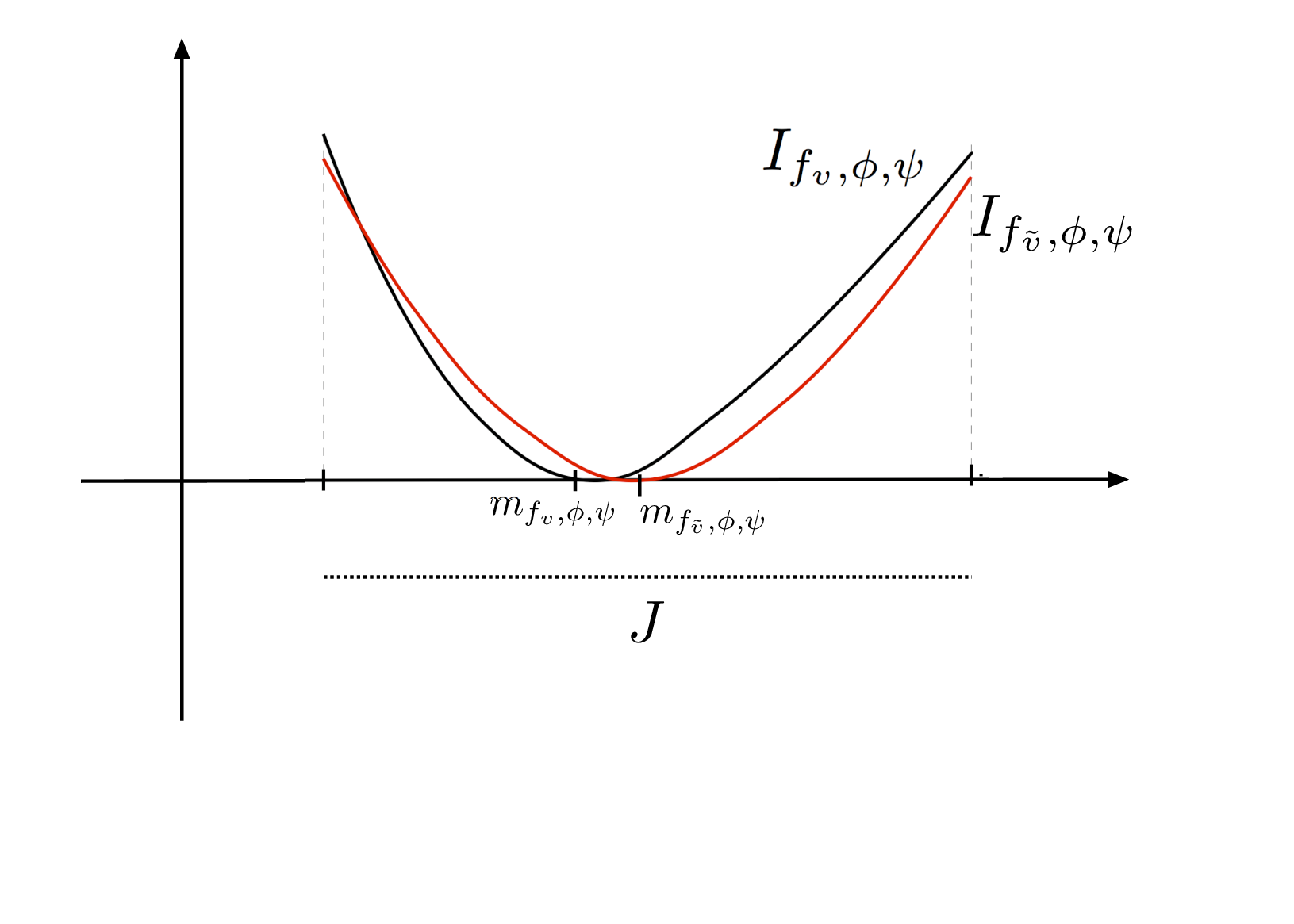}
\caption{Continuity of the rate functions}
\end{figure}

Let us mention that some upper and lower large deviation bounds for a larger class of
transformations and the class of $C^0$-observables were obtained previously in \cite{AP06,Va12}.
The previous provides sharper results for H\"older continuous observables.

\section{Preliminaries}\label{s.preliminaries}

In this section we provide some preparatory results needed for the proof of the
main results. Namely, we recall some properties of the transfer operators.

\subsection{Analytic functions on Banach spaces}\label{sec:analytic}
In what follows we recall the notion of analyticity for functions on Banach spaces. Let $E , F$ be Banach spaces and denote by $\mathcal{L}^{i}_{s}(E , F)$ the space of symmetric $i$-linear transformations from
$E^{i}$ to $F$. For notational simplicity, given $P_{i} \in \mathcal{L}^{i}_{s}(E , F)$ and $h\in E$ we set
$P_{i}(h) := P_{i}(h,\ldots,h)$.

\begin{definition}
Let $E , F$ be Banach spaces and $U\subset E$ an open subset. We say the function $f : U \subset E \rightarrow F$
is \emph{analytic} if for all $x \in U$ there exists $r > 0$ and for every $i\ge 1$ there exists $P_{i} \in \mathcal{L}^{i}_{s}(E , F)$ (depending on $x$) such that
$$
f(x + h) = f(x) + \sum_{i=1}^{\infty}\frac{P_{i}(h)}{i!}
$$ for all $h \in B(0 , r)$ and the convergence is uniform.
\end{definition}

Analytic functions on Banach spaces have completely similar properties to real analytic and complex analytic functions.
For instance, if $f : U \subset E \rightarrow F$ is analytic then $f$ is $C^{\infty}$ and for every $x\in U$ one has
$P_{i}=D^{i}f(x)$. For more details see for example \cite[Chapter~12]{Cha85}. In our setting we will be mostly
interested in considering the Banach spaces $\mathbb R$, $\mathbb C$, $C^{r+\al}(M,\mathbb R)$ or
$\cL(C^{r+\al}(M,\mathbb R),C^{r+\al}(M,\mathbb R))$. In any case, to prove analyticity of a function in an open subset of
its domain it is enough to write it in every point as a power series whose the terms are given by
symmetric $i$-linear transformations.

\subsection{Differentiability of operators on Banach spaces}\label{sec:Banach-diff}

Here we discuss several ways of differentiating operators and functionals acting on Banach spaces.
Given Banach spaces $E,F$,  we will be mostly interested to analyze the differentiability of a family:
\begin{itemize}
\item[(i)] 
$
(P_q)_{q\in U}
$
of operators in $L(E,F)$ parametrized on some open subset $U$ of a vector space;
\item[(ii)] $
(\mu_q)_{q\in U}
$
of linear functionals $L(E,\mathbb K)$, $\mathbb K=\mathbb R$ or $\mathbb C$,
parametrized on some open subset $U$ of a vector space.
\end{itemize}
Although (ii) is a special case of (i) we state in this way to highlight that
item (i) will includes e.g. transfer operators acting on different Banach spaces 
and item (ii) includes e.g. the topological pressure function and probability 
measure functionals.
There are two distinct possibilities for differentiability that we recall. On the one hand, given the Banach spaces $E, F$,
the family of linear operators $U \ni q \mapsto P_q\in \cL(E, F)$ is {\em pointwise differentiable}
if for any fixed $\varphi \in E$ the map
$$
\begin{array}{ccc}
U & \to & F \\
 q  & \mapsto &  P_q(\varphi) 
\end{array}
$$
is differentiable. On the other hand, the family of linear operators $U \ni q \mapsto P_q\in \cL(E, F)$ is 
{\em norm operator differentiable}, or just {\em differentiable} for short, if the map
$$
\begin{array}{ccc}
U & \to & \cL(E,F) \\
 q  & \mapsto &  P_q (\cdot)
\end{array}
$$
is differentiable. This is clearly stronger than pointwise differentiability.
Subsection~\ref{sec:Banach-manifold} below will allow to make the same considerations 
on these different notions of differentiability for maps acting on Banach manifolds, which
include the space of local diffeomorphisms on a compact manifold.
Most results on linear response formula deal with the pointwise differentiability
of Gibbs equilibrium states whereas the parameter is either a smooth family of potentials
or a  smooth family of dynamical systems.
In this paper, if not stated otherwise, we will always consider the (strong) differentiability 
of Gibbs measures $\mu_{\phi, f}$ as functionals and differentiability of transfer operators 
parametrized on open sets of potentials and dynamics.

\subsection{Banach manifolds and Fr\'echet differentiability}\label{sec:Banach-manifold}

In what follows we recall the structure of Banach manifold on the space $\text{Diff}_{\text{loc}}^r(M)$ of $C^r$-local diffeomorphisms on $M$, $r\ge 1$, and make precise the notion of differentiability used along this paper. The description of Banach manifold
on the space of $C^r$-diffeomorphisms given below we will follow closely \cite{Fr79}, and
refer the reader to \cite{Pa68,Fr79} for more details and proofs.
\smallskip

If the compact manifold $M$ is a submanifold of $\mathbb R^m$ it follows from \cite[Theorem~4.1]{Fr79} that
$C^r(M,M)$ is a submanifold of the Banach space $C^r(M,\mathbb R^m)$. Since the space $\text{Diff}_{\text{loc}}^r(M)$
of $C^r$-local diffeomorphisms is an open subset of $C^r(M,M)$, the later implies
that one can view $\text{Diff}_{\text{loc}}^r(M)$ as a Banach manifold also modeled by the Banach space
$C^r(M,\mathbb R^m)$.
In what follows we recall the atlas used for $C^r(M,M)$ since it helps to make precise its tangent space.

By Whitney's embedding theorem,
any smooth manifold can be embedded in $\mathbb R^{m}$ with $m=2\dim M$.
Then $C^r(M,M)$ is embedded on the Banach space $C^r(M,\mathbb R^m)$.
In fact, $C^r(M,M)$ inherits a structure of Banach manifold modeled by the Banach space as we now
describe.
Consider a tubular neighborhood $\pi: U \to M$ of $M$ in $\mathbb R^m$,
the tangent bundle $TM=\{(y,v) \in \mathbb R^{2m}: y\in M \,\text{and}\, v\in T_y M\} \subset M \times \mathbb R^m$,
the normal bundle $TM^\perp=\{(y,v) \in \mathbb R^{2m}: y\in M \,\text{and}\, <v,w>=0, \,\forall w \in T_y M\} \subset M \times \mathbb R^m$ and the projections $P, P^\perp: M \times \mathbb R^m \to \mathbb R^m$
so defined so that $P(x,v)$ is the orthogonal projection of $v$ on $T_x M$ and $P^\perp(x,v) = v - P(x,v) \in T_xM^\perp$.
Now choose a fixed $f \in C^r(M,M)$. Then one can decompose
$
E:= C^r(M, \mathbb R^m) = E^t \oplus E^n
$
with
$$
E^t =\{ H \in C^r(M, \mathbb R^m) : (f(x),H(x)) \in TM, \; \forall x\in M \}
$$
and
$$
E^n =\{ H \in C^r(M, \mathbb R^m) : (f(x),H(x)) \in TM^\perp, \; \forall x\in M \}.
$$
Franks~\cite{Fr79} proved that if $\omega_\pi : C^r(M, U) \to C^r(M, M)$ denote the projection $\omega_\pi(g) = \pi\circ g$
then
the map
$$
\begin{array}{rcl}
\alpha : B^t \times B^n & \to & C^r(M, \mathbb R^m)\\
	(h^t, h^n) & \mapsto & \omega_\pi(f + h^t) + h^n
\end{array}
$$
is a local diffeomorphism at $(0,0)$ and $\alpha^{-1}(C^r(M,M)) = B^t \times \{0\}$
(here $B^t \subset E^t$ and $B^n \subset E^n$ denote small balls so that
$f + h^t + h^n \in C^r(M, U)$). Up to diminish the balls $B^t, B^n$ if necessary, this guarantees that $\alpha$
is invertible. Therefore, if $V=\alpha(W)$ and $\varphi=(\alpha_W)^{-1}$ then $(V,\varphi)$ is a chart and
$\varphi( V \cap C^r(M, M)) = B^t \subset E^t$. This justifies that $C^r(M,M)$ is a submanifold of
$C^r(M, \mathbb R^m)$ and that for any $f \in C^{r}(M , M)$ the tangent space
$T_{f} C^{r}(M , M)$ is naturally identified with the space
\begin{equation}\label{eq:spaceG}
\Gamma^{r}_{f} := \{\gamma \in C^{r}(M , TM) : \gamma(x) \in T_{f(x)}M, \forall x \in M\}
\end{equation}
of $C^r$-sections (or vector fields) over $f$. The space $\Gamma^{r}_{f}$ is Banachable, that is,
since it is naturally isomorphic to
\begin{equation}\label{eq:spaceGB}
E^{t} = \{H \in C^{r}(M , \mathbb R^{m}) : (f(x),\gamma(x)) \in TM \},
\end{equation}
then it inherits a structure of Banach space.
One should mention the later identification is independent of the embedding of $M$ (cf. page 238 in \cite{Fr79}).
Throughout we will consider the space $\text{Diff}_{\text{loc}}^r(M)$
as a Banach manifold modeled by $C^r(M, \mathbb R^m)$,
from which $T_{f} \text{Diff}_{\text{loc}}^r(M) \simeq E^t \subset C^r(M, \mathbb R^m)$ for every ${f} \in \text{Diff}_{\text{loc}}^r(M)$.

\begin{definition}
Let $F$ be a Banach space. Given $f_0 \in \text{Diff}_{\text{loc}}^r(M)$, we say that a function $\Psi : \text{Diff}_{\text{loc}}^r(M) \subset C^r(M, \mathbb R^m) \to F$
is \emph{Fr\'echet differentiable at $f_0$ } if there exists a continuous linear functional
$D\Psi(f_0) : T_{f_0} \text{Diff}_{\text{loc}}^r(M) \subset C^r(M, \mathbb R^m) \to  F$ so that
\begin{equation}\label{eq:deriv}
\lim_{H \to 0} \frac1{\|H\|_{C^r(M, \mathbb R^m)}} \; \| \Psi(f_0 + H) -\Psi (f_0) - D\Psi(f_0) H\|_F=0
\end{equation}
where $H$ is taken to converge to zero in $T_{f_0} \text{Diff}_{\text{loc}}^r(M) = T_{f_0} C^{r}(M , M)$. The functional $D\Psi(f_0)$ is
the derivative of $\Psi$ at $f_0$. We say that $\Psi$ is $C^1$-\emph{differentiable} if the Fr\'echet derivative $D\Psi(f)$ exists at every $f \in \text{Diff}_{\text{loc}}^r(M)$ and the map $\text{Diff}_{\text{loc}}^r(M) \ni f \mapsto D\Psi(f) \in B(C^r(M, \mathbb R^m), F)$
is continuous.
\end{definition}

For notational simplicity, when no confusion is possible we shall omit the
spaces $C^r(M, \mathbb R^m)$ and $F$ in the norms whenever using the expression~\eqref{eq:deriv} to compute derivatives of vector valued transformations.
The following results will be instrumental:

\begin{proposition}\cite{Fr79}
Let $M$ be a compact manifold. Then for any $r,s\ge 1$ the composition map
	$$
	\begin{array}{ccc}
	C^{r+s}(M,M) \times C^r(M,M) & \to & C^r(M,M)\\
	(f,g) & \mapsto & f\circ g
	\end{array}
	$$
	is $C^s$-differentiable.
\end{proposition}

Clearly the previous proposition endows $\text{Diff}_{\text{loc}}^r(M)$ with a structure of a Lie group. Finally, we shall
comment on the differentiability of functions on the product space $\text{Diff}_{\text{loc}}^r(M) \times C^{s}(M, \mathbb R)$,
for $s>0$. Since  $C^{s}(M, \mathbb R)$ is a Banach space then $\text{Diff}_{\text{loc}}^r(M) \times C^{s}(M, \mathbb R)$ is a Banach manifold modeled by the Banach space $C^{r}(M, \mathbb R^m) \times C^{s}(M, \mathbb R)$.
So, up to consider a chart, we are interested in the differentiability of a map on the product of Banach spaces.
We endow the product
$X\times Y$ of Banach spaces with the norm $\| \cdot \|_{X\times Y}=\|\cdot\|_X + \|\cdot\|_Y$.
We need the following simple result, whose proof can be found in \cite[Theorem~7.1]{La97}.

\begin{lemma}
Let $X,Y, Z$ be normed linear spaces, let $U \subset X  \times Y$ be an open set and consider
a function $\Psi: U \to Z$.  Given $(x,y)\in U$ set $U_x=\{ y\in Y : (x,y) \in U \}$ and $U_y=\{ x\in X : (x,y) \in U \}$
and assume the maps $\Psi_x : U_x \subset Y \to Z$ and $\Psi_y : U_y \subset X \to Z$ are $C^1$-differentiable. Then
$\Psi$ is $C^1$-differentiable in $U$.
\end{lemma}
For more results on the differentiability of functions on vector spaces or Banach manifolds we refer the reader to e.g. \cite{DM07, AMR07}.

\subsection{Spectral radius of Ruelle-Perron-Frobenius operators and conformal measures}\label{Conformal Measure}

Let $\cL_{f,\phi}: C^0(M,\mathbb R) \to C^0(M,\mathbb R)$ be the Ruelle-Perron-Frobenius transfer operator
associated to $f:M\to M$ and $\phi:M\to\real$ previously defined by
$$
\cL_{f,\phi} g(x) = \sum_{f(y)=x} e^{\phi(y)} g(y).
$$
for every $ g\in C^0(M,\mathbb R)$.  We consider also the dual operator $\cL^*_{f,\phi}:\cM(M)\to\cM(M)$ acting on
the space $\cM(M)$ of Borel measures in $M$ by
$$
\int g \, d(\cL_{f,\phi}^*\eta) = \int (\cL_{f,\phi} g) \, d\eta
$$
for every $g \in C^0(M,\mathbb R)$. Let $r(\mathcal L_{f,\phi})$ be the
spectral radius of $\cL_{f,\phi}$.  In our context  conformal measures associated to
the spectral radius always exist. More precisely,

\begin{proposition}\label{p.conformal.measure}
Assume that $f$ satisfies assumptions (H1), (H2). If  $\phi$ satisfies $\sup\phi-\inf\phi<\log\deg(f)-\log q$
 then there exists a conformal  measure $\nu=\nu_{f,\phi}$ such that $\cL_{f,\phi}^* \nu=\la \nu$,
where $\la=r(\cL_{f,\phi})$. Moreover, $\nu$ is a non-lacunary Gibbs measure and
$P_{\text{top}}(f,\phi)	= \log \la$.
\end{proposition}

\begin{proof}
See Theorem~B, Theorem~4.1 and Proposition~6.1 in \cite{VV10}.
\end{proof}

\subsection{Spectral gap for the transfer operator in $C^\alpha(M,\mathbb R)$}\label{s.s.cones}

Recall that the H\"older constant of $\vr\in C^\alpha(M,\mathbb R)$ is
the least constant $C>0$ such that $|\vr(x)-\vr(y)|\le C d(x,y)^\alpha$ for  all
points $x\ne y$. For any $\de>0$,  the local H\"older constant $|\vr|_{\alpha,\delta}$ is the corresponding
notion for points $x,y$ such that  $d(x,y)<\delta$.
If $\de$ is small then there exists a positive integer $m$ such that every $(C,\al)$- H\"older continuous map in
balls of radius $\de$ is globally $(Cm,\al)$-H\"older continuous (see \cite[Lemma~3.5]{CV13}).
This put us in a position to state the precise relation on the constants $L$, $\si$, $q$ and $\vep_\phi$
on the hypothesis (H1), (P) and (P'). We assume:
\begin{equation}\label{eq.vep}
e^{\vep_\phi}\cdot\left(\frac{(\deg(f)-q) \sigma^{-\alpha} + q L^\alpha [1+(L-1)^\alpha] }{\deg(f)} \right)
	+ \vep_\phi 2m L^\al \diam(M)^\al
	<1
\end{equation}
and
\begin{equation}\label{eq.vepp}
[1+\vep_\phi ] \cdot e^{\vep_\phi} \cdot \left(\frac{(\deg(f)-q) \sigma^{-\alpha} + q L^\alpha [1+(L-1)^\alpha] }{\deg(f)} \right)
	<1
\end{equation}
This choice was taken to obtain the following cone invariance.

\begin{theorem}\label{t.cone.invariance}
Assume that $f$ satisfies (H1), (H2) and that $\phi$ satisfies (P). Then there exists 
$0< \hat \la< 1$ such that $\cL_{f,\phi}(\Lambda_{\kappa,\delta}) \subset \Lambda_{\hat\la \kappa,\delta} $ for every large positive constant $\kappa$, where
$$
\Lambda_{\kappa,\delta}
	=\left\{\vr\in C^0(M,\mathbb R) : \vr>0 \text{ and } |\vr|_{\alpha,\delta}\leq \kappa \inf \vr \right\}.
$$
is a cone of locally H\"older continuous observables. Moreover, given $0< \hat\la< 1$,  the cone
${\Lambda}_{\hat\la \kappa,\delta}$ has finite ${\Lambda}_{\kappa,\delta}$-diameter in the projective metric
$\Theta_k$. Furthermore, if $\vr \in \La_{\kappa,\de}$ satisfies $\int \varphi \, d\nu_{f,\phi}=1$ and $h_{f,\phi}$
denotes the $\Theta_\kappa$-limit of $\vr_n= \tilde \cL^n_\phi(\vr)$ then, $\vr_n$ converges exponentially fast to $h_{f,\phi}$ in the H\"older norm.
\end{theorem}

\begin{proof}
See Theorem~4.1, Proposition~4.3 and Corollary~4.5 in \cite{CV13}.
\end{proof}

Hence the normalized operator $\tilde\cL_{f,\phi}=\la_{f,\phi}^{-1}\cL_{f,\phi}$ has the spectral gap property.

\begin{theorem}\label{t.gap}
There exists $0< r_0< 1$ such that the operator $\tilde \cL_{f,\phi}$ acting on the space $C^\al(M,\mathbb R)$
admits a decomposition of its  spectrum given by $\Sigma= \{1\} \cup \Sigma_0$, where $\Sigma_0$ contained in a
ball $B(0, r_0)$. Furthermore, there exists $C>0$ and $\tau\in(0,1)$ such that
$\| \tilde \cL^n_{f,\phi} \varphi -h_{f,\phi} \, \int \varphi\, d\nu_{f,\phi} \|_\alpha \leq C \tau^n \|\varphi\|_\alpha$ for all $n\ge 1$
and $\varphi\in C^\al(M,\mathbb R)$, where  $h_{f,\phi} \in C^\alpha(M,\mathbb R)$ is the unique fixed point for
$\tilde \cL_{f,\phi}$ such that $\int h_{f,\phi} \, d\nu_{f,\phi}=1$.
\end{theorem}

\begin{proof}
See Theorem~4.6, Proposition~4.3 and Corollary~4.5 in \cite{CV13}.
\end{proof}

As a consequence of the previous results it follows that the density of the equilibrium state with
respect to the corresponding  conformal measure vary continuously in the $C^0$-norm. We recall the
precise statement and the proof of the result since some estimates in the proof will be needed later on.

\begin{proposition}\label{prop:C0cont}
Let $\cF$ be a family of local homeomorphisms with inverse Lipschitz and $\cW$ be a family of H\"older potentials as above. 
Then the topological pressure $\cF\times \cW \ni (f,\phi) \mapsto \log \la_{f,\phi}= P_{\text{top}}(f,\phi)$ and  the
density function
\[
\begin{array}{ccc}
\mathcal F\times \cW & \to & (C^\al(M,\mathbb R),\|\cdot \|_0) \\
(f,\phi) & \mapsto & \frac{d\mu_{f,\phi}}{d\nu_{f,\phi}}
\end{array}
\]
are continuous. Moreover, $h_{f,\phi}=\lim \la_{f,\phi}^{-n}\, \cL^n_{f,\phi} 1$ and the convergence is uniform
in a neighborhood of $(f,\phi)$.
\end{proposition}

\begin{proof}
Recall that $P_{\text{top}}(f,\phi)=\log \la_{f,\phi}$ where $\la_{f,\phi}$
is the spectral radius of the operator $\cL_{f,\phi}$. Moreover, it follows from the proof of
Theorem~\ref{t.cone.invariance} that for any $\vr \in \La_{\kappa,\de}$ satisfying $\int \varphi \, d\nu_{f,\phi}=1$ one has
in particular
\begin{equation}\label{eq.C0contt}
\left\|\la^{-n}_{f,\phi}\cL_{f,\phi}^{n}\varphi- \frac{d\mu_{f,\phi}}{d\nu_{f,\phi}}\right\|_{0}
	\leq C \tau^n
\end{equation}
for all $n\ge 1$. Notice the previous reasoning applies to $\varphi\equiv 1\in\La_{\kappa,\de}$. Moreover, since
the spectral gap property estimates depend only on the constants $L,\si$ and $\deg(f)$ it follows that all transfer
operators $\cL_{\tilde f,\tilde \phi}$ preserve the cone $\La_{\kappa,\de}$ for all pairs $(\tilde f,\tilde\phi)$ and that
the constants $R_1$ and $\De$ can be taken uniform in a small neighborhood $\cU$ of $(f,\phi)$. Furthermore,
one has that $\int \la_{f,\phi}^{-1}\cL_{f,\phi}\; d\nu_{f,\phi}=1$ and so the convergence
$$
\lim_{n \to +\infty} \frac 1 n \log \|{\tilde {\cL^n}}_{\tilde f,\phi}(1) \|_0
	= \lim_{n \to +\infty} \frac 1 n \log\left\| {\la^{-n}_{\tilde f, \phi}}\,  {\cL^n}_{\tilde f , \phi}  (1) \right\|_0
	= 0
$$
given by Theorem~\ref{t.cone.invariance} can be taken uniform in $\cU$. This is the key
ingredient to obtain the continuity of the topological pressure and density function. Indeed, let $\vep>0$ be fixed
and take $n_0\in \mathbb N$ such that
$
\Big |\frac{1}{n_0} \log \|{\cL^{n_0}}_{\tilde f, \phi}(1)\|_0 - \log(\la_{\tilde f, \phi})\Big |
	< \frac\epsilon3.
$
for all $\tilde f  \in \cU$. Moreover,  using $P_{\text{top}}(f,\phi)=\log \la_{f,\phi}$ by triangular inequality we get
\begin{align*}
\Big | \Ptop(f,\phi) - \Ptop(\tilde f,\phi)\Big |
	& \leq \Big |\frac{1}{n_0} \log \|{\cL^{n_0}}_{\tilde f, \phi}(1)\|_0 - \log(\la_{\tilde f, \phi})\Big | \\
	&+ \Big |\frac{1}{n_0} \log \|{\cL^{n_0}}_{ f, \phi}(1)\|_0 - \log(\la_{ f, \phi})\Big | \\
	& + \Big |\frac{1}{n_0} \log \|{\cL^{n_0}}_{ f, \phi}(1)\|_0 - \frac{1}{n_0} \log \|{\cL^{n_0}}_{ \tilde f, \phi}(1)\|_0 \Big |.
\end{align*}
Now, it is not hard to check that, for $n_0$ fixed, the function $\cU \to C^0(M,\mathbb R)$
$$
\tilde f \mapsto \cL^{n_0}_{\tilde f,\phi} 1
	=\sum_{\tilde f^{n_0}(y)=x} e^{S_{n_0}\phi(y)}
$$
is continuous. Consequently, there exists a neighborhood $\mathcal V\subset \cU$ of $f$ such that
$|\frac{1}{n_0} \log \|{\cL^{n_0}}_{ f, \phi}(1)\|_0 - \frac{1}{n_0} \log \|{\cL^{n_0}}_{ \tilde f, \phi}(1)\|_0|<\vep/3$
for every $\tilde f\in \mathcal V$. Altogether this proves that $| \Ptop(f,\phi) - \Ptop(\tilde f,\phi)\Big |<\vep$
for all $\tilde f \in \mathcal V$. Since $\vep$ was chosen arbitrary we obtain that both the leading eigenvalue
and topological pressure functions vary continuously with the dynamics $f$.
Finally, by equation~\eqref{eq.C0contt} above applied to $\varphi\equiv 1$ and triangular inequality
we obtain that
\begin{align*}
\left\|\frac{d\mu_{\tilde f,\phi}}{d\nu_{\tilde f,\phi}} -  \frac{d\mu_{f,\phi}}{d\nu_{f,\phi}} \right\|_{0}
	& \leq 2C \tau^n
	+ \left\|\la^{-n}_{\tilde f,\phi}\cL_{\tilde f,\phi}^{n}1- \la^{-n}_{f,\phi}\cL_{f,\phi}^{n}1  \right\|_{0}
\end{align*}
for all $n$. Hence, proceeding as before one can make the right hand side above as close to zero as
desired provided that $\tilde f$ is sufficiently close to $f$. This proves the continuity of the density function and
finishes the proof of the proposition.
\end{proof}

\subsection{Spectral gap for the transfer operator in $C^r(M,\mathbb R)$}\label{s.s.conesCr}

Here we recall the analogous results for the action of the transfer operator in the space
of smooth observables. In particular we have the corresponding spectral gap property for the action of
Ruelle-Perron-Frobenius transfer operators in the space of smooth observables whose proof
can be found in \cite[Section~5]{CV13}.

\begin{theorem}\label{t.gapCr}
There exists $0< r_0< 1$ such that the operator $\tilde \cL_{f,\phi}$ acting on the space $C^r (M,\mathbb R)$ ($r\ge 1$)
admits a decomposition of its  spectrum given by $\Sigma= \{1\} \cup \Sigma_0$, where $\Sigma_0$ contained in a
ball $B(0, r_0)$. In consequence, there exists $C>0$ and $\tau\in(0,1)$ such that
$\| \tilde \cL^n_{f,\phi} \varphi -h_{f,\phi} \, \int \varphi\, d\nu_{f,\phi} \|_r \leq C \tau^n \|\varphi\|_r$ for all $n\ge 1$
and $\varphi\in C^r(M,\mathbb R)$, where  $h_{f,\phi} \in C^r(M,\mathbb R)$ is the unique fixed point for
$\tilde \cL_{f,\phi}$.
\end{theorem}

Let us mention also that by Proposition~5.4 in \cite{CV13} one has that
\begin{equation}\label{eq:uniformCr}
\Ptop(f,\phi)
	=\lim_{n\to\infty} \frac1n \log \|\cL_{f,\phi}^n1\|_r
\end{equation}
and that the limit can be taken uniform in a $C^r$ neighborhood of $(f,\phi)$. This is the key fact that
will be used later on to prove the differentiability of the topological pressure.

\section{Differentiability results}\label{sec:diff}

In this section we address the regularity of the Perron-Frobenius operator, spectral radius and
corresponding eigenmeasure and eigenfunction. For simplicity we address first the dependence on the potential
and later on the dynamical system.
In particular, Theorems~\ref{thm:general.differentiability} and ~\ref{thm:C} will be proven on Subsection
~\ref{subsec:difdynamics1} and ~\ref{subsec:difdynamics2} while Theorem~\ref{thm:B} is proved along Subsection~\ref{subsec:difpotential} below.

\subsection{Differentiation with respect to the potential}\label{subsec:difpotential}

First we fix $f$ and will focus on the differentiability questions with respect to the potential $\phi$.
Let $\mathcal{W} $ be an open set of potentials in $C^{\alpha}(M, \R)$, $\alpha > 0$, satisfying
the condition  (P),  endowed with the $C^{\alpha}$- topology.
For notational simplicity, when
no confusion is possible, we write simply $\cL_{f,\phi}$, $ \la_{f,\phi}$ and $h_\phi$ omitting the dependence on $f$.

\begin{proposition}\label{prop:derivative.L.phi}
Assume that $r+\al>0$. The map $C^{r+\al}(M,\mathbb R) \ni \phi \mapsto \mathcal{L}_{f,\phi}^{n} \in \cL(C^{r+\al}(M,\mathbb R), C^{r+\al}(M,\mathbb R))$
is analytic, hence $C^\infty$. Moreover, for every
vectors $g , H \in C^{{r+\al}}$ and for
every $n\ge 1$, the first derivative acting in $H$ is given by
\begin{equation}\label{eq:derivative.L.phi}
(D_{\phi}\mathcal{L}_{f,\phi}^{n}(g))_{|\phi_{0}} ( H)
	= \sum_{i = 1}^{n}\mathcal{L}_{f,\phi_{0}}^{i}
	  (H\cdot \mathcal{L}_{f,\phi_{0}}^{n - i}(g)).
\end{equation}
\end{proposition}

\begin{proof}
Note that
$$
\mathcal{L}_{f, \phi + H}(g)
	= \mathcal{L}_{f,\phi}(e^{H}g)
	= \sum_{i = 0}^{\infty}\mathcal{L}_{f,\phi} \left(\frac{1}{i!}H^{i}g\right)
	= \mathcal{L}_{f,\phi}(g) +
		 \sum_{i = 1}^{\infty}\frac{1}{i!} \mathcal{L}_{f,\phi}\left(H^{i}g\right).
$$
Since $\cL_{f,\phi}$ is a bounded linear operator acting on $C^{r+\al}(M,\mathbb R)$ then
there exists a constant $K>0$ so that $\|\mathcal{L}_{f,\phi}\left(H^{i}g\right) \|_{r+\al}
\le K \|g\|_{r+\al} \, \| H \|^i_{r+\al}$ for all $i\ge 1$. In particular
$$
\| \sum_{i = 1}^{\infty}\frac{1}{i!} \mathcal{L}_{f,\phi}\left(H^{i}g\right) \|_{r+\alpha}
	\le K \|g\|_{r+\alpha}  \sum_{i = 1}^{\infty}\frac{1}{i!} \|H \|^i_{r+\alpha}
	\le K \|g\|_{r+\alpha} e^{\|H \|_{r+\alpha}}< \infty,
$$
which implies that the sum in right hand side above is convergent.
Let us denote by $\cL^i_s( C^{r+\al}(M,\mathbb R), C^{r+\al}(M,\mathbb R) )$ the space of symmetric $i$-linear maps with domain
in $[C^{r+\al}(M,\mathbb R)]^i$ into $C^{r+\al}(M,\mathbb R)$.
Note also that the maps
$$
C^{r+\al}(M,\mathbb R) \ni \phi
	\mapsto \Big( H \mapsto \mathcal{L}_{f,\phi}(H^i \cdot ) \Big) \in \cL^i_s( C^{r+\al}(M,\mathbb R))
$$
are continuous for every $i \in \natural$, and that the product between functions is also continuous in
$C^{r+\al}(M,\mathbb R)$.
Then for $k \in \natural$ observe that, using once more the continuity of the transfer operator
$\cL_{f,\phi}$ acting on the space $C^{r+\al}(M,\mathbb R)$, we get
\begin{align*}
 \sup_{\|g\|_{r+\al}= 1} \!\!\! & \frac{\| \mathcal{L}_{f, \phi + H}(g) - \mathcal{L}_{f,\phi}(g) -  \sum_{i = 1}^{k}\frac{1}{i!} \mathcal{L}_{f,\phi}\left(H^{i}g\right) \|_{{r+\al}}}{ \|H\|^k_{{r+\al}}} \\
	& \quad\quad\qquad\!\leq\! \sum_{i = k+ 1}^{\infty} \sup_{\|g\|_{r+\al}= 1} \!\!\frac{\|  \frac{1}{i!} \mathcal{L}_{f,\phi}	 
	\left(H^{i}g\right) \|_{{r+\al}}}{\|H\|^k_{{r+\al}}} \\
	& 
	\quad\quad\qquad\!\leq\! K \sum_{i = k+ 1}^{\infty}  \frac{1}{i!}
	\| H \|^{i-k}_{{r+\al}}
	\!\leq\! K \sum_{i = k+ 1}^{\infty}  \frac{1}{(i-k-1)!}
	\| H \|^{i-k}_{{r+\al}}
	  \\
	& 
	\quad\quad\qquad\! = \! K \|H\|_{{r+\al}} \; e^{\| H \|^{i-k}_{{r+\al}}}
\end{align*}
which converges to zero as $H$ tends to zero (we used the notation $0!=1$). By Theorem 1.4 in \cite{Fr79}, this implies that $\phi \mapsto \cL_{f,\phi}$ is $C^k$, for any $k \in \natural$, and its $k$-th derivative applied in $H$ is $\mathcal{L}_{f,\phi}\left(H^{i} \cdot \right)$.
Note that this also implies that $\phi \mapsto \cL_{f,\phi}$ is analytic.
By applying the chain rule to the composition $\phi \mapsto \cL^n_{f,\phi}(g)$ we  finish the proof of the proposition.
\end{proof}

\begin{remark}\label{complex}
We observe that the same argument as used in the proof of the previous theorem guarantees
the analiticity of the Ruelle-Perron-Frobenius operator operator acting on the space of complex observables.
More precisely, given a potential $\phi \in C^{r+\al}(M, \mathbb C)$ then the transfer operator
$\mathcal{L}_{f,\phi} : C^{r+\al}(M, \mathbb C) \to C^{r+\al}(M, \mathbb C)$ is analytic and \eqref{eq:derivative.L.phi}
holds for every $H \in C^{r+\al}(M, \mathbb C)$.
\end{remark}

\begin{remark}\label{rem:analytic}
Our previous argument implies in particular that the map $t\mapsto
\cL_{f,t\phi}$ is real analytic: given $g\in C^\al(M,\mathbb R)$ the previous argument shows that
$\mathcal{L}_{f, t\phi}(g)$ is also equal to the power series
$$
\sum_{i = 0}^{\infty}\mathcal{L}_{f,\phi} \left(\frac{1}{i!} [(1-t)\phi]^{i}g\right)
	\!=\! \mathcal{L}_{f,\phi}(g) + \mathcal{L}_{f,\phi}( (1-t)\phi g)
		\!+\! \sum_{i = 2}^{\infty}\mathcal{L}_{f,\phi}\left(\frac{1}{i!}[(1-t)\phi]^{i}g\right)
$$
which is convergent.
\end{remark}

Let us mention that \cite{VV10} proved that the sequence $\frac1n \sum_{j=0}^{n-1} f^j_*\nu_{f,\phi}$ converges to the
unique equilibrium state $\mu$. Here we deduce much stronger properties fundamental to the proof that the spectral radius of the Ruelle-Perron-Frobenius operator varies differentially with respect to the potential $\phi$. We show
that $ (\tilde{\cL}_{f,\phi}^n)^*\xi$ converge exponentially fast to $\nu_{f,\phi}$ for any probability measure $\xi\in \cM(M)$. More precisely,

\begin{proposition}\label{prop:dual}
Assume $\phi \in \cW$. There exists $C>0$ and $\tau\in (0,1)$ such that for every $\varphi\in C^\al(M,\mathbb R)$ and every
probability measure $\xi\in \cM(M)$ it holds that
\begin{equation*}
\left|  \int \varphi \; d (\tilde{\cL}_{f,\phi}^n)^*\xi -\int h_{f,\phi} \, d\xi \int \varphi \, d\nu_{f,\phi}  \right|
	\leq C \tau^n \|\varphi\|_\alpha.
\end{equation*}
\end{proposition}

\begin{proof}
The proof is a simple consequence of the spectral gap property. In fact,
\begin{align*}
\left|  \int \varphi \, d (\tilde{\cL}_{f,\phi}^n)^*\xi -\int h_{f,\phi} d\xi \int \varphi d\nu_{f,\phi}  \right|
	& \leq \int \left|   \tilde{\cL}_{f,\phi}^n  (\varphi) - h_{f,\phi}  \int \varphi d\nu_{f,\phi}  \right| d \xi \\
	& \leq  \left\|   \tilde{\cL}^n_{f,\phi}  (\varphi) - h_{f,\phi}  \int \varphi d\nu_{f,\phi}  \right\|_0
	 \leq C \tau^n \|\varphi\|_\alpha,
\end{align*}
where $C$ and $\tau$ are given by Theorem~\ref{t.gap}. This proves our proposition.
\end{proof}

In the case $f$ is an expanding map,
$\phi,\psi \in C^\al(M,\mathbb R)$, $\mu_\phi$ is the unique equilibrium state for $f$ with respect to $\phi$ then
the pressure function $P(t)=\Ptop(f,\phi+t\psi)$ is complex analytic and $P'(0) = \int \psi \, d\mu_\phi$
(cf.~\cite[Proposition~4.40]{PP90}).
The following proposition asserts that in our non-uniformly expanding context the pressure function is analytic
as a function of the potential $\phi$ and gives the expected expression for its derivative.

\begin{proposition}\label{raio}
The spectral radius map $\mathcal{W} \ni  \phi \mapsto \lambda_{ f, \phi}$ is analytic.
Furthermore, given
$ \phi_{0} \in \mathcal{W}$ and $H \in C^{\alpha}(M, \R)$ we have:
$$
D_{\phi}\lambda_{f, \phi \;|\phi_{0}} \cdot H
	= \lambda_{f, \phi_{0}} \cdot \int h_{f, \phi_{0}} \cdot H \; d\nu_{f, \phi_{0}}.
$$
\end{proposition}

\begin{proof}
Note that it is immediate that $\phi \mapsto \la_{f,\phi}$ is analytic, since $\phi \mapsto \cL_{f,\phi}$
is analytic (in the norm operator topology), and since the spectral radius of $\cL_{f,\phi}$ coincides
with an isolated eigenvalue of $\cL_{f,\phi}$ with multiplicity one.
Let us calculate explicitly the derivative of $\phi \mapsto \la_{f,\phi}$.

Let $\phi_{0} \in \mathcal{W}$ be fixed.
 It follows from the $C^{0}$-statistical stability statement in Proposition~\ref{prop:C0cont} that
$\la_{f,\phi}^{-n}  \mathcal{L}_{f,\phi}^{n}(1)  \to h_{f,\phi}$
and that the limit is  uniform in a small neighborhood $W$ of $\phi_{0}$. Moreover, since
$\tilde{\mathcal{L}}_{f, \phi}^{n}(1)(x) \leq K$ for some constant  $K$ that can be taken uniform in $W$
it follows that $h_{f,\phi}$ can be taken uniformly bounded from above for all $\phi\in W$. Since the sequence
$\tilde{\mathcal{L}}_{f,\phi}^{n}(1)$ is Cauchy in the projective metric it also follows that  $h_{f,\phi}$ can be
taken uniformly bounded away from zero for all $\phi\in W$. In consequence,
$
\lim_{n\to\infty} \frac{1}{n}\log\int\mathcal{L}_{f,\phi}^{n} 1 \; d\nu_{f,\phi_{0}}
	= \log \lambda_{f,\phi}
$
uniformly with respect to  $\phi\in W$. Hence, we consider the family of functionals $F_n: W\to \mathbb R$
given by
$$
F_{n}({\phi})
	= \frac{1}{n}\log\int\mathcal{L}_{{f,\phi}}^{n} 1\, d\nu_{f,\phi_{0}},
$$
which are well defined and converge to the constant $\log \la_{f,\phi}$, and prove
that the derivatives of  $F_{n}$ converge uniformly as $n$ tends to infinity. Write $DF_{n}({\phi}) \cdot H$ as
\begin{equation}\label{abcd}
	  \frac{\int D_{\phi}\mathcal{L}^{n}_{f,\phi}(1)_{|{\phi}} \cdot H d\nu_{f,\phi_{0}} }{ n \cdot \int	
		\mathcal{L}^{n}_{{f,\phi}}(1)d\nu_{f,\phi_{0}}}
	= \frac{ \int \sum_{i = 1}^{n}\mathcal{L}_{{f,\phi}}^{i}(\mathcal{L}_{{f,\phi}}^{n-i}(1) \cdot H) \, d\nu_{f,\phi}}
		{n \cdot \int \mathcal{L}^{n}_{{f,\phi}}(1)d\nu_{f, \phi_{0}}}
	 = \frac{A_{n}(\hat{\phi}) \cdot H}{\int\tilde{\mathcal{L}}^{n}_{f,\hat{\phi}}(1)d\nu_{f, \phi_{0}}},
 \end{equation}
where $A_n$, that uses the normalized operators $\tilde \cL_{f,\phi}=\la_{f,\phi}^{-1}\cL_{f,\phi}$, is given by
$$
A_{n}({\phi}) \cdot H
	= \frac{1}{n}\int \sum_{i = 1}^{n}\tilde{\mathcal{L}}_{{f,\phi}}^{i} (\tilde{\mathcal{L}}_{{f,\phi}}^{ n - i}(1)
	\cdot H) \; d\nu_{f, \phi_{0}}.
$$
Taking into account Proposition~\ref{prop:dual} it follows that
\begin{align*}
 |A_{n}({\phi}) \cdot H & - \frac{1}{n}\sum_{i = 0}^{n - 1}\int\tilde{\mathcal{L}}_{{f,\phi}}^{i}(1) \cdot H d\nu_{{f, \phi}} \cdot \int h_{{f, \phi}} d\nu_{f, \phi_{0}}| \\
& \leq
\frac{1}{n}\sum_{i = 1}^{n}\left| \int \tilde{\mathcal{L}}_{{f,\phi}}^{ n - i}(1) \cdot H d(\tilde{\mathcal{L}}_{{f,\phi}}^{\ast i}\nu_{f,\phi_{0}}) \!-\! \int\tilde{\mathcal{L}}_{{f,\phi}}^{n - i}(1) \cdot H d\nu_{{f, \phi}}  \int h_{{f,\phi}} d\nu_{f, \phi_{0}} \right| \\
& \leq \frac{1}{n}\sum_{i = 1}^{n} C\tau^{i} \cdot \|\tilde{\mathcal{L}}_{{f,\phi}}^{n - i}(1) \cdot H\|_{\alpha}
\leq \frac{1}{n}\sum_{i = 1}^{n}4 C\tau^{i} \cdot (C\tau^{n - i} + \|h_{{f,\phi}}\|_{\alpha}) \cdot\|H\|_{\alpha},
\end{align*}
which is uniformly convergent to zero with respect to ${\phi}$ and unitary $H \in C^{\alpha}(M, \R)$.
Furthermore,
$$
\frac{1}{n}\sum_{i = 0}^{n - 1}\int\tilde{\mathcal{L}}_{{f,\phi}}^{i}(1) \cdot H d\nu_{{f, \phi}} \cdot \int h_{{f,\phi}} d\nu_{ \phi_{0}} \xrightarrow[n \to \infty]{} \int h_{{f,\phi}} \cdot H d\nu_{{f,\phi}} \cdot \int h_{{f,\phi}} d\nu_{f,\phi_{0}}
$$
and this  convergence is uniform with respect to ${\phi}$ e $H$. Since $\int\tilde{\mathcal{L}}^{n}_{{f,\phi}} 1\, d\nu_{f, \phi_{0}}$ converges to $\int h_{{f,\phi}}d\nu_{f,\phi_{0}}$ uniformly with respect to ${\phi}$, we obtain that
$$
DF_{n}({\phi}) \cdot H
	= \frac{A_{n}({\phi}) \cdot H}{\int\tilde{\mathcal{L}}^{n}_{{f,\phi}}(1)d\nu_{f,\phi_{0}}}
	\rightarrow \int h_{{f,\phi}} \cdot H \; d\nu_{{f,\phi}},
$$
where the convergence is uniform with respect to ${\phi}$ and  $H \in C^{\alpha}(M, \R)$ satisfying
$\|H\|_{\alpha} = 1$. Now, just observe that $e^{F_n(\phi)}$ is differentiable and uniformly convergent to $\la_{f,\phi}$.
Thus, as a consequence of the chain rule it follows that
$$
D_{\phi}\lambda_{f, \phi \;|\phi_{0}} \cdot H =\lambda_{f,\phi_{0}} \cdot \int h_{f,\phi_{0}} \cdot H \; d\nu_{f,\phi_{0}}.
$$
This finishes the proof of the proposition.
\end{proof}

In our context $\Ptop(f,\phi)=\log \la_{f,\phi}$. Thus the arguments in the later proof yield the following immediate consequence:

\begin{corollary}\label{cordifpot}
The map $\mathcal{W} \ni  \phi \mapsto P_{top}(f, \phi)$ is analytic. Furthermore, given $ \phi_{0} \in \mathcal{W}$ and $H \in C^{\alpha}(M, \R)$ we have:
$$
D_{\phi}P_{top}(f, \phi)_{ \;|\phi_{0}} \cdot H
	= \int h_{f,\phi_{0}} \cdot H \; d\nu_{f, \phi_{0}}
	= \int H \; d\mu_{f, \phi_{0}}.
$$
\end{corollary}

From Proposition~\ref{prop:C0cont} the invariant density is H\"older continuous function and varies continuously
in the $C^0$-topology. Here we show that it varies differentially with respect to the potential.

\begin{proposition}\label{pontofixo}
The map $\mathcal{W}^{r+\alpha} \ni \phi \mapsto h_{f, \phi} \in C^{r+\alpha}(M, \R)$ is analytic.
\end{proposition}

\begin{proof}

For the sequel, we recall some fundamental facts in spectral theory.
Given a  Banach space $E$ and a linear operator $L \in \mathcal{L}(E,E)$,
we say that a subset $S$ of the spectrum $\sigma(L)$ is an spectral component
if it is open and closed in $\sigma(L)$. In such case, $S^c:= \sigma(L)\setminus S$
is also an spectral component. If $E$ is a complex space, denoting by $E_S$ and $E_{S^c}$ the invariant
subspaces associated respectively to $S$ and $S^c$, the projection over $E_S$ that vanishes
in $E_{S^c}$ is given by
\begin{equation}
P_{E_S}= \frac{1}{2\pi i} \int_\ga \rho(z) dz, \label{eqproj}
\end{equation}
where $\rho(z) := [z I - L]^{-1}$ is the resolvent map and $\gamma$ is a closed regular curve contained
in the resolvent set such that $S$ is in the bounded open region confined by $\gamma$.
In the case
$E$ is not a complex space, but just a real space, we complexity $L$ extending it in the natural way
to an operator $\hat L$ defined in $\hat E:= E + i E$, and the same above remains valid.
In this case if $S \subset \re$, then the real $L$-invariant subspace associated to $S$ is
just the image of projection
$P_{E_S}$ restricted to the copy of $E$ in $\hat E$. Note that in any case, equation \ref{eqproj}
gives the invariant space in a implicit way. It only guarantees that the projection is an analytic
map with respect to $L \in \mathcal{L}(E,E)$. However, in our case, this also implies that if we set
$L= \mathcal{L}_{f, \phi}$ the
dominant  eigenspace varies analytically with respect to $\phi$. In fact,
$$
h_{f, \phi}=  \frac{1}{2\pi i}\Big( \int_\ga [z I- \mathcal{L}_{f, \phi}]^{-1}dz\Big)(1),
$$
since, by \cite{CV13}, $h_{f, \phi}$ is the projection of $1$ over
the dominant eigenspace of $\mathcal{L}_{ f, \phi}$ that vanishes in
$\{g : \int g d\nu_{f, \phi}= 0\}=  \{g:  \la^{-n} {\mathcal{L}}^n_{f, \phi}(g) \to 0  \text{ as } n \to +\infty\}$, which
corresponds to the invariant subspace associated to the non-dominant part
of $\sigma(\mathcal{L}_{ f, \phi})$. This implies that $\mathcal{W} \ni \phi \mapsto h_{f, \phi} \in C^{r+\alpha}(M, \R)$ is analytic, since $\phi \mapsto \mathcal{L}_{ f, \phi}$ is analytic by Proposition \ref{prop:derivative.L.phi}.
which finishes the proof of the proposition.
\end{proof}

Now we use the previous information to deduce that the conformal measures $\nu_\phi$ vary analytically.
The precise statement is as follows:

\begin{proposition}\label{prop:dif.conformal.potential}
The map $\mathcal{W}^{r+\al} \ni \phi \mapsto \nu_{f,\phi} \in \mathcal({C}^{r +\alpha}(M,\mathbb R))^{\ast}$ is analytic. In particular, the map $\mathcal{W}^{r+\al} \ni \phi \mapsto \int g\;  d\nu_{f,\phi}$ is analytic for any fixed $g \in C^{\alpha}(M, \R)$.
\end{proposition}

\begin{proof}

Fix any $g \in C^{\alpha}(M, \R)$. Then, for any $x \in M$, $ \int g\;  d\nu_{f,\phi}= \frac{P(g)}{P(1)}(x)$,
where
$$
P(g)=  \frac{1}{2\pi i}\Big( \int_\ga [z I- \mathcal{L}_{f, \phi}]^{-1}dz\Big)(g).
$$
Therefore taking any $x_0 \in M$ we just have  $\nu_{f,\phi}( \cdot)= \frac{P(\cdot)}{P(1)}(x_0)$, which implies the
analiticity of $\mathcal{W} \ni \phi \mapsto \nu_{f,\phi} \in \mathcal({C}^{r +\alpha} (M,\mathbb R))^{\ast}$.

\end{proof}

We will now deduce the analyticity of the equilibrium states $\mu_\phi$ with respect to the
potential $\phi$. In fact, using that $\mu_{f,\phi} = h_{f,\phi} \; \nu_{f,\phi}$ the following consequence is
immediate from our previous two differentiability results.

\begin{corollary}\label{cor:dif.eq.potential}
The map $\mathcal{W}^{r+\al} \ni \phi \mapsto \mu_{f,\phi} \in \mathcal({C}^{r +\alpha}(M,\mathbb R))^{\ast}$ is analytic and, consequently,  $\mathcal{W}^{r+\al} \ni \phi \mapsto \int g \;d\mu_{f,\phi}$ is analytic for any fixed $g \in C^{\alpha}(M, \R)$. 
\end{corollary}

Now, notice that it follows from our previous results that $h_{f,\phi}$, $\nu_{f,\phi}$ and $\mu_{f,\phi}$ vary analytically with
respect to the potential $\phi$ by the explicit power series formulas obtained in the previous results. This finishes the proof
of Theorem~\ref{thm:B}.

\subsection{Differentiability of the topological pressure with respect to the dynamics}\label{subsec:difdynamics1}

In this subsection we prove the differentiability of the topological pressure using the differentiability of
inverse branches for the dynamics. More precisely,

\begin{lemma}{(Local Differentiability of inverse branches)}
Let $r \geq 1$, $0 \leq k \leq r$ and $f:M \to M$ be a $C^r$-local diffeomorphism on a compact connected manifold $M$.
Let $B= B(x, \delta)\subset M$ be a ball such that the inverse branches $f_1, \dots, f_s : B \to M$ are
well defined diffeomorphisms onto their images. Then $C^r(M, M) \ni f \mapsto (f_1, \dots, f_s) \in C^{r-k}$ is a $C^{k}$ map.
\end{lemma}

\begin{proof}

Let $F: C^r(M,M)\times [C^{r-k}(B, M)]^s \to  [C^{r-k}(B; M)]^s$
given by
$$
F(h, \underbrace{h_1, \dots, h_s}_{:= \underline{h}})= (h \circ h_1, \dots, h\circ h_s).
$$
Note that $F$ is $C^k$. In fact, on one hand, $\partial_h F$ is $C^\infty$ (in suitable charts, we
see it as a continuous linear map in $h$).
By Theorem~4.2 and Corollary~4.2 in \cite{Fr79}, the composition $h\mapsto h\circ h_i$ is differentiable and
a precise expression for the derivative is given. In our setting,
 for an increment $H= (H_1, \dots, H_s) \in T[C^{r-k}(B; M)]^s$, we obtain that
$\partial_{h_j} F \cdot H_j= h'\circ h_j \cdot H_j$, which is clearly a $C^{k}$ map.

Note that $F(f, f_1, \dots, f_s)= (id, \dots, id)$.
For the point $(f, f_1, \dots, f_s)$, we have that
$\partial_{\underline{h}} F(f, f_1, \dots, f_s) \cdot H= (f'\circ f_1 \cdot H_1, \dots, f'\circ f_s \cdot H_s)$, is an isomorphism, since
$f$ is a local diffeomorphism and so $[f' \circ f_j(x)]$ is invertible, for any $x \in M$.
Therefore, by Implicit Function Theorem, we obtain that the map $G: C^r(M,M)\to [C^{r-k}(B, M)]^s$ given by
$f \mapsto (f_1, \dots, f_s)$ is a $C^{k}$ map, and its derivative applied to an increment $h \in T [C^r(M, M)]$ is
$$
(DG\cdot h)(x)= (-f'_1(x) \cdot h \circ f_1(x), \dots, -f'_s(x) \cdot h \circ f_s(x))
$$
This finishes the proof of the lemma.
\end{proof}

Using the differentiability of the inverse branches we establish the general fact that the transfer operator
varies differentiably in the operator norm of the space $L(C^{r+\alpha}(M, \mathbb R), C^{r -1}(M,\mathbb R))$
with respect to the underlying dynamical system.
This will be the key point to deduce that $f\mapsto \mu_f$ is differentiable and deduce further dynamical consequences as differentiability of Lyapunov exponents.

\begin{proposition}{(Differentiability of transfer operators)} \label{propdif}
Let $r \geq 1$, $\alpha> 0$ and let $f:M \to M$ be a $C^{r+\alpha}$-local diffeomorphism
on a compact connected manifold $M$
and $\phi\in C^{r+\alpha}(M,\mathbb R)$ be any fixed potential.
The map
$$
\begin{array}{ccc}
\text{Diff}^{\;r+\alpha}_{loc}(M)  & \to & \cL(C^{r+\alpha}(M, \mathbb R), C^{r-1}(M,\mathbb R))\\
f &  \mapsto & \cL_{f, \phi}
\end{array}
$$
is $C^1$-differentiable.
\end{proposition}

\begin{proof}
Let $\{\vr_j: j= 1, \dots, l\}$ be a $C^\infty$  partition of unity associated to some finite covering $B_1, \dots B_l$
of $M$ by balls with radius smaller or equal to $\delta> 0$ and define the auxiliary  operators
$\cL_j= \cL_{j, f, \phi}:= \cL_{f, \phi} \cdot \vr_j$. In particular it holds that $\cL_{f, \phi}= \sum_{j= 1}^{l} \cL_j$.
Therefore, all we need to prove is that any auxiliary operator $\cL_j$ is differentiable as a function of $f$.

We claim first the following pointwise differentiability: for any fixed $g \in C^{r}( M, \re)$ the map
$f\mapsto \cL_j(g)  \in C^{r-1}( M, \re) $ is  differentiable.
Recall that $\vr_j$ vanishes outside the ball $B_j$. We write $f_1, \dots, f_s$ for
the inverse branches of $f$ in $B_j$, and also $T_i= \partial_f f_i$, for $i= 1, \dots, s$.
Therefore, by a slight abuse of notation, since $\cL_j(g)|_{\ov{B_j}^c} \equiv 0$ we have
$$
\cL_{j}(g)= \sum_{i= 1}^s g(f_i) \cdot e^{\phi}(f_i) \, \vr_j,
$$
which implies that $\partial_f \cL_{j}(g)\cdot H=  \sum_{i= 1}^s (g \cdot e^{\phi})' \circ f_i \cdot [T_i \cdot H] \cdot \vr_j$
does exist, which proves our claim.

As a consequence of the previous pointwise differentiability result it is clear that $C^{r+\al}(M , \re) \ni g \mapsto \partial_f \cL_{j}(g) \in C^{r-1}(M,\mathbb R)$ is linear and continuous. Now, fix $f_{0} \in  \text{Diff}^{r}_{loc}(M)$, a chart $\Phi$ defined in a neighborhood 
of $f_0$ and $\Phi^{-1}(f_0) = \hat{f_{0}} \in C^r(M, \mathbb R^m)$
(cf. Subsection~\ref{sec:Banach-manifold}). 
For any $g \in C^{r}(M, \mathbb R)$, the fundamental theorem of 
calculus  implies that 
\begin{align*}
& \|\cL_{j, \Phi^{-1}( \hat{f_{0}} + H)}(g) - \cL_{j, \Phi^{-1}( \hat{f_{0}})}(g) - \partial_{\hat{f}}\cL_{j, \Phi^{-1}( \hat{f})}(g)_{| \hat{f_{0}}} \cdot H \|_{r-1} \\
& = \|\int_{0}^{1} \Big[\partial_{\hat{f}}\cL_{j, \Phi^{-1}( \hat{f})}(g)_{|\hat{f_{0}} + tH} - \partial_{\hat{f}}\cL_{j, \Phi^{-1}( \hat{f})}(g)_{|\hat{f_{0}}} \Big]\cdot H \,dt\, \|_{r-1} \\ 
& \leq \sum_{i=1}^{s}\int_{0}^{1}\| \Big[ (ge^{\phi})'_{|\big(\Phi^{-1}(\hat{f_{0}} + t H)\big)_{i}} \cdot [T_{i|\Phi^{-1}(\hat{f_{0}} + tH)} \cdot \partial_{\hat{f}}\Phi^{-1}(\hat{f})_{|\hat{f_{0}} +tH} \cdot H] \\
& \qquad \qquad - (g e^{\phi})'_{|\big(\Phi^{-1}(\hat{f_{0}} )\big)_{i}} \cdot [T_{i|\Phi^{-1}(\hat{f_{0}})} \cdot \partial_{\hat{f}}\Phi^{-1}(\hat{f})_{|\hat{f_{0}}} \cdot H] \Big] \cdot \varphi_{j} \|_{r-1} \; dt \\
&
\leq
 \sum_{i=1}^{s}\int_{0}^{1} \Big[ \|  (ge^{\phi})'_{|\big(\Phi^{-1}(\hat{f_{0}} + t H)\big)_{i}} \cdot [T_{i|\Phi^{-1}(\hat{f_{0}} + tH)} \cdot \partial_{\hat{f}}\Phi^{-1}(\hat{f})_{|\hat{f_{0}} +tH} \cdot H]  \\
&
\qquad \qquad -  (ge^{\phi})'_{|\big(\Phi^{-1}(\hat{f_{0}} )\big)_{i}} \cdot [T_{i|\Phi^{-1}(\hat{f_{0}} + tH)} \cdot \partial_{\hat{f}}\Phi^{-1}(\hat{f})_{|\hat{f_{0}} + tH} \cdot H] \|_{r-1} \\
  &
  \qquad \qquad +
   \|  (ge^{\phi})'_{|\big(\Phi^{-1}(\hat{f_{0}} )\big)_{i}} \cdot [T_{i|\Phi^{-1}(\hat{f_{0}} + tH)} \cdot \partial_{\hat{f}}\Phi^{-1}(\hat{f})_{|\hat{f_{0}} +tH} \cdot H] \\
   &
\qquad \qquad -
  (ge^{\phi})'_{|\big(\Phi^{-1}(\hat{f_{0}} )\big)_{i}} \cdot [T_{i|\Phi^{-1}(\hat{f_{0}})} \cdot \partial_{\hat{f}}\Phi^{-1}(\hat{f})_{|\hat{f_{0}} } \cdot H] \|_{r-1}           \Big] \cdot \|\varphi_{j} \|_{r-1} \, dt.
\end{align*}
Thus, if $C:= 2^{r+1} \|H\|_{r-1} \, \|\varphi_j\|_{r-1} \,   \|g e^\phi\|_{r+ \alpha}$ one can bound the left hand side using  
\begin{align*}
& \|\cL_{j, \Phi^{-1}( \hat{f_{0}} + H)}(g) - \cL_{j, \Phi^{-1}( \hat{f_{0}})}(g) - \partial_{\hat{f}}\cL_{j, \Phi^{-1}( \hat{f})}(g)_{| \hat{f_{0}}} \cdot H \|_{r-1}  \\
\\
&
\le C \sum_{i=1}^{s}\int_{0}^{1} \Big[ \| \big(\Phi^{-1}(\hat{f_{0}} +t H)\big)_{i} - \big(\Phi^{-1}(\hat{f_{0}} )\big)_{i} \|_{r-1} \cdot \|T_{i|\Phi^{-1}(\hat{f_{0}} + tH)} \cdot \partial_{\hat{f}}\Phi^{-1}(\hat{f})_{|\hat{f_{0}} +tH} \|_{r-1}
\\
& 
\qquad \qquad +
\|T_{i|\Phi^{-1}(\hat{f_{0}} + tH)} \cdot \partial_{\hat{f}}\Phi^{-1}(\hat{f})_{|\hat{f_{0}} + tH} - T_{i|\Phi^{-1}(\hat{f_{0}})} \cdot \partial_{\hat{f}}\Phi^{-1}(\hat{f})_{|\hat{f_{0}} } \|_{r-1}  \Big] dt.
\end{align*}
This implies that
$$
\lim_{H \to 0}\sup_{\|g\|_{r+\alpha} = 1}\frac{
\|\cL_{j, \Phi^{-1}( \hat{f_{0}} + H)}(g) - \cL_{j, \Phi^{-1}( \hat{f_{0}})}(g) - \partial_{\hat{f}}\cL_{j, \Phi^{-1}( \hat{f})}(g)_{| \hat{f_{0}}} \cdot H \|_{r-1}  }{\|H\|_{r+\alpha}} = 0.
$$
and finishes the proof of the proposition.
\end{proof}

\begin{proposition}\label{prop:diff.transfer.f}
Let $r\ge 1$,  $\alpha> 0$,  and $\phi, g \in C^{r+\alpha}(M, \R)$ be fixed. Then, the map  ${Diff}_{loc}^{r+\alpha} \ni f \mapsto
\mathcal{L}_{f,\phi}^{n} \in L(C^{r+\alpha}(M, \R), C^{r-1}(M, \R))$ is $C^1$-differentiable. Furthermore, given
$H \in \Gamma^{r+\alpha}_{f_{0}}$,
$g_{1} , g_{2} \in C^{r+\alpha} (M, \R)$ and $t \in \R$ it holds
\begin{itemize}
\item[i)] $D_{f}(\mathcal{L}_{f,\phi}^{n}(g))_{|f_{0}} \cdot H
	= \sum_{i = 1}^{n}\mathcal{L}_{f_{0},\phi}^{i - 1} (D_{f}\mathcal{L}_{f,\phi}
	(\mathcal{L}_{f_{0},\phi}^{ n - i}(g))_{|f_{0}} \cdot H)$;
\item[ii)] there exists $c_{f,\phi}>0$ so that
	$\|D_{f}\mathcal{L}_{f,\phi}(g)_{|f_{0}} \cdot H\|_{0} \leq c_{f,\phi} \|g\|_{1} \; \|H\|_{1}$;
\item[iii)] $D_{f}\mathcal{L}_{f,\phi}^{n}(g_{1} + tg_{2})_{|f_{0}} \cdot H =   D_{f}\mathcal{L}_{f,\phi}^{n}(g_{1})_{|f_{0}} \cdot H + tD_{f}\mathcal{L}_{f,\phi}^{n}(g_{2})_{|f_{0}} \cdot H$;
\item[iv)] if $\phi \equiv 0$, then $D_{f}\mathcal{L}_{f}^{n}(1)_{|f_{0}} \cdot H \equiv 0$.
\end{itemize}
\end{proposition}

\begin{proof}
Property i) is obtained by induction as follows.
The case $n= 1$ follows from the previous proposition. Now suppose the formula is valid for $n$.
Using the induction assumption and the fact that $\cL^n_{f_0+ H,\phi}(g)$ also belongs in
$C^r(M, \R)$ we obtain for all fixed diffeomorphism $f_0$ and $H \in \Gamma^{r+\alpha}_{f_{0}}$ that
 \begin{align*}
 \mathcal{L}_{f_0+ H,\phi}^{n+ 1}(g)
 	& =  \cL_{f_0,\phi} (\cL_{f_0+ H,\phi}^n(g))+ D_f\cL_{f,\phi}|_{f_0}(\cL_{f_0+ H}^n(g)) \cdot H + o(H) \\
	& = \cL_{f_0,\phi} \Big( \cL^n_{f_0}(g)+ \sum_{i = 1}^{n}\mathcal{L}_{f_{0},\phi}^{i - 1} (D_{f}\mathcal{L}_{f,\phi}
	(\mathcal{L}_{f_{0},\phi}^{ n - i}(g))_{|f_{0}} \cdot H)+ \hat o(H) \Big) \\
	& + D_f\cL_{f,\phi}|_{f_0}(\cL_{f_0+ H,\phi}^n(g)) \cdot H + o(H) \\
	& = \cL_{f_0,\phi} \Big( \cL^n_{f_0,\phi}(g)+  \sum_{i = 1}^{n}\mathcal{L}_{f_{0},\phi}^{i - 1} (D_{f}\mathcal{L}_{f,\phi}
	(\mathcal{L}_{f_{0,\phi}}^{ n - i}(g))_{|f_{0}} \!\cdot \! H)+ \hat o(H)\Big) + \hat o(H) \\
	& + D_f\cL_{f,\phi}|_{f_0}(\cL_{f_0,\phi}^n(g)) \!\cdot \! H+ (D_f\cL_{f,\phi}|_{f_0}(D_{f}(\mathcal{L}_{f,\phi}^{n}(g))_{|f_{0}} \cdot H
	+ o(H)) \cdot H \\
	& = \cL_{f_0,\phi} \Big(\cL^n_{f_0,\phi}(g)\Big)+ \sum_{i = 1}^{n+1}\mathcal{L}_{f_{0},\phi}^{i - 1} (D_{f}\mathcal{L}_{f,\phi}
	(\mathcal{L}_{f_{0},\phi}^{ (n+ 1) - i}(g))_{|f_{0}} \cdot H+ \tilde o(H),
 \end{align*}
 where $o(H), \hat o(H), \tilde o(H)$ are terms converging to zero faster than $\|H\|$. This finishes the proof
 of i).
 Part ii) is obtained by straightforward computation using the explicit formula from the previous proposition.
Part iii) follows using that
\begin{align*}
 \mathcal{L}_{f + H,\phi}^{n}(g_{1} + tg_{2})
 	&  = \mathcal{L}_{f + H,\phi}^{n}(g_{1}) + t\mathcal{L}_{f + H,\phi}^{n}(g_{2}) \\
	& =  \mathcal{L}_{f,\phi}^{n}(g_{1})  + t\mathcal{L}_{f,\phi}^{n}(g_{2}) + D_{f}\mathcal{L}_{f,\phi}^{n}(g_{1}) \cdot H \\
	& + tD_{f}\mathcal{L}_{f,\phi}^{n}(g_{2}) \cdot H +r_{1}(H) + r_{2}(H)
\end{align*}
for $r_1(H), r_2(H)$ that tend to zero as $H$ approaches zero. Finally, part iv) follows immediately from the
fact that $\mathcal{L}_{f,0}^{n}(1) \equiv \deg(f)^{n}$ and that $\deg(f)$ is locally constant. This finishes the proof.
\end{proof}

Throughout, $f$ will denote a local diffeomorphism satisfying (H1) and (H2) and $\phi$ a H\"older potential
such that (P) holds.

\begin{lemma}\label{lesimplif}
For any probability measure $\eta$, the  topological pressure $\Ptop(f,\phi)$ is given by
$$
\Ptop(f, \phi)
	=  \lim_{n \to +\infty} \frac{1}{n} \log \left[ \int \cL_{f, \phi}^n(1) \,d\eta \right].
$$
In particular, for any given $x \in M$
$$
\Ptop(f, \phi)=  \lim_{n \to +\infty} \frac{1}{n} \log [\cL_{f, \phi}^n(1)(x) ].
$$
\end{lemma}

\begin{proof}
Since the second assertion above is a direct consequence of the first one with $\eta=\de_x$ (the Dirac measure
at $x$) it is enough to prove the first one. Let $\eta$ be any fixed probability measure.
Recall that  the topological pressure is the logarithm of the spectral radius of
the transfer operator $\cL_{f, \phi}$, that is,  $\Ptop(f,\phi)=\log \la_{f,\phi}$. Moreover, since $\cL_{f, \phi}$ is a positive
operator then the spectral radius can be computed as
$$
\la_{f,\phi}
	= \lim_{n \to +\infty} \sqrt[n]{\|\cL_{f, \phi}^n\|}
	= \lim_{n \to +\infty} \sqrt[n]{\| \cL_{f, \phi} ^n(1)\|_0}
$$
Using that the functions $\la_{f,\phi}^{-n} \,\cL^n_{\phi, f}(1)$ are uniformly convergent to
the eigenfunction $h_{f,\phi}$ which is bounded away from zero and infinity
(see \cite[Proposition~4.4]{CV13})  there exists $K>0$ and
$n_0\ge 1$ such that $ K^{-1}\leq \la_{f,\phi}^{-n} \,\cL^n_{\phi, f}(1) \leq K$ for all  $n\ge n_0$. In consequence, we
get
\begin{equation*}
\lim_{n\to\infty} \frac1n \log \int \la_{f,\phi}^{-n} \,\cL^n_{f,\phi}(1) \, d\eta=0,
\end{equation*}
which proves the lemma.
\end{proof}

The next lemma will be fundamental to study the differentiability of equilibrium states. In fact we
show that the topological pressure associated to smooth potentials is differentiable.

\begin{lemma}[Differentiability of Topological Pressure with respect to dynamics]\label{lepress}
Let $\phi$ be a fixed $C^{2}$ potential on $M$ satisfying (P').  Then the topological pressure function
$P_\phi: \cF^{2}(M) \to \re$ given by $P_\phi(f)= \Ptop(f, \phi)$ is $C^1$-differentiable with respect to $f$.
\end{lemma}

\begin{proof}
By the last lemma we are reduced to prove the differentiability
of the function
$$
C^{2} \ni f  \mapsto P(f,\phi)=\lim_{n \to +\infty} \frac{1}{n} \log \int \cL_{f,\phi}^n(1) \,  d\nu_{f_{0},\phi}
$$
for some fixed $f_{0}$.
We will use derivation of sequence $P_n(f) =   \frac{1}{n} \log \int \cL_{f,\phi}^n(1) \, d\nu_{f_{0},\phi}$,
which converge to the topological pressure of $f$ uniformly in a small neighborhood of $f_{0}$.
By the chain rule, the derivative of $P_n$ with respect to $f$ is given by
 $$
 D_f P_n(f)
	= \frac{\nu_{f_{0},\phi} ((\frac{d}{df} \cL_{f,\phi}^n(1)(\cdot))}
		{n \, \nu_{f_{0},\phi} (\cL_{f,\phi}^n(1)(\cdot))}.
 $$

This yields that
\begin{align*}
D_fP_{n}(\hat{f}) \cdot (H)
	& =
\frac{\int D_{f }\mathcal{L}^{n}_{f ,\phi}(1)_{|\hat{f} } \cdot (H) d\nu_{f_{0},\phi} }{ n \cdot \int\mathcal{L}^{n}_{\hat{f},\phi}(1)d\nu_{f_{0},\phi}} \\
	& =
\frac{ \int \sum_{i = 1}^{n}\mathcal{L}_{\hat{f},\phi }^{i - 1} (D_{f}\mathcal{L}_{f ,\phi}(\mathcal{L}_{\hat{f},\phi}^{ n - i}(1))_{|\hat{f} } \cdot (H)) d\nu_{f_{0},\phi } }{ n \cdot \int\mathcal{L}^{n}_{\hat{f},\phi }(1)d\nu_{f_{0},\phi}}.
\end{align*}
In fact after multiplication by $\lambda_{\hat f,\phi}^{-n}$ in both the numerator and denominator,
the later can be written also as the sum
$$
\frac{ \int \sum_{i = 1}^{n} \tilde{\mathcal{L}}^{i-1}_{\hat{f} , \phi}(\sum_{j = 1}^{\deg(\hat{f})} e^{{\phi}(\hat{f}_{j}(\cdot))} \cdot D\tilde{\mathcal{L}}_{\hat{f} ,\phi}^{ n - i}(1)_{|\hat{f}_{j}(\cdot)} \cdot [(T_{j|\hat{f}} \cdot H)(\cdot)]) \; d\nu_{f_{0},\phi } }{n \lambda_{\hat{f},\phi}\cdot \int\tilde{\mathcal{L}}^{n}_{\hat{f} ,\phi}(1)d\nu_{f_{0},\phi }}
$$
\begin{equation}\label{abc}
+\frac{ \int \sum_{i = 1}^{n}\tilde{\mathcal{L}}_{\hat{f} ,\phi}^{i - 1} (\sum_{j = 1}^{deg(\hat{f})} e^{{\phi}(\hat{f}_{j}(\cdot))} \cdot \tilde{\mathcal{L}}_{\hat{f} }^{ n - i}(1)(\hat{f}_{j}(\cdot)) \cdot D{\phi}_{|\hat{f}_{j}(\cdot)} \cdot [(T_{j|\hat{f}} \cdot H)(\hat{f}_{j}(\cdot))]) \; d\nu_{f_{0},\phi } }{n \lambda_{\hat{f},\phi}\cdot \int\tilde{\mathcal{L}}^{n}_{\hat{f},\phi}(1)d\nu_{f_{0},\phi}},
\end{equation}
where $\hat f_j$ denote the inverse branches of the map $\hat f$.
To analyze the previous expressions we consider the two sums below
$$
B_{n}(\hat{f} ) \cdot H = \frac{1}{n}\int \sum_{i = 1}^{n}\tilde{\mathcal{L}}_{\hat{f},\phi }^{i - 1} (\sum_{j = 1}^{\deg(\hat{f})} e^{{\phi}(\hat{f}_{j}(\cdot))} \cdot D\tilde{\mathcal{L}}_{\hat{f},\phi}^{ n - i}(1)_{|\hat{f}_{j}(\cdot)} \cdot [(T_{j|\hat{f}} \cdot H)(\cdot)]) \; d\nu_{f_{0},\phi }
$$
and 
$$
C_{n}(\hat{f} ) \cdot H
	\! = \!\frac{1}{n}\int \sum_{i = 1}^{n}
	\tilde{\mathcal{L}}_{\hat{f},\phi }^{i - 1}
	\Big(\!
	\sum_{j = 1}^{\deg(\hat{f})} e^{{\phi}(\hat{f}_{j}(\cdot))}    \tilde{\mathcal{L}}_{\hat{f} ,\phi}^{ n - i}(1)(\hat{f}_{j}(\cdot))
		 D{\phi}_{|\hat{f}_{j}(\cdot)}  [(T_{j|\hat{f}} \cdot H)(\hat{f}_{j}(\cdot))] \Big)  d\nu_{f_{0},\phi }.
$$

To establish our result we will use the following:

\vspace{.2cm}
\noindent \emph{Claim 1:}
$B_{n}(\hat{f} )\cdot H$ is uniformly convergent on $(\hat{f} , \phi)$ and $H \in \Gamma^{2}_{\hat{f}}$ with $\|H\|_{2}
 \leq 1$ to the expression $\int \sum_{j = 1}^{\deg(\hat{f})} e^{{\phi}(\hat{f}_{j}(\cdot))} \cdot Dh_{\hat{f} , \hat{\phi}|\hat{f}_{j}(\cdot)} \cdot [(T_{j|\hat{f}} \cdot H)(\cdot)] d\nu_{\hat{f} ,\phi} \cdot 
\int h_{\hat{f},\phi } d\nu_{f_{0},\phi }$.
\vspace{.2cm}

\noindent \emph{Claim 2:}
$C_{n}(\hat{f}) \cdot H$ is uniformly convergent on $(\hat{f} , \phi)$ and $H \in \Gamma^{2}_{\hat{f}}$ with $\|H\|_{2}  \leq 1$ to the expression
$$
\int \sum_{j = 1}^{\deg(\hat{f})} e^{{\phi}(\hat{f}_{j}(\cdot))} \cdot h_{\hat{f} , {\phi}}(\hat{f}_{j}(\cdot)) \cdot D{\phi}_{|\hat{f}_{j}(\cdot)} \cdot [(T_{j|\hat{f}} \cdot H)(\hat{f}_{j}(\cdot))] \; d\nu_{\hat{f} } \cdot 
\int h_{\hat{f} , {\phi}} d\nu_{f_{0},\phi}.
$$
\vspace{.2cm}

We notice that our result will be a direct consequence of the two claims above. Indeed, using \eqref{abc}
it follows that
$$
D_f P_{n}(\hat{f}) \cdot (H) =
\frac{B_{n}(\hat{f} , \hat{\phi}) \cdot H}{\lambda_{\hat{f},\phi }\int\tilde{\mathcal{L}}^{n}_{\hat{f},\phi }(1)d\nu_{f_{0},\phi }} +
\frac{C_{n}(\hat{f} , \hat{\phi}) \cdot H}{\lambda_{\hat{f} ,\phi}\int\tilde{\mathcal{L}}^{n}_{\hat{f},\phi }(1)d\nu_{f_{0},\phi }}.
$$
Moreover, using that $\int\tilde{\mathcal{L}}^{n}_{\hat{f} ,\phi}(1)d\nu_{f_{0} ,\phi}$ converges to
$\int h_{\hat{f} , {\phi}}d\nu_{f_{0} ,\phi}$ and the uniform limits given by Claims 1 and 2 we obtain that
$DP_{n}(\hat{f} ) \cdot H$ is convergent as $n\to\infty$ to the sum
\begin{align*}
& \lambda_{\hat{f} ,{\phi}}^{-1} \int \sum_{j = 1}^{\deg(\hat{f})} e^{{\phi}(\hat{f}_{j}(\cdot))} \cdot Dh_{\hat{f} {\phi}|\hat{f}_{j}(\cdot)} \cdot [(T_{j|\hat{f}} \cdot H_{1})(\cdot)] \; d\nu_{\hat{f} ,\phi} \\
& + \lambda_{\hat{f} ,{\phi}}^{-1} \int \sum_{j = 1}^{\deg(\hat{f})} e^{{\phi}(\hat{f}_{j}(\cdot))} \cdot h_{\hat{f} , {\phi}}(\hat{f}_{j}(\cdot)) \cdot D{\phi}_{|\hat{f}_{j}(\cdot)} \cdot [(T_{j|\hat{f}} \cdot H)(\hat{f}_{j}(\cdot))] \; d\nu_{\hat{f},\phi},
\end{align*}
where the convergence is uniform on $(\hat{f} , {\phi})$ and $H \in \Gamma^{2}_{\hat{f}}$ such that  $\|H\|_{2} = 1$.
Hence
\begin{align*}
D_{f }\lambda_{f , \phi \;|{\hat f}, \phi} \cdot H
	& = \sum_{j = 1}^{\deg({\hat f})}  \int e^{\phi(\hat f_{j}(\cdot))}  Dh_{{\hat f} ,
		\phi|\hat f_{j}(\cdot)}     [(T_{j|{\hat f}} \cdot H)(\cdot)] \; d\nu_{{\hat f} , \phi} \\
	& + \!\!\sum_{j = 1}^{deg({\hat f})} \int e^{\phi(\hat f_{j}(\cdot))}    h_{{\hat f} , \phi}(\hat f_{j}(\cdot))    D\phi_{|\hat f_{j}
		(\cdot)}    [(T_{j|{\hat f}} \cdot H)(\hat f_{j}(\cdot))] d\nu_{{\hat f},\phi}.
\end{align*}
which proves the lemma. Therefore, in the remaining we prove the previous claims.
As for Claim~1, observe that the following uniform convergence holds
\begin{align*}
& \frac{1}{n} \sum_{i = 0}^{n - 1} \int \sum_{j = 1}^{\deg(\hat{f})} e^{{\phi}(\hat{f}_{j}(\cdot))} \cdot D\left[\tilde{\mathcal{L}}_{\hat{f} ,\phi}^{ i}(1)\right]_{|\hat{f}_{j}(\cdot)} \cdot [(T_{j|\hat{f}} \cdot H)(\cdot)] \; d\nu_{\hat{f}} \cdot \int h_{\hat{f} ,{\phi}} d\nu_{ f_{0} ,\phi} \\
& \xrightarrow[n \to \infty]{}
\int \sum_{j = 1}^{\deg(\hat{f})} e^{{\phi}(\hat{f}_{j}(\cdot))} \cdot Dh_{\hat{f} , {\phi}|\hat{f}_{j}(\cdot)} \cdot [(T_{j|\hat{f}} \cdot H)(\cdot)] d\nu_{\hat{f} , \hat{\phi}} \cdot \int h_{\hat{f} , {\phi}} d\nu_{f_{0} ,\phi}.
\end{align*}
In fact, since the transfer operator $\cL_{\hat f,\phi}$ acting on the space
$C^r(M,\mathbb R)$ has a spectral gap and $\tilde \cL_{\hat f, \phi}^i (1)$ tends to $h_{\hat f, \phi}$ as
$i\to\infty$ (see Theorem~\ref{thm.oldstuff} and Subsection~5.1 in \cite{CV13}) then $D\tilde \cL_{\hat f, \phi}^i (1)$
converges to $h_{\hat f, \phi}$ in the $C^0$-topology. Moreover, one also has that
\begin{align*}
 \Big| B_{n}(\hat{f} , \hat{\phi}) \cdot H
 	& \!-\! \frac{1}{n} \sum_{i = 0}^{n - 1} \int \!\!\sum_{j = 1}^{\deg(\hat{f})} e^{{\phi}(\hat{f}_{j}(\cdot))} D
	\tilde{\mathcal{L}}_{\hat{f} , {\phi}}^{ i}(1)_{|\hat{f}_{j}(\cdot)}  [(T_{j|\hat{f}} \cdot H)(\cdot)] \; d\nu_{\hat{f} ,\phi}
	\int h_{\hat{f} , {\phi}} d\nu_{f_{0} ,\phi}\Big| \\
	&\leq  \frac{1}{n}\sum_{i = 1}^{n} \Big| \int  \sum_{j = 1}^{\deg(\hat{f})} e^{{\phi}(\hat{f}_{j}(\cdot))}  D
	\tilde{\mathcal{L}}_{\hat{f} ,\phi}^{ n - i}(1)_{|\hat{f}_{j}(\cdot)}  [(T_{j|\hat{f}} \cdot H)(\cdot)]  \;
	 d(\tilde{\mathcal{L}}_{\hat{f} ,\phi}^{\ast i - 1}\nu_{f_{0} ,\phi}) \\
	 & - \int  \sum_{j = 1}^{\deg(\hat{f})} e^{{\phi}(\hat{f}_{j}(\cdot))}  D\tilde{\mathcal{L}}_{\hat{f} ,\phi}^{ n - i}(1)_{|\hat{f}_{j}(\cdot)}  [(T_{j|\hat{f}} \cdot H)(\cdot)]  \; d\nu_{\hat{f}} \cdot \int h_{\hat{f} , {\phi}} d\nu_{f_{0} ,\phi}\Big|  \\
	  & \leq  \frac{1}{n}\sum_{i = 1}^{n}C \tau^{i - 1}\deg(\hat{f})\|e^{{\phi}}\|_{1}
	  \!\! \max_{1 \le j \le \deg(\hat{f})} \{\|(T_{j|\hat{f}} \cdot H)\|_{1}\} \,
	  \|\tilde{\mathcal{L}}_{\hat{f} ,\phi}^{ n - i}(1)\|_{2} \\
	  & \leq  \frac{1}{n}\sum_{i = 1}^{n} C \tau^{i - 1}\deg(\hat{f})\|e^{{\phi}}\|_{1}
	  \!\! \max_{1 \le j \le \deg(\hat{f})} \{\|(T_{j|\hat{f}} \cdot H)\|_{1}\}  [C \tau^{n - i} + \|h_{\hat{f} , {\phi}}\|_{2}]
\end{align*}
which is uniformly convergent to zero with respect to  $(\hat{f} , {\phi})$ and all $H \in \Gamma^{2}_{\hat{f}}$ with
$\|H\|_{2} \leq 1$. This proves Claim~1. We now proceed to prove Claim~2.
\begin{align*}
& \frac{1}{n} \sum_{i = 0}^{n - 1} \int \sum_{j = 1}^{deg(\hat{f})} e^{{\phi}(\hat{f}_{j}(\cdot))}
 \tilde{\mathcal{L}}_{\hat{f},\phi}^{i}(1)(\hat{f}_{j}(\cdot))  D{\phi}_{|\hat{f}_{j}(\cdot)}  [(T_{j|\hat{f}} \cdot H)(\hat{f}_{j}(\cdot))]  d\nu_{\hat{f} , {\phi}} \cdot \int h_{\hat{f} , {\phi}} d\nu_{f_{0} ,\phi} \\
& \xrightarrow[n \to +\infty]{} \int \sum_{j = 1}^{deg(\hat{f})} e^{{\phi}(\hat{f}_{j}(\cdot))}
h_{\hat{f} , {\phi}}(\hat{f}_{j}(\cdot))  D{\phi}_{|\hat{f}_{j}(\cdot)} \cdot [(T_{j|\hat{f}} \cdot H)(\hat{f}_{j}(\cdot))]  d\nu_{\hat{f},\phi } \cdot \int h_{\hat{f} , {\phi}} d\nu_{f_{0} ,\phi},
\end{align*}
uniformly with respect to  $(\hat{f} , {\phi})$ and $H$. Then,  the difference between $C_{n}(\hat{f} ) \cdot H$ and
$\frac{1}{n} \! \sum_{i = 0}^{n - 1} \! \int \sum_{j = 1}^{\deg(\hat{f})} e^{{\phi}(\hat{f}_{j}(\cdot))}
	\tilde{\mathcal{L}}_{\hat{f} ,\phi}^{i}(1)(\hat{f}_{j}(\cdot)) D{\phi}_{|\hat{f}_{j}(\cdot)}
	[(T_{j|\hat{f}} \cdot H)(\hat{f}_{j}(\cdot))] d\nu_{\hat{f},\phi}  \int h_{\hat{f} , {\phi}} d\nu_{f_{0},\phi}$
is bounded from above (in absolute value) by
\begin{align*}
 & \frac{1}{n}\sum_{i = 1}^{n} \Big| \int \sum_{j = 1}^{deg(\hat{f})} e^{{\phi}(\hat{f}_{j}(\cdot))}
	\tilde{\mathcal{L}}_{\hat{f} ,\phi}^{ n - i}(1)(\hat{f}_{j}(\cdot))  D{\phi}_{|\hat{f}_{j}(\cdot)}
	[(T_{j|\hat{f}} \cdot H)(\hat{f}_{j}(\cdot))] \; d(\tilde{\mathcal{L}}_{\hat{f} ,\phi}^{\ast i - 1}\nu_{f_{0} ,\phi}) \\
	& -\!  \int \sum_{j = 1}^{\deg(\hat{f})} e^{{\phi}(\hat{f}_{j}(\cdot))}
	\tilde{\mathcal{L}}_{\hat{f} ,\phi}^{n - i}(1)(\hat{f}_{j}(\cdot))  D{\phi}_{|\hat{f}_{j}(\cdot)}
	[(T_{j|\hat{f}} \cdot H)(\hat{f}_{j}(\cdot))] \; d\nu_{\hat{f} , {\phi}} \cdot \int h_{\hat{f} , \hat{\phi}}
	d\nu_{f_{0},\phi} \Big|\\
	&\le \frac{1}{n}\sum_{i = 1}^{n} C \tau^{i-1}\|\sum_{j = 1}^{\deg(\hat{f})} e^{{\phi}(\hat{f}_{j}(\cdot))}
	\tilde{\mathcal{L}}_{\hat{f},\phi}^{ n - i}(1)(\hat{f}_{j}(\cdot)) \cdot D{\phi}_{|\hat{f}_{j}(\cdot)}
	 [(T_{j|\hat{f}} \cdot H)(\hat{f}_{j}(\cdot))]\|_{1} \\
	 & \leq \frac{1}{n}\sum_{i = 1}^{n} 4 \hat C  C\tau^{i-1}\deg(\hat{f}) \cdot \|e^{{\phi}}\|_{1} \cdot \|
	 \tilde{\mathcal{L}}_{\hat{f} ,\phi}^{ n - i}(1)\|_{1}  \\
	 & \leq \frac{1}{n}\sum_{i = 1}^{n} 4 \hat C C \tau^{i-1}\deg(\hat{f}) \cdot \|e^{{\phi}}\|_{1}
	 \cdot (C\tau^{n - i} + \|h_{\hat{f} , {\phi}}\|_{1})
\end{align*}
where $\hat C=\|{\phi}\|_{1} \cdot \max_{j = 1 , \ldots, \deg(\hat{f})} \{\|[(T_{j|\hat{f}} \cdot H)(\hat{f}_{j}(\cdot))]\|_{1} \}$.
Since the later expression is uniformly convergent to zero with respect to $(\hat{f} ,{\phi})$ and
$ H \in \Gamma^{2}_{\hat{f}}$ such that  $\|H\|_{2} \leq 1$ this proves Claim~2 and finishes the proof
of the lemma.
\end{proof}

\begin{corollary}\label{cor:jointdiff}
The topological pressure  $P_{top}:\cF^{2} \times \cW^{2} \to \re$ is $C^1$-differentiable.
\end{corollary}
\begin{proof}
Just note that the derivatives calculated in Corollary \ref{cordifpot} and in the Lemma above are partial derivatives
for the function $P_{top} (f, \phi)$, and jointly continuous with respect to both variables $f$ and $\phi$.

\end{proof}

\subsection{Differentiability of maximal entropy measure with respect to dynamics}\label{subsec:difdynamics2}

Through this section we deal with maximal entropy measures and henceforth we fix the potential
$\phi \equiv 0$ and  fix $f_0$ local diffeomorphism satisfying (H1) and (H2). For that reason we shall omit the
dependence on $\phi$.
Recall that for every $C^1$ local diffeomorphism $f$ satisfying (H1) and (H2) we have maximal eigenvalue
$\la_{f} = \deg(f)$, eigenfunction $h_{f} = \frac{d\mu_f}{d\nu_f} = 1$ and conformal measure $\nu_{f} = \mu_{f}$
for the Perron-Frobenius operator. In particular, the topological entropy $h_{\text{top}}(f)=\log \deg(f)$ is constant.

Let $r\in \mathbb N_0$ and $\al\in [0,1)$ be such that $r+\al>0$ and $f\in \cF^{r+\al}$ be given. It follows from  Theorems~\ref{t.gap}
and ~\ref{t.gapCr} that all transfer operators $\mathcal{L}_{f }: C^{k+\al}(M, \R) \to C^{k+\al}(M, \R)$ have
the spectral gap property for $k+\al \in \{\al, 1+\al,\dots, r+\al\} \cap \mathbb R^*_+ $, provided that $f$ is
sufficiently $C^{r+\al}$-close to $f_0$.
In consequence, it is not hard to check that if $E^{k+\al}_{0,f}=\{\hat g \in C^{k+\al}(M, \R) :  \int \hat g \,d\nu_f =0\}$
then $C^{k+\al}(M, \R) = \{\ell \, h_f : \ell\in \mathbb R\} \oplus E^{k+\al}_{0,f}$ is a
$\tilde \cL_{f,\phi}$-invariant decomposition in $C^{k+\al}(M, \R)$. Furthermore there are constants $C_{f,k+\al}>0$
and $\tau_{f,k+\al}\in(0,1)$ such that for all  $\hat g\in  E^{k+\al}_{0,f}$ it follows:
$$
\|\tilde\cL_{f}^n \hat g \|_k \le C_{f,k+\al} \tau_{f,k+\al}^n \|\hat g\|_k,
	\quad \text{ for all } n\ge 1.
$$
Set $C_f=\max\{ C_{f,k+\al} : k\le r \}$, $\tau_f=\max\{ \tau_{f,k+\al} : k\le r \}$. We also set $c_{f}$ to be a bound
for the norm of of $D_f\tilde\cL_f$. Notice that
these constants can be taken uniform in a neighborhood of $f_0$.
For that reason, we shall omit the dependence of $C_f$, $c_f$, and $\tau_f$ on $f$.
Consider also the spectral projections
$P^{k+\al}_{0,f}: C^{k+\al}(M,\mathbb R) \to E^{k+\al}_{0,f}$ given by
$P^{k+\al}_{0,f}(g)= g - \int g \, d\nu_f$. In what follows,  when no confusion is possible we shall omit
the dependence on $f$ in the corresponding subspaces and spectral projections.

\begin{theorem} \label{thm:diff.max}
The map $\mathcal{F}^{2} \ni f \mapsto \mu_{f} \in (C^{2}(M, \R))^{\ast}$ is
differentiable. In particular, for any $g \in C^{2}(M, \R)$ the map $\mathcal{F}^{2} \ni f \mapsto \int g \; d\mu_{f}$ is
$C^1$-differentiable and its derivative acting in  $H \in \Gamma^{2}_{f_{0}}$ is given by
$$
D_{f}\mu_{f}(g)_{|f_{0}} \cdot H
	= \sum_{i = 0}^{\infty}\int D_{f}\tilde{\mathcal{L}}_{f}(\,\tilde \cL_{f_{0}}^{i}(P_{0}(g))\, ) \cdot H \,d\mu_{f_{0}}.$$
\end{theorem}

\begin{proof}
Let $g \in C^{2}(M, \R)$ and $f_{0} \in \mathcal{F}^{2}$ be fixed. We define a sequence of
maps $F_{n}: \mathcal{F}^{2} \to \mathbb R$ given by $F_{n}(f ) = \int \tilde{\mathcal{L}}_{f}^{n}(g) \, d\mu_{f_{0}}$ and notice that $F_{n}(f)$ is convergent to $\int g \,d\mu_f$, whereas the convergence is uniform
in a sufficiently small neighborhood of $f_0$ and for $g$ in the unit sphere of $C^{2}(M,\R)$,
as a consequence of the estimates in Proposition~\ref{prop:dual}. Indeed, for the
potential $\phi\equiv 0$ one has $\mu_{f,\phi}= \nu_{f,\phi}$ and $h_{f,\phi}=1$, and there are
constants $C>0$ and $\tau\in (0,1)$ (uniform in a neighborhood of $(f,\phi)$, such that for every
probability measure $\xi\in \cM(M)$
\begin{equation*}
\left|  \int \tilde{\cL}_{f,\phi}^n \varphi \; d \xi - \int \varphi \, d\mu_{f,\phi}  \right|
	\leq C \tau^n \|\varphi\|_\alpha.
\end{equation*}
 for every $\varphi\in C^\al(M,\mathbb R)$ and $n\ge 1$.
Moreover, if $H \in \Gamma^{2}_{f}$ then
\begin{align*}
DF_{n}(f) \cdot H 
                        & =  \sum_{i = 1}^{n - 1}\int \tilde{\mathcal{L}}_{f}^{i - 1}(D_{f}\tilde{\mathcal{L}}_{f}
                        (\tilde{\mathcal{L}}_{f}^{n - i}(g)) \; \cdot H) \; d\mu_{f_{0}} \\
                        & =   \sum_{i = 1}^{n - 1}\int \tilde{\mathcal{L}}_{f}^{i - 1}(D_{f}\tilde{\mathcal{L}}_{f}
                        \left( \int g \, d\mu_{f} + \tilde{\mathcal{L}}_{f}^{n - i}(P_{0}(g)) \right) \; \cdot  H) \; d\mu_{f_{0}} \\
                        & =   \sum_{i = 1}^{n - 1}\int \tilde{\mathcal{L}}_{f}^{i - 1}(D_{f}\tilde{\mathcal{L}}_{f}
                        (\tilde{\mathcal{L}}_{f}^{n - i}(P_{0}(g))) \; \cdot  H) \; d\mu_{f_{0}}.
\end{align*}
On the other hand, since we assumed $\phi\equiv 0$ then $\mu_{f}=\nu_f$ and $\cL_{f,\phi}^*\mu_f=\mu_f$.
Thus,
$$
\sum_{i = 1}^{n - 1}\int D_{f}\tilde{\mathcal{L}}_{f}(\tilde{\mathcal{L}}_{f}^{n - i}(P_0(g)) \; \cdot  H \; d\mu_{f}
= \sum_{i = 0}^{n - 1}\int D_{f}\tilde{\mathcal{L}}_{f}(\tilde{{\mathcal{L}}_{f}^{i}}(P_0(g))\; \cdot  H \; d\mu_{f}
$$
and
\begin{align*}
\sum_{i = 0}^{n - 1}|\int D_{f}\tilde{\mathcal{L}}_{f}(\tilde{\mathcal{L}}_{f}^{i}(P_{0 , f}(g)))\; \cdot  H \; d\mu_{f}|
	& \leq  \sum_{i = 0}^{n - 1}\| D_{f}\tilde{\mathcal{L}}_{f}(\tilde{\mathcal{L}}_{f}^{i}(P_{0 , f}(g)))_{f}\; \cdot H\|_{0} \\
         & \leq  \sum_{i = 0}^{n - 1}\|\tilde{\mathcal{L}}_{f}^{i}(P_{0 , f}(g))\|_{1} \cdot \|H\|_{1}\cdot c \\
         & \leq  \sum_{i = 0}^{n -1} C\tau^{i} \cdot 2\cdot \|g\|_{1} \cdot c \cdot \|H\|_{1},
\end{align*}
that is bounded from above by $\frac{C}{1 - \tau}\cdot 2\|g\|_{1} c \|H\|_{1}$.
So the previous upper bound
is uniform for $g$ in the unit sphere of $C^{2}(M,\R)$ and $H$ in the unit sphere of $\Gamma^{2}_{f}$. This implies that the limit
$$
\ds\lim_{n \to +\infty}\sum_{i = 1}^{n - 1}
	\int D_{f}\tilde{\mathcal{L}}_{f}(\tilde{\mathcal{L}}_{f}^{n - i}(P_{0}(g))) \cdot  H \; d\mu_{f},
$$
does exist and is uniform with respect to the dynamics, $g$ in the unit sphere of $C^{2}(M,\R)$ and $H$ in the unit sphere of $\Gamma^{2}_{f}$.
We proceed and estimate
\begin{align*}
|& D  F_{n}(f) \cdot H
	 - \sum_{i = 1}^{n - 1}\int D_{f}\tilde{\mathcal{L}}_{f}(\tilde{\mathcal{L}}_{f}^{n - i}(P_0(g)))
	\cdot  H \, d\mu_{f}|  \\
	& \le  \sum_{i = 1}^{n - 1}
	\left|
	\int \tilde{\mathcal{L}}_{f}^{i - 1}(D_{f}\tilde{\mathcal{L}}_{f}(\tilde{\mathcal{L}}_{f}^{n - i}
	(P_0(g))) \cdot  H) \, d\mu_{f_{0}}
	\!-\! \int D_{f}\tilde{\mathcal{L}}_{f}(\tilde{\mathcal{L}}_{f}^{n - i}(P_0(g)))  \cdot  H \,
	d\mu_{f}
	\right| \\
	& = \sum_{i = 1}^{n - 1}
	\left|
	\int D_{f}\tilde{\mathcal{L}}_{f}(\tilde{\mathcal{L}}_{f}^{n - i}(P_0(g))) \cdot  H \, d\tilde{\mathcal{L}}_{f}^{\ast i - 1}
	(\mu_{f_{0}})
	\!-\! \int D_{f}\tilde{\mathcal{L}}_{f}(\tilde{\mathcal{L}}_{f}^{n - i}(P_0(g))) \cdot  H \, d\mu_{f}
	\right| \\
	& \le
	\sum_{i = 1}^{n - 1}C \tau^{i - 1}  2\|D_{f}\tilde{\mathcal{L}}_{f}(\tilde{\mathcal{L}}_{f}^{n - i}
	(P_0(g)))  H\|_{1}
	 \leq \sum_{i = 1}^{n - 1} C \tau^{i - 1} \cdot 2  \|\tilde{\mathcal{L}}_{f}^{n - i}(P_0(g))\|_{2}
	 c  \|H\|_{2} \\
	&  \leq \sum_{i = 1}^{n - 1} C \tau^{i - 1} \cdot 2 \cdot C \tau^{n - i} \cdot 2\cdot \|g\|_{2} \cdot
	c \|H\|_{2}
	 \leq  4 c C^2 \, (n - 1) \tau^{n - 1}  \cdot\|g\|_{2} \cdot \|H\|_{2}	
\end{align*}
which converges to zero. Thus
$
\lim DF_{n}(f) \cdot H
	= \sum_{i = 1}^{\infty} \int D_{f}\tilde{\mathcal{L}}_{f}(\tilde{\mathcal{L}}_{f}^{n - i}(P_0(g))) \cdot  H \,d\mu_{f}
$
uniformly with respect to the dynamics $f$, and $g$ in the unit sphere of $C^{2}(M,\R)$ and $H$ in the unit sphere of $\Gamma^{2}_{f}$. One can deduce that
for all $f$ close to $f_0$ the sequence $F_{n}(f)$ converges uniformly to $\int g \,d\mu_{f}$ and the sequence $DF_{n}$ is also uniformly convergent to the continuous linear functional defined above. We conclude that
$$
D_{f}\mu_{f}(g)_{|f_{0}} \cdot H
	= \sum_{i = 0}^{\infty}\int D_{f}\tilde{\mathcal{L}}_{f}(\tilde{\mathcal{L}}_{f_{0}}^{i}(P_{0}(g))) \cdot H d\mu_{f_{0}}.
$$
This finishes the proof of the theorem.
\end{proof}

Since previous results contain the statements of Theorems~\ref{thm:B} and ~\ref{thm:C}, their proofs are now complete.

\section{Stability and differentiability in dynamical systems}

In this section we prove that many dynamical objects vary continuously or differentially with respect
to perturbations of the potential and dynamics. We organize this into subsections for the readers convenience.

\subsection{Smoothness of the correlation function}\label{subsec:smooth.cor}

Our purpose here is to prove Corollary~\ref{cor:smooth.correlation} on the smoothness of the
correlation function $(f,\phi) \mapsto C_{\varphi , \psi}(f, \phi, n) $ and its asymptotic behavior when
$n$ tends to infinite.
Let $\cF^{2}$ an open set of local diffeomorphisms, $\cW^{\al}$ be an open set of potentials
as introduced before. Consider the observables $\varphi,\psi\in C^\al(M,\mathbb R)$ and $n\in \mathbb N$.
First we will prove that the map $(f,\phi) \mapsto C_{\varphi, \psi} (f , \phi , n)$  is analytic in $\phi$ and
differentiable in $f$ whenever $\phi\equiv 0$. Recall one can write
$$
C_{\varphi , \psi}(f, \phi, n)
= \int\varphi\Big[\tilde{\mathcal{L}}^{n}_{f ,\phi}(\psi h_{f , \phi}) - h_{f , \phi} \int\psi d\mu_{f , \phi} \Big] d\nu_{f , \phi}.
$$
This expression varies analytically with $\phi$ since it is composition of analytic functions.
As for the differentiability of the correlation function with respect to $f$, if $\phi\equiv 0$ then clearly
$\mu_f=\nu_f$, $h_{f}=1$. Moreover,  for $f_0\in \cF^{2}$ and $H\in \Gamma^{2}_{f_{0}}$
\begin{align*}
D_{f}C_{\varphi , \psi}(f , 0 , n)_{|f_{0}} \cdot H
	& =  [D_{f}\mu_{f|f_{0}} \cdot H]\Big( \varphi( \tilde{\mathcal{L}}^{n}_{f_{0}}(\psi) - \int\psi d\mu_{f_{0}}  ) \Big) \\
	& + \int \varphi \cdot \Big(D_{f}\tilde{\mathcal{L}}_{f}^{n}(\psi)_{|f_{0}}\cdot H -  [D_{f}\mu_{f|f_{0}} \cdot H](\psi) \Big)  d\mu_{f_{0}} \\
	& = \sum_{i = 0}^{\infty}\int D_{f}\tilde{\mathcal{L}}_{f}\Big(\,\tilde \cL_{f_{0}}^{i}\big(P_{0}(\varphi(
\tilde{\mathcal{L}}^{n}_{f_{0}}(\psi) - \int\psi d\mu_{f_{0}}  ))\, \big) \Big)_{|f_{0}} \cdot H \,d\mu_{f_{0}} \\
	&+ \sum_{i = 1}^{n}\int \varphi\tilde{\mathcal{L}}^{i - 1}_{f_{0}} \big( D_{f}\tilde{\mathcal{L}}_{f}(\tilde{\mathcal{L}}^{n - i}_{f_{0}}\psi)_{f_{0}} \cdot H \big)d \mu_{f_{0}} \\
	& - \sum_{i = 0}^{\infty}\int D_{f}\tilde{\mathcal{L}}_{f}\Big(\,\tilde \cL_{f_{0}}^{i}\big(P_{0}(\psi) \big) \Big)_{f_{0}} \cdot H d\mu_{f_{0}} \cdot  \int\varphi d\mu_{f_{0}}.
\end{align*}
Hence we deduce
\begin{align}
D_{f}C_{\varphi , \psi}(f , 0 , n)_{|f_{0}} \cdot H
	& =   \sum_{i = 0}^{\infty}\int D_{f}\tilde{\mathcal{L}}_{f}\Big(\,\tilde \cL_{f_{0}}^{i}\big(P_{0}(\varphi( \tilde{\mathcal{L}}^{n}_{f_{0}}(\psi) - \int\psi d\mu_{f_{0}}  ))\, \big) \Big)_{|f_{0}} \cdot H \,d\mu_{f_{0}}  \nonumber \\ \nonumber
	&+ \sum_{i = 0}^{n - 1}\int \varphi\tilde{\mathcal{L}}^{n - i - 1}_{f_{0}} \big( D_{f}\tilde{\mathcal{L}}_{f}(\tilde{\mathcal{L}}^{i}_{f_{0}}\psi)_{f_{0}} \cdot H \big)d \mu_{f_{0}} \\  \nonumber
	& - \sum_{i = 0}^{\infty}\int D_{f}\tilde{\mathcal{L}}_{f}\Big(\,\tilde \cL_{f_{0}}^{i}\big(P_{0}(\psi) \big) \Big)_{f_{0}} \cdot H d\mu_{f_{0}} \cdot  \int\varphi d\mu_{f_{0}}.  \nonumber
\end{align}
Consider the series
 $$
 A_{n}(f_{0} , H) = \sum_{i = 0}^{\infty}\int D_{f}\tilde{\mathcal{L}}_{f}\Big(\,\tilde \cL_{f_{0}}^{i}\big(P_{0}(\varphi( \tilde{\mathcal{L}}^{n}_{f_{0}}(\psi) - \int\psi d\mu_{f_{0}}  ))\, \big) \Big)_{|f_{0}} \cdot H \,d\mu_{f_{0}}
 $$
and
\begin{align*}
B_{n}(f_{0}, H)
	& = \sum_{i = 0}^{n - 1}\int \varphi\tilde{\mathcal{L}}^{n - i - 1}_{f_{0}} \big( D_{f}\tilde{\mathcal{L}}_{f}(\tilde{\mathcal{L}}^{i}_{f_{0}}\psi)_{f_{0}} \cdot H \big)d \mu_{f_{0}}  \\
	& - \sum_{i = 0}^{n- 1}\int D_{f}\tilde{\mathcal{L}}_{f}\Big(\,\tilde \cL_{f_{0}}^{i}\big(P_{0}(\psi) \big) \Big)_{f_{0}} \cdot H d\mu_{f_{0}} \cdot  \int\varphi d\mu_{f_{0}}.
\end{align*}
We will prove that both expressions $A_{n}(f_{0}, H)$ and $B_{n}(f_{0}, H)$ converge uniformly to zero in
$\{H \in \Gamma^{2}_{f_{0}} : \|H\|_{2} \leq 1 \}$ and for all $f$ close enough to $f_{0}$.
In fact, on the one hand
\begin{align*}
 |A_{n}(f_{0} , H)|
 	& \leq \sum_{i = 0}^{\infty}c \cdot \|H\|_{2} \cdot \|\tilde \cL_{f_{0}}^{i}\big(P_{0}(\varphi( \tilde{\mathcal{L}}^{n}_{f_{0}}(\psi) - \int\psi d\mu_{f_{0}}  ))\, \big)\|_{2} \\
	& \leq \sum_{i = 0}^{\infty} c \cdot \|H\|_{2} C \tau^{i} \cdot \|P_{0}\|_{2} \|\varphi\|_{2} C \tau^{n}\|\psi - \int\psi d\mu_{f_{0}}\|_{2}
\end{align*}
which is uniformly convergent to zero for $\{H \in \Gamma^{2}_{f_{0}} : \|H\|_{2} \leq 1 \}$ and $f$ close
to $f_{0}$.
On the other hand, $|B_{n}(f_{0} , H)|$ is bounded from above by
\begin{align*}
&  \sum_{i = 0}^{n-1}\|\varphi\|_{2} \cdot \|\tilde{\mathcal{L}}^{n - i - 1}_{f_{0}} \big( D_{f}\tilde{\mathcal{L}}_{f}(\tilde{\mathcal{L}}^{i}_{f_{0}}\psi)_{f_{0}} \cdot H \big) - \int D_{f}\tilde{\mathcal{L}}_{f}\Big(\,\tilde \cL_{f_{0}}^{i}\big(P_{0}(\psi) \big) \Big)_{f_{0}} \cdot H d\mu_{f_{0}}\|_{0}\\
& \le \sum_{i = 0}^{n-1} C \tau^{n -i - 1} \, \|\varphi\|_{2} \, \|D_{f}\tilde{\mathcal{L}}_{f}(\tilde{\mathcal{L}}^{i}_{f_{0}}\psi)_{f_{0}} \cdot H - \int D_{f}\tilde{\mathcal{L}}_{f}\Big(\,\tilde \cL_{f_{0}}^{i}\big(P_{0}(\psi) \big) \Big)_{f_{0}} \cdot H d\mu_{f_{0}}\|_{2}\\
& \leq \sum_{i = 0}^{n-1}\|\varphi\|_{2} \cdot C^{2}\tau^{n - 1} 4 \cdot c \cdot \|\psi\|_{2} \|H\|_{2}
\end{align*}
which is again uniformly convergent to zero as $n$ goes to infinity. This proves not only that  the correlation function
for the maximal entropy measure is differentiable in $f$ but also that $D_{f}C_{\varphi , \psi}(f , 0 , n)_{|f_{0}}$ is uniformly
convergent to zero as $n$ goes to infinity. This finishes the proof of Corollary~\ref{cor:smooth.correlation}.

\subsection{Stability of the Central Limit Theorem}\label{sec:CLT}

Our purpose here is to prove Theorem~\ref{thm:CLT}.
Let $\cW^{2}$ be an open set of
$C^{2}$ potentials and $\cF^{2}$ an open set of $C^{2}$ local diffeomorphisms  satisfying the conditions (H1), (H2) and (P') with uniform constants as before. For any $\psi\in C^{\al}(M,\mathbb R)$ consider the mean and variance given, respectively,
by
\begin{equation*}\label{eq:mean.variance}
m_{f,\phi}=\int\psi \;d\mu_{f, \phi}
	\quad\text{and}\quad
\sigma_{f,\phi}^{2}
		= \int \tilde \psi^{2}\;d\mu_{f,\phi} + 2\sum_{j = 1}^{\infty}\int \tilde \psi(\tilde \psi\circ f^{j}) \;d\mu_{f,\phi},
\end{equation*}
where $\tilde\psi=\psi-m_{f,\phi}$. We omit the dependence of $m_{f,\phi}$ and $\sigma_{f,\phi}^{2}$
on  $\psi$ for notational simplicity. By invariance of the measure $\mu_{f,\phi}$ we can also write
$$
\sigma_{f,\phi}^{2}
	= \lim_{n\to\infty} \frac1n \int \Big( \sum_{j=0}^{n-1}  \tilde \psi\circ f^j \Big)^2 \; d \mu_{f,\phi}
	\ge 0.
$$
Moreover, it follows from the exponential decay of correlations that the  Central Limit Theorem holds
(see e.g.~\cite[Corollary~2]{CV13}). So, either $\sigma_{f,\phi}^{2}=0$  and consequently $\psi = u\circ f - u+\int \psi \,d\mu_{f,\phi}$ for some $u \in L^{2}(X , \mathcal{F} , \mu_{f,\phi})$ or the random variables
$\frac1{\sqrt{n}} \sum_{j = 0}^{n - 1} \psi \circ f^{j} $ converge in distribution to the Gaussian
$\mathcal{N}(m_{f,\phi} , \sigma^2_{f,\phi})$.

First we will prove that the functions $(f,\phi) \mapsto m_{f,\phi}$ and $(f,\phi) \mapsto \sigma^2_{f,\phi}$ are analytic on  $\phi$ and that are differentiable on $f$ in the case of the maximal entropy measure, provided that $\psi$ is smooth enough.
According to Remark~\ref{complex}  the map $\phi\mapsto \cL_{f,\phi}$ is analytic whenever it acts on the space of
complex valued observables.
Thus, if $T(z) = \frac{\lambda_{f,\phi + iz\psi}}{\lambda_{f , \phi}}$ it follows from the classical perturbation theory and Nagaev's method that $\sigma_{f,\phi}^{2}(\psi) = - D^{2}_{z}T(z)_{|z = 0}$ (see e.g.~\cite{Sar12} for details) and the dependence is analytic in $\phi$.
In fact one can check
\begin{align*}
\sigma_{f,\phi}^{2}(\psi) =
	&\Big(\int{\psi} \;d\mu_{f , \phi}\Big)^{2}
	+ \int(I - \tilde{\mathcal{L}}_{f , \phi\;|E_{0}})^{-1}
		\Big({\psi} h_{f , \phi} - h_{f , \phi}\, \int{\psi} d\mu_{f , \phi}\Big)  {\psi} \; d\nu_{f , \phi} \\
	& + \int{\psi} \;d\mu_{f , \phi} \; \int (I - \tilde{\mathcal{L}}_{f , \phi\;|E_{0}})^{-1}(1 - h_{f , \phi}) \, {\psi} \;d\nu_{f , \phi}\\
	&= \Big(\int{\psi} \;d\mu_{f , \phi}\Big)^{2}+ \int \sum_{k=0}^{\infty} \tilde \cL^k_{f,\phi \mid_{E_0}}
		\Big({\psi} h_{f , \phi} - h_{f , \phi}\, \int{\psi} d\mu_{f , \phi}\Big)  {\psi} \; d\nu_{f , \phi} \\
	&+ \int{\psi} \;d\mu_{f , \phi} \; \int \sum_{k=0}^{\infty} \tilde \cL^k_{f,\phi \mid_{E_0}} (1 - h_{f , \phi}) \, {\psi} \;d\nu_{f , \phi}
\end{align*}
Hence, for any fixed $\psi\in C^{1+\al}(M,\mathbb R)$ the variance map
$(f,\phi)\mapsto \sigma_{f,\phi}^{2}(\psi)$ is continuous since it is obtained as composition of continuous functions and $\tilde \cL^k_{f,\phi } (1 - h_{f , \phi})$ is uniformly convergent to zero in a neighborhood of $(f , \phi)$.
We obtain further regularity in the case of maximal entropy measures. More precisely,

\begin{lemma}
Let $\phi\equiv 0$ and $\psi\in C^{2}(M,\mathbb R)$ be given. Then the variance map
$
\mathcal{F}^{2} \ni  f \mapsto \sigma^2_{f , 0}(\psi)
$
$C^1$-differentiable.
\end{lemma}

\begin{proof}
Notice that we want to study the variance
$$
\sigma^{2}_{f , 0}(\psi) = \int \tilde \psi^{2}\;d\mu_{f} +2\sum_{n = 1}^{\infty}C_{\tilde \psi , \tilde \psi}(f , 0 , n)
$$
where $\tilde\psi=\tilde\psi(f)=\psi-\int \psi \;d\mu_f$, that is differentiable in $f$. So, to prove the differentiability
of the previous expression, using the chain rule, we are reduced to prove the differentiability of the map
$
f \mapsto \sum_{n = 1}^{\infty}C_{\tilde \psi , \tilde \psi}(f , 0 , n)
$
assuming that the observable $\tilde \psi$ is fixed and independent of $f$. Fix $\hat f \in \mathcal{F}^{2}$ and  proceed to consider the sequence of functions
$$
F_{k}(f) = \sum_{n= 0}^{k} C_{\tilde \psi , \tilde \psi}(f , 0 , n),
$$
which is differentiable and uniformly convergent to  $F(f)=\sum_{n= 0}^{\infty}C_{\tilde \psi , \tilde \psi}(f , 0 , n)$
in a small neighborhood of  $\hat f$. We claim that the derivatives of $F_k$ are
uniformly convergent. In fact,
$
D_{f} F_k(f) \cdot H
	= \sum_{n= 0}^{k}D_{f}C_{\tilde\psi , \tilde\psi}(f , 0 , n)\mid_{{f}} \cdot H
$
and so, under the notations of Subsection~\ref{subsec:smooth.cor},
\begin{align*}
D_{f}F_{k}(f)_{|{f}} \cdot H
	 = 2\sum_{n = 1}^{k}\Big[ & A_{n}({f} , H) + B_{n}({f} , H) \\
	 & + \sum_{i = 0}^{n- 1}\int D_{f}\tilde{\mathcal{L}}_{f}\Big(\,\tilde \cL_{{f}}^{i}\big(P_{0}(\tilde\psi) \big) \Big)_{{f}} \cdot H
	d\mu_{{f}} \cdot  \int\varphi d\mu_{{f}} \\
	 &- \sum_{i = 0}^{\infty}\int D_{f}\tilde{\mathcal{L}}_{f}\Big(\,\tilde \cL_{{f}}^{i}\big(P_{0}(\tilde\psi) \big)
	\Big)_{{f}} \cdot H d\mu_{{f}} \cdot  \int\tilde\psi d\mu_{{f}} \Big].
\end{align*}
On the one hand, $\sum_{n = 1}^{n}|A_{n}({f} , H) + B_{n}({f} , H)|$ is bounded from above by
\begin{align*}
\sum_{n = 1}^{n} &
	\Big[ \sum_{i = 0}^{\infty} c \cdot \|H\|_{2} C \tau^{i} \cdot \|P_{0}\|_{2} \|\tilde\psi\|_{2} C \tau^{n}\|\tilde\psi - \int\tilde\psi d\mu_{{f}}\|_{2}\\
	&  +
	\sum_{i = 0}^{n-1}\|\tilde\psi\|_{2} \cdot C^{2}\tau^{n - 1} 4 \cdot c \cdot \|\tilde\psi\|_{2}
	\|H\|_{2} \Big],
\end{align*}
that is summable. Then it is well defined the limit $ \sum_{n = 0}^{\infty} A_{n}({f} , H) + B_{n}({f} , H) $ and the convergence is uniform for $f$ in a small neighborhood of $\hat f$ and $\|H\|_{2} = 1$. On the other hand, in view of the differentiability of the maximal entropy measure,
\begin{align*}
\sum_{j = 1}^{n} \Big|
	&\sum_{i = 0}^{j- 1}\int D_{f}\tilde{\mathcal{L}}_{f}\Big(\,\tilde \cL_{\hat{f}}^{i}\big(P_{0}(\tilde\psi) \big) \Big)_{\hat{f}} \cdot H
	d\mu_{\hat{f}} \cdot  \int\tilde\psi d\mu_{\hat{f}} \\
	& - \sum_{i = 0}^{\infty}\int D_{f}\tilde{\mathcal{L}}_{f}\Big(\,\tilde \cL_{\hat{f}}^{i}\big(P_{0}(\tilde\psi) \big) \Big)_{\hat{f}} \cdot H
	d\mu_{\hat{f}} \cdot  \int\tilde\psi d\mu_{\hat{f}} \Big| \\
	& \leq  4 c C^{2} \, \sum_{j = 1}^{n}(j - 1) \tau^{j - 1} \cdot
	{\|\tilde\psi\|_{1} \cdot \|H\|_{1}} \cdot \int\tilde\psi
	d\mu_{\hat{f}},
\end{align*}
and so, we get the convergence of the series
\begin{align*}
\lim_{n \to \infty} \sum_{j = 0}^{n}\sum_{i = 0}^{j- 1}
	&\int D_{f}\tilde{\mathcal{L}}_{f}\Big(\,\tilde \cL_{\hat{f}}^{i}\big(P_{0}(\tilde\psi) \big) \Big)_{\hat{f}} \cdot H
d\mu_{\hat{f}} \cdot  \int\tilde\psi d\mu_{\hat{f}} \\
	& -
\sum_{i = 0}^{\infty}\int D_{f}\tilde{\mathcal{L}}_{f}\Big(\,\tilde \cL_{\hat{f}}^{i}\big(P_{0}(\tilde\psi) \big)
\Big)_{\hat{f}} \cdot H d\mu_{\hat{f}} \cdot  \int\tilde\psi d\mu_{\hat{f}}
\end{align*}
also uniform in a neighborhood of $\hat f$ and with $\|H\|_{2} = 1$. This proves that
 $D_{f}F_{k}(f)_{|\hat{f}} \cdot H$ is uniformly convergent, proving that  $F$ is differentiable. This finishes the
 proof of the lemma.
\end{proof}

Finally, if  $\sigma^2_{f,\phi}> 0$ one can use the continuity of the function $(f,\phi) \mapsto \sigma^2_{f,\phi}$ to
obtain $\mathcal U \subset \cF^{2} \times \cW^{2}$ open such that for every $(\tilde f, \tilde \phi)
\in \mathcal U$ it holds that $\sigma^2_{\tilde f,\tilde \phi}> 0$.
In consequence, if $\psi$ is not a coboundary in $L^2(\mu_{f,\phi})$ then the same property holds
for all close $\tilde f$ and $\tilde \phi$. This finishes the proof of Theorem~\ref{thm:CLT} and Corollary~\ref{cor:cohomo}.

\subsection{Differentiability of the free energy and stability  of large deviations}\label{sec:deviations}

In this section we prove the differentiability of the free energy function and deduce some further properties
for large deviations corresponding to Theorems~\ref{thm:FreeEnergy} and \ref{thm:differentiability.LDP}.

\subsubsection{Free energy function}\label{s.free.energy}

First we establish some properties of the free energy function as consequence of the spectral gap property.
Recall that an observable $\psi: M\to \mathbb R$ is \emph{cohomologous to a constant} if there exists
$A \in \mathbb R$ and an observable $\tilde \psi : M \to \mathbb R$  such that $\psi=\tilde \psi \circ f - \tilde \psi+A$.
Now we prove the following:

\begin{proposition}\label{prop:free.energy}
Let $f$ and $\phi$ be as above and satisfy assumptions (H1), (H2) and (P). Then for any
H\"older continuous observable $\psi:M\to \R$ there exists $t_{\phi,\psi}>0$ such that
for all $|t|\le t_{\phi,\psi}$ the following limit exists
$$
\cE_{f,\phi,\psi}(t)
	:=\lim_{n\to\infty} \frac1n \log \int e^{t S_n \psi} \; d\mu_{f,\phi}
	= \Ptop(f, \phi +t\psi) -\Ptop(f, \phi).
$$
Moreover, if $\psi$ is cohomologous to a constant then $t \mapsto \cE_{f,\phi,\psi}(t)$ is affine and otherwise
$t \mapsto \cE_{f,\phi,\psi}(t)$ is real analytic, strictly convex. Furthermore, if $(f, \phi)\in \cF^{2} \times
\cW^{2}$ 
then for every $t\in\mathbb (-t_{\phi,\psi}, t_{\phi,\psi})$ the function $\cF^{2} \ni f\mapsto \cE_{f,\phi,\psi}(t)$ is differentiable and $\cF^{2} \ni f\mapsto \cE_{f,\phi,\psi}'(t)$ is continuous.
\end{proposition}

\begin{proof}
The first part of the proof goes along some well known arguments that we include here for completeness. Observe first that
for all $n\in \mathbb N$
\begin{align*}
\int e^{t S_n \psi} \; d\mu_{f,\phi}  & =  \int \la_{f,\phi}^{-n} \, \cL_{f,\phi}^n  (h_{f,\phi} e^{t S_n \psi}) \; d\nu_{f,\phi} \\
&= \left(\frac{\la_{f,\phi+t\psi}}{\la_{f,\phi}}\right)^n  \int \la_{f,\phi+t\psi}^{-n} \, \cL_{f,\phi+t\psi}^n  (h_{f,\phi}) \; d\nu_{f,\phi}.
\end{align*}
Since (P) is an open condition, then for every $|t|\le t_{\phi,\psi}$ the potential $\phi+t\psi$ satisfies (P) provided that
$t_{\phi,\psi}$ is small enough.

\begin{figure}[htb]
\includegraphics[width=8cm]{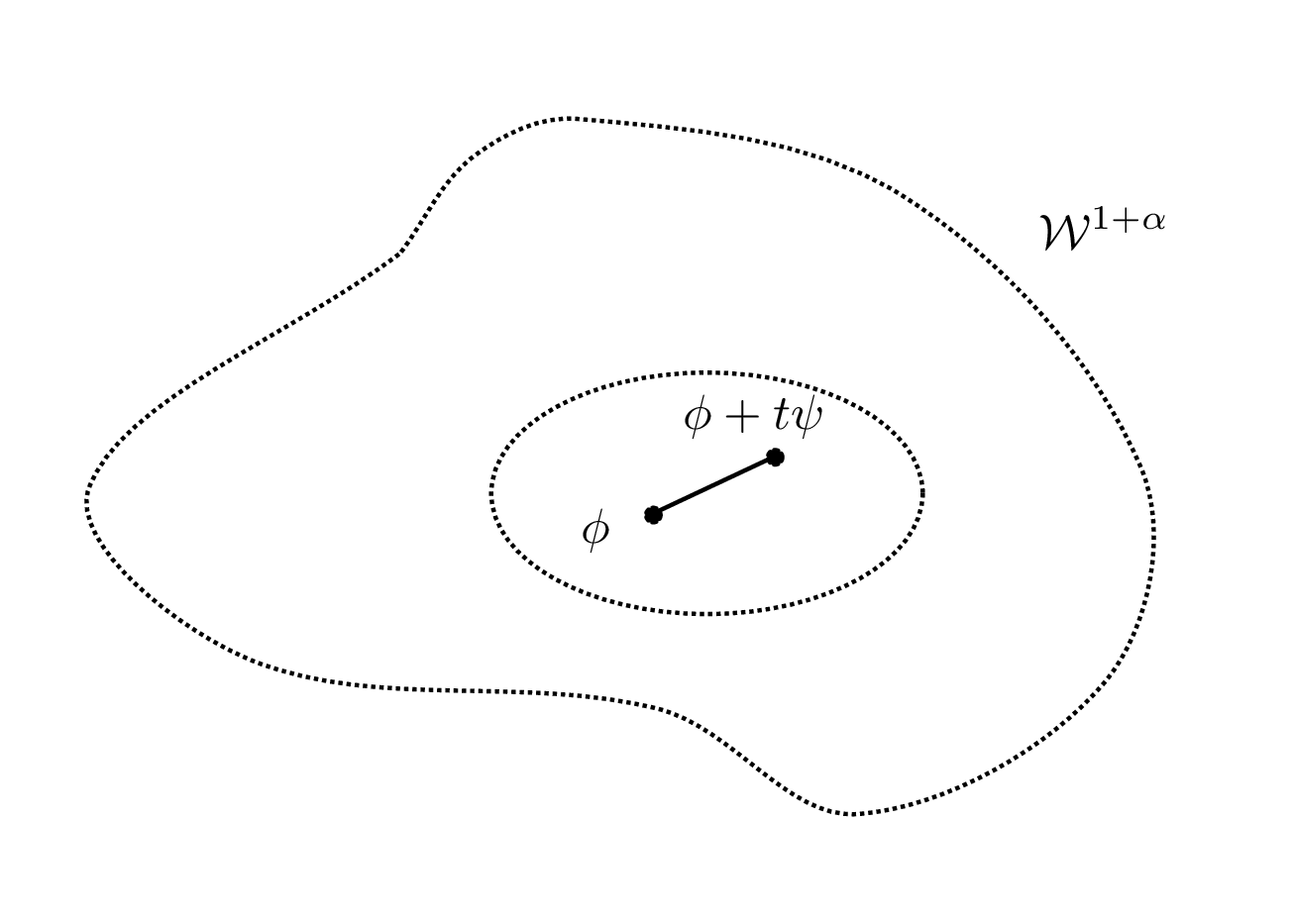}
\caption{Open domain in $\cW^{1+\alpha}$}
\end{figure}

 Since $h_{f,\phi}$ is positive and bounded away from zero and infinity this implies that
$\la_{f,\phi+t\psi}^{-n} \, \cL_{f,\phi+t\psi}^n  (h_{f,\phi})$ is uniformly convergent to $h_{f,\phi+t\psi} \cdot \int h_{f,\phi}d\nu_{f , \phi +t\psi}$, thus uniformly bounded from zero and infinity for all large $n$.
Therefore using the dominated convergence theorem
$$
\lim_{n\to\infty} \frac1n \log \int e^{t S_n \psi} \; d\mu_{f,\phi}
	= \log\la_{f, \phi +t\psi} -\log\la_{f, \phi}
	=\Ptop(f, \phi +t\psi)- \Ptop(f, \phi),
$$
proving the first assertion of the proposition.
Now, assume first that there exists  $A \in \mathbb R$ and a potential $\tilde \psi : M \to \mathbb R$  such that
$\psi=\tilde \psi \circ f - \tilde \psi+A$. Then it follows from the variational principle and invariance that
\begin{align*}
\Ptop(f, \phi +t\psi)
	& =\sup_{\mu\in \cM_1(f)} \left\{ h_\mu(f) +\int (\phi +t\psi) \, d\mu \right\} \\
	& = t A + \sup_{\mu\in \cM_1(f)} \left\{ h_\mu(f) +\int \phi \, d\mu \right\} \\
	& =  t A + \Ptop(f, \phi)
\end{align*}
and, consequently, $\cE_{f,\phi,\psi}(t)=tA$ is affine.

Now, we proceed to prove that if $\psi$ is not cohomologous to a constant then the free energy function
is strictly convex. Since $t\mapsto \Ptop(f,\phi+t\psi)$ is real analytic (recall Remark~\ref{rem:analytic})
then to prove that $t\mapsto \cE_{f,\phi,\psi}(t)$ is strictly convex it is enough to show that $\cE_{f,\phi,\psi}''(t)>0$
for all $t$.
Assume that there exists $t$ such that $\cE_{f,\phi,\psi}''(t)=0$.
Up to replace $\phi$ by the potential $\tilde \phi=\phi+t\psi$ we may assume without loss of generality
that $t=0$, that is, $\cE_{f,\phi,\psi}''(0)=0$. Hence, using Corollary~\ref{cordifpot} and differentiation under the
sign of integral we obtain
$$
\cE_{f,\phi,\psi}'(t)
	=\left.\frac{d \Ptop(f,\phi+t\psi)}{dt}\right|_{t=0}
	=\int \psi \,d\mu_{f,\phi+t\psi}
	=\lim_{n\to\infty} \frac1n \frac{ \int (S_n\psi) \, e^{t S_n\psi} d\mu_{f,\phi}}{\int e^{t S_n\psi} d\mu_{f,\phi} }
$$
(hence $\cE_{f,\phi,\psi}'(0) =\int \psi \,d\mu_{f,\phi}$). Using Theorem~B and differentiating again with respect to $t$
under the sign of integral it follows that
\begin{align*}
\cE_{f,\phi,\psi}''(t)
	& =\lim_{n\to\infty} \frac1n
		\left[
		\frac{ \int (S_n\psi)^2 \, e^{t S_n\psi} d\mu_{f,\phi}}{\int e^{t S_n\psi} d\mu_{f,\phi} }
		- \left(\frac{ \int S_n\psi \, e^{t S_n\psi} d\mu_{f,\phi}}{\int e^{t S_n\psi} d\mu_{f,\phi} } \right)^2
		\right]
	 \ge 0
\end{align*}
(hence $\cE_{f,\phi,\psi}''(0)=\lim_{n\to\infty} \frac1n [\int (S_n\psi)^2 \,d\mu_{f,\phi} -(\int S_n\psi \, d\mu_{f,\phi})^2] >0$)
because, if $\mu_n= e^{t S_n\psi} d\mu_{f,\phi}$ the inequality is equivalent to $\int S_n \psi \, d\mu_n  \le (\int (S_n\psi)^2 \, d\mu_n )^\frac12  (\int 1 \, d\mu_n )^\frac12$ that holds by
H\"older's inequality.
In particular,  $\cE_{f,\phi,\psi}''(t)=0$ if and only if $\psi$ is cohomologous to a  constant. Thus we conclude
that $\cE_{f,\phi,\psi}$ is a strictly convex function.
Finally,
using that the topological pressure is differentiable with respect to the dynamics the proof of the
proposition is now complete.
\end{proof}

The following result illustrates some characteristics of the behavior on the free energy function.

\begin{corollary}
For any H\"older continuous potential $\psi$ so that $\int \psi\, d\mu_{f,\phi}=0$
the free energy function $[-t_{\phi,\psi},t_{\phi,\psi}] \ni t\to \cE_{f,\phi,\psi}(t)$ satisfies:
\begin{enumerate}
\item $\cE_{f,\phi,\psi}(0)=0$ and $\cE_{f,\phi,\psi}(t)\ge 0$ for all $t\in (-t_{\phi,\psi},t_{\phi,\psi})$;
\item $t\inf\psi \le  \cE_{f,\phi,\psi}(t) \le t \sup\psi$ for all $t\in(0,t_{\phi,\psi}]$;
\item $t\sup\psi \le  \cE_{f,\phi,\psi}(t) \le t \inf \psi$ for all $t\in[-t_{\phi,\psi},0)$.
\end{enumerate}
\end{corollary}

\begin{proof}
It follows from the first part of Proposition~\ref{prop:free.energy} that $\cE_{f,\phi,\psi}(0)=0$. Now, since
$\cE_{f,\phi,\psi}''(t)>0$ then $\cE_{f,\phi,\psi}'$ is strictly increasing. Therefore, using $\cE_{f,\phi,\psi}'(0)
=\int \psi \,d\mu_{\phi}=0$ it follows that $\cE_{f,\phi,\psi}$ is strictly increasing for $t\in (0,t_{\phi,\psi})$ and
strictly decreasing for $t\in (-t_{\phi,\psi},0)$. This proves item (1) above.
Finally, (2) and (3) is a simple consequence of the mean value theorem and, using $\cE_{f,\phi,\psi}'(t)
=\int \psi \,d\mu_{\phi+t\psi} $, the fact that
$
\inf \psi
	\le \cE_{f,\phi,\psi}'(t)
	\le \sup\psi.
$
This finishes the proof of the corollary.
\end{proof}

In what follows assume that $\psi$ is not cohomologous to a constant and that $m_{f,\phi}=\int \psi \, d\mu_{f,\phi}=0$.
Therefore, since the function
$[-t_{\phi,\psi},t_{\phi,\psi}] \ni t\to \cE_{f,\phi,\psi}(t)$ is strictly convex it is well defined the ``local"
Legendre transform $I_{f,\phi,\psi}$ given by
\begin{equation*}
I_{f,\phi,\psi}(s)
	= \sup_{-t_{\phi,\psi}\le t \le t_{\phi,\psi}} \; \big\{ st-\cE_{f,\phi,\psi}(t) \big\}.
\end{equation*}

\begin{remark}
This is a convex function since it is supremum of affine functions and, using that $\cE_{f,\phi,\psi}$ is strictly convex
and non-negative,  $I_{f,\phi,\psi} \ge 0$. Moreover, notice that since $\cE_{f,\phi,\psi+c}(t)= \cE_{f,\phi,\psi}(t)+ct$ then we also get that $I_{f,\phi,\psi+c}(t)= I_{f,\phi,\psi}(t-c)$ for every $c,t\in\mathbb R$.
\end{remark}

When the free energy function is differentiable it is not hard to check the variational property
$
I_{f,\phi,\psi} (\cE'_{f,\psi}(t) )
	= t \cE'_{f,\psi}(t) - \cE_{f,\phi,\psi}(t)
$
and the domain of $I_{f,\phi,\psi}$ contains the interval $[\cE'_{f,\psi}(-t_{\phi,\psi}), \cE'_{f,\psi}(t_{\phi,\psi})]$.
Moreover, $I_{f,\phi,\psi}(s)=0$ if and only if $s=m_{f,\phi}$ belongs to the domain of $I_{f,\phi,\psi}$.
It is also well known that the strict convexity of $\cE_{f,\phi,\psi}$ together with differentiability of
$\cE_{f,\phi,\psi}$ yields that $[-t_{\phi,\psi},t_{\phi,\psi}] \ni t \mapsto I_{f,\phi,\psi}(t)$ is strictly convex and differentiable. Using the previous remark we collect all of these statements in the following:

\begin{corollary}
Let $f\in \cF$ be arbitrary and let $\psi$ be an H\"older continuous observable. Then the rate function
$I_{f,\phi,\psi}$ satisfies:
\begin{enumerate}
\item The domain $[\cE'_{f,\psi}(-t_{\phi,\psi}), \cE'_{f,\psi}(t_{\phi,\psi})]$ contains $m_{f,\phi}=\int \psi \,d\mu_{f,\phi}$;
\item $I_{f,\phi,\psi}\ge 0$ is strictly convex and $I_{f,\phi,\psi}(s)= 0$ if and only $s=\int \psi \, d\mu_{f,\phi}$;
\item $s\mapsto I_{f,\phi,\psi}(s)$ is real analytic.
\end{enumerate}
\end{corollary}

\subsubsection{Estimating deviations}

Now we use the previous free energy function to obtain a ``local" large deviation results.
In fact, the following results hold from Gartner-Ellis theorem (see e.g.~\cite{DZ98,RY08}) as a consequence
of the differentiability of the free energy function.

\begin{theorem}\label{thm:LDP.Ellis}
Let $f$ be a local diffeomorphism so that (H1) and (H2) holds and let $\phi$ be a H\"older continuous
potential such that (P) holds.
Given any interval $[a,b]\subset [\cE'_{f,\psi}(-t_{\phi,\psi}), \cE'_{f,\psi}(t_{\phi,\psi})]$ it holds that
$$
\limsup_{n\to\infty} \frac1n \log \mu_{f,\phi}
	\left(x\in M : \frac1n S_n\psi(x) \in [a,b] \right)
	\le-\inf_{s\in[a,b]} I_{f,\phi,\psi}(s)
$$
and
$$
\liminf_{n\to\infty} \frac1n \log \mu_{f,\phi}
	\left(x\in M : \frac1n S_n\psi(x) \in (a,b) \right)
	\ge-\inf_{s\in(a,b)} I_{f,\phi,\psi}(s)
$$
\end{theorem}

Finally, to finish this section we deduce a regular dependence of the large deviations rate function with respect to the
dynamics and potential. To the best of our knowledge these results are new even in the uniformly expanding setting.

\begin{proposition}\label{prop:differentiability.rate}
Let $\psi$ be a H\"older continuous observable. There exists an interval $J\subset \mathbb R$
containing $m_f,\phi$ such that for all $[a,b]\subset J$ and $f\in \cF^{1+\al}$
$$
\limsup_{n\to\infty} \frac1n \log \mu_{f,\phi}
	\left(x\in M : \frac1n S_n\psi(x) \in [a,b] \right)
	\le-\inf_{s\in[a,b]} I_{f,\phi,\psi}(s)
$$
and
$$
\liminf_{n\to\infty} \frac1n \log \mu_{f,\phi}
	\left(x\in M : \frac1n S_n\psi(x) \in (a,b) \right)
	\ge-\inf_{s\in(a,b)} I_{f,\phi,\psi}(s)
$$
Moreover, if $V$ is a compact metric space and  $V \ni v \mapsto f_v \in \mathcal{F}^{2}$ is a continuous and injective map then the rate function $(s,v) \mapsto I_{f_v,\phi,\psi}(s)$ is continuous on $J\times V$.
\end{proposition}

\begin{proof}
Fix $f_0\in \cF$. We obtain a large deviation principle for Birkhoff averages on subintervals of a
given interval $[\cE'_{f_0,\phi,\psi}(-t_{\phi,\psi}), \cE'_{f_0,\phi,\psi}(t_{\phi,\psi})]$ given by Theorem~\ref{thm:LDP.Ellis}. Observe that
the interval $[\cE'_{f,\phi,\psi}(-t_{\phi,\psi}), \cE'_{f,\phi,\psi}(t_{\phi,\psi})]$ is non-degenerate and varies continuously with
$f$ and $\psi$.
Hence, we may take a non-degenerate interval $J$ contained in all intervals $[\cE'_{f,\phi,\psi}(-t_{\phi,\psi}),
\cE'_{f,\phi,\psi}(t_{\phi,\psi})]$ for all $f\in \cF$ sufficiently close to $f_0$. This proves the first assertion above.

Finally, from the variational relation using the Legendre transform and the convexity of the free energy
function  (that is, $\cE''_{f,\phi,\psi}(t)>0$ for all $t$) we get that for any $s\in J$ there exists a unique $t=t(s,v)$
such that $s=\cE'_{f_v,\psi}(t) $ and
\begin{equation}\label{eq.var.rate}
I_{f_v,\phi,\psi} (s)= s \cdot t(s, v) - \cE_{f,\phi,\psi}(t(s, v)).
\end{equation}
Now, we consider the continuous skew-product
$$
\begin{array}{ccc}
F: V \times J & \to & V \times \mathbb R\\
(v,t) & \mapsto & (v, \cE'_{f_v,\phi,\psi}(t))
\end{array}
$$
and notice that it is injective because it is strictly increasing along the fibers.
Since $V\times J$ is a compact metric space then $F$ is a homeomorphism
onto its image $F(V\times J)$. In particular this shows that for every $(v,s)\in F(V\times J)$ there
exists a unique $t=t(v,s)$ varying continuously with $(v,s)$ such that $F(v,t(v,s))=(v, s)$ and
$s=\cE'_{f_v,\phi,\psi}(t)$. Finally, relation~\eqref{eq.var.rate} above yields that $(s,v) \mapsto I_{f_v,\phi,\psi}(s)$ is
continuous on $J\times V$.  This finishes the proof of the corollary.
\end{proof}

It is not hard to check that the rate function is real analytic with respect to the
potential. However, since the proof is much simpler than the previous one we shall omit it
and leave as an exercise to the reader.

\section{Some examples}\label{s.examples}

In this section we provide some applications where we discuss mainly the smooth or continuous variation relevant
dynamical quantities for non-uniformly expanding maps obtained through bifurcation theory. In particular,
although our results apply for uniformly expanding dynamics we discuss some robust examples that can be
far  from being expanding.

\subsection{One-dimensional examples}\label{s.examples.one}

\subsubsection{Discontinuity of Perron-Frobenius operator for circle expanding maps}\label{ex:discontinuous}

In order to illustrate the discontinuity of the Perron-Frobenius operator when acting on the space of
functions with low regularity. We provide a one-dimensional example just for simplicity.

We claim that he transfer operator $\cL_{f,\phi} : C^{\al}(S^1,\mathbb R) \to C^{\al}(S^1,\mathbb R)$ associated to
the doubling circle map $f : S^1 \to S^1$ is discontinuous both in the operator norm as well as pointwise.
In fact, up to consider the metric $\tilde d(x,y)=d(x,y)^\al$ we are reduced to prove the discontinuity of the
transfer operators acting on the space of Lipschitz observables.
The key idea is that the composition operator $\vr \to \vr \circ g$ acting in the
space of Lipschitz functions does not vary continuously with $g$, as we now detail.

Let $S^1 \simeq \re/[-1/2, 1/2)$ be the circle and consider the expanding maps of the circle
$f_n(x)= 2(x+ \frac1{10n}) (\!\!\mod 1)$, $n\ge 1$. It is clear that the sequence $(f_n)_n$ is convergent to
the doubling map $f(x)= 2 x (\!\!\mod 1)$ in the $C^\infty$-topology.

Now, pick a Lipschitz function $\varphi$ in the circle so that $\varphi(x)= |x|$ for $|x| \leq 1/8$ and
$\varphi(x)= 0$  for $1/2 \geq |x|\geq 1/5$, and consider the potential $\phi \equiv 0$. In this way,
if  $0< x_n< y_n< 1/10n$, we obtain that
\begin{align*}
\text{Lip}((\cL_{f_n,\phi}- \cL_{f,\phi})(\vr))
	& \geq  \frac{|\cL_{f_n,\phi}(\varphi)(y_n)- \cL_{f_n,\phi}(\varphi)(x_n) +\cL(\varphi)(x_n)- \cL(\varphi)(y_n)|}{y_n- x_n} \\
	& =  \frac{| |y_n/2 - 1/10n|- |x_n/2- 1/10n|+ |x_n/2|- |y_n/2\|}{y_n- x_n} \\
	& = \frac{ |-y_n - x_n|}{y_n- x_n}= 1 = Lip(\varphi)
\end{align*}
for all $n\in\mathbb N$. Thus the sequence of transfer operators $(\cL_{f_n,\phi})_n$ does not converge to
$\cL_{f,\phi}$ even in the strong operator topology.
Nevertheless we have that
$$
(f,\phi) \mapsto \cL_{f_n,\phi} 1=\sum_{f(y)=x} e^{\phi(y)}
$$
is indeed continuous, which was enough for us to prove the differentiability of the topological pressure function.

\subsubsection{Manneville-Pomeau maps}

Given $\al >0$, let $f_\al:[0,1]\to [0,1]$ be the local
diffeomorphism given by
\begin{equation}\label{eq.Man.Pom}
f_\al(x)= \left\{
\begin{array}{cl}
x(1+2^{\alpha} x^{\alpha}) & \mbox{if}\; 0 \leq x \leq \frac{1}{2}  \\
2x-1 & \mbox{if}\; \frac{1}{2} < x \leq 1
\end{array}
\right.
\end{equation}
and the family of potentials $\varphi_{\al,t}=-t\log|Df_\al|$. Note that $f$ is a $C^{r+\beta}$-local
diffeomorphism, where $r=1+[\alpha]$ and $\beta=\alpha-[\alpha]$. In consequence $\varphi_{\al,t}\in C^{r-1+\beta}(M,\mathbb R)$.
Since it is not hard to check that a similar construction of
an expanding map with an indifferent fixed point can be realized as a circle local diffeomorphism, we will deal with this
family for simplicity. Moreover, conditions (H1) and (H2) are clearly verified for all $f_\al$. It is well known that if $\al\in(0,1)$ then an intermittency phenomenon occurs for the potentials $\varphi_{\al,t}$ at $t=1$. However no phase transitions occur
at high temperature, as we now discuss.

Assume first $\al\in(0,2)$. The family $\varphi_{\al,t}$ of $C^{r-1+\beta}$-potentials do satisfy condition (P) for all $|t|\le t_0$ small, that can be taken not depending on $\al$ since $\alpha<2$ and
\begin{equation*}
|\varphi_{\al,t}(x)-\varphi_{\al,t}(y)|
	= |t\log|Df_\al(x)|-t\log|Df_\al(y)\|
	=|t| \log \frac{|Df_\al(x)|}{|Df_\al(y)|}
	\leq |t| \log ( 2+\al ).
\end{equation*}
In fact, if on the one hand, $\sup \varphi_{\al,t}-\inf \varphi_{\al,t} \leq t_0 \log 4$, on the other hand the H\"older constant
of  $\varphi_{\al,t}$ can be made small provided that $t$ is small.
Hence, it follows from Theorem~\ref{thm.oldstuff} that for all $|t|\le t_0$ there exists a unique equilibrium state $\mu_{\al,t}$
for $f_\al$ with respect to $\varphi_{\al,t}$, it has exponential decay of correlations in the space of H\"older
observables,  and that  the pressure
$t \mapsto \Ptop(f_\al,-t \log|Df_\al|)$ and the equilibrium state  $t \mapsto \mu_{\al,t}$
are continuous in the interval $(-t_0,t_0)$. In consequence, the Lyapunov exponent function
 $t \mapsto \lambda(\mu_{\al,t})=\int \log |f_\al'| \,d\mu_{\al,t}$ is also continuous.

Now we will discuss the case that $\al \in [2,+\infty)$. In this case one can say that $f_\al$ is
at least $C^{3}$ and the potentials $\varphi_{\al,t}$ are at least $C^{2}$.
Moreover, $|\varphi_{\al,t}'(x)| \le |t| 2^\al (1+\al) \al |x|^{\al-1}$ can be taken
uniformly small, thus satisfying (P'), provided that $|t|\le t_\al$ small. Therefore our results imply that no
transition occurs once one considers the order of contact $\al$ of the indifferent fixed point to increase.
In fact, not only the maximal entropy measure varies differentially with the contact order $\al$ of the
indifferent fixed point as we deduce from  Corollary~\ref{cor:diff}  that the topological pressure
$$
\begin{array}{ccc}
(1,+\infty) \times [-t_\al,t_\al] & \to & \mathbb R \\
 (\al,t) & \mapsto & \Ptop(f_\al, -t\log|Df_\al| )
\end{array}
$$
and the Lyapunov exponent function
$$
\begin{array}{ccc}
(1,+\infty) \times [-t_\al,t_\al] & \to & \mathbb R \\
 (\al,t) & \mapsto & \int \log|Df_\al| \;d\mu_{\al,t}
\end{array}
$$
are differentiable.

\subsubsection{Bifurcations of circle expanding maps}

Let $f$ be a $C^{r+\al}$-expanding map of the circle $S^1$ with degree $d$ and let $p\in S^1$ be a fixed point
for $f$, with $r\ge 2$ and $\al\ge 0$.  Assume that $(f_t)_{t\in[0,1]}$ is a one-parameter family of $C^{r+\al}$-local
diffeomorphisms of the circle such that $f_0=f$, all maps $f_t$ satisfy hypothesis (H1) and (H2) with uniform constants and $f_1$ exhibits a periodic attractor at $p$.

Then, Theorems~\ref{thm:B} and \ref{thm:C} yield that there exists a $C^{3}$-neighborhood $\cF$ of the curve  $(f_t)_{t\in [0,1]}$ and a $C^{\al}$ neighborhood $\cW$ of the constant zero potential such that the pressure function $\cF\times \cW \ni (\tilde f,\tilde \phi) \mapsto \Ptop (\tilde f,\tilde \phi)$ is analytic on $\tilde \phi$
and differentiable in $\tilde f$. Moreover, both the maximal entropy measure function $\tilde f\mapsto \mu_{\tilde f}$ and the Lyapunov exponent function $\tilde f \mapsto \int \log |D\tilde f| \; d\mu_{\tilde f}$ varies differentiably.
In particular, since
$$
t
	\mapsto \dim_H (\mu_{ f_t})= \frac{h_{\mu_{ f_t}}( f_t)}{\int \log |D f_t| \; d\mu_{ f_t}}
		= \frac{\log d }{\int \log |D f_t| \; d\mu_{ f_t}}
$$
then the Hausdorff dimension of the maximal entropy measure is smooth on $t$ along the bifurcation.

\subsection{Higher dimensional examples}

\subsubsection{Derived from expanding maps}\label{ex.saddle}

Let $f_0:\torus^d\to\torus^d$ be a linear expanding map. Fix some
covering $\cU$ by domains of injectivity for $f_0$ and some $U_0\in\cU$ containing a fixed (or
periodic) point $p$. Then deform $f_0$ on a small neighborhood of
$p$ inside $U_0$ by a pitchfork bifurcation in such a way that $p$
becomes a saddle for the perturbed local diffeomorphism $f$. In particular,
such perturbation can be done in the $C^r$-topology, for every $r>0$.
By construction, $f$ coincides with $f_0$ in the complement of $P_1$,
where uniform expansion holds. Observe that we may take the
deformation in such a way that $f$ is never too contracting in
$P_1$, which guarantees that conditions (H1) and (H2) hold.
Since the later are open conditions let $\mathcal F^2$ be a small open neighborhood of $f$
by $C^{2}$ local diffeomorphisms satisfying (H1) and (H2).
Since condition (P') is clearly satisfied by  $\phi\equiv 0$ one can take $\mathcal W^2$ to be an
open set of $C^{2}$-potentials close to zero and satisfying (P') with uniform constants. It follows from
\cite{VV10,CV13} that there exists a unique equilibrium state for $f$ with respect to $\phi$, it has full support,
is has exponential decay of correlations in the space of H\"older observables and that equilibrium states and topological pressure vary continuously with the dynamics.

Concerning higher regularity of these functions it follows from Theorems~\ref{thm:B} and \ref{thm:C} that the pressure function $(f,\phi) \mapsto P(f,\phi)$ is analytical in $\phi$ and differentiable  with respect to $f$, the invariant density function $(f,\phi)\mapsto h_{f,\phi}$  and the equilibrium state function $(f,\phi) \mapsto \mu_{f,\phi}$ are analytical
in $\phi$. Furthermore, if one considers perturbations in the $C^3$-topology then the largest, smallest and
sum of Lyapunov exponents and the metric entropy of the equilibrium states $\mu_{f,\phi}$ vary continuously with
respect to $f$ and $\phi$; the largest, smallest and sum of Lyapunov exponents  of the maximum entropy $\mu_{f,0}$ vary
differentially with respect to $f$.
Finally, the unique measure of maximal entropy $\mu_{f}$ is differentiable with respect to $f$.

\begin{remark}
Let us mention an easy modification of the previous example allows to consider multidimensional
expanding maps with indifferent periodic points. In consequence all the results discussed above hold
also in this context.
\end{remark}

\subsubsection{Non-uniformly expanding repellers though Hopf bifurcations}\label{ex.Hopf}

Hopf bifurcations constitute an important class of bifurcations and arise in many physical phenomena
as e.g. the Selkov model of glycolysis or the Belousov- Zhabotinsky reaction.
We obtain applications also to these class of examples.

 Let $f_0$ be a linear endomorphism on the 2-dimensional torus $M = \mathbb T^2$ with two complex conjugate eigenvalues  $\sigma e^{i\zeta}$ with $\sigma > 3$ and $\theta\in\mathbb R$ satisfying the non-resonance condition $k\zeta\notin
 2\pi\mathbb Z$ for $k\in\{1,2,3,4\}$. Following \cite{HV05} we consider a one parameter family $(\hat f_t)_t$ of $C^5$-local
 diffeomorphisms going through a Hopf bifurcation at $t=0$ in a small neighborhood $V$ of the fixed point corresponding
 to the origin.
 More precisely, in local cylindrical coordinates $(r,\theta)$ in a neighborhood of zero the map $f_0$ can be expressed
 as $f_0(r,\theta)=(\sigma r, \theta+\zeta)$. So, proceeding as in \cite{HV05} we can obtain a one-parameter family
$$
\hat f_t(r,\theta)=(g(t, r^2) \,r , \theta+\zeta)
$$
with $g$ being a real valued $C^\infty$ map on $[-1,1]^2$ and constants $C,\delta>0$ such that $g(t,0)=1-t\le g(t,s)$
for all $s\ge 0$,  that $g(t,s)=\sigma$ whenever $s\ge \delta_0$, that $\partial_s g(t,s)\in(0,C/\delta]$ for all $0\le s <\delta$, and also there exists $1<\sigma_1<\sigma$ and $0<\delta_1<\delta$ satisfying $g(t,s)>\sigma_1$ for all $s>\sigma_1$
and $\partial_s g(t,s)\ge \partial_s g(t,0)$ for $s\in (0,\delta_1]$.
Using the non-resonance condition, for any family $(f_t)_t$ that is $C^5$-close to $(\hat f_t)_t$ there exists a curve of fixed points $(p_t)_t$ close to the origin that also go through a Hopf bifurcation at some parameter $t_*$ (depending continuously of the family) close to zero.

The complement $\Lambda_t$ of the basin of attraction of the periodic attractor $p_t$ is a
repeller. Moreover, since $\Lambda_t$ contains an invariant circle obtained as the boundary of the immediate basin of
attraction of $p_t$ then cannot be a uniformly expanding repeller.
Nevertheless, if $t_0$ is not much larger than $t_*$ then the curve $(f_t)_{t\in[0,t_0]}$ can be assumed to satisfy
conditions (H1) and (H2) with uniform constants. In particular, for any small H\"older continuous potential $\phi$
$$
(t,\phi) \mapsto \Ptop (f_t,\phi)
$$
is differentiable and the equilibrium state $(t,\phi)\mapsto \mu_{f_t,\phi}$ varies continuously.
Moreover, it follows from our results that for any $t\in[0,t_0]$ there exists a unique maximal entropy measure $\mu_t$
with exponential decay of correlations, that $\supp(\mu_t)=M$ for $t\in[0,t_*]$ and $\supp(\mu_t)=\Lambda_t$ is the
repeller for $t\in(t_*,t_0]$. In fact, $t\mapsto \mu_t$ varies differentially along the Hopf bifurcation.

\subsubsection{Large deviations for (non)-uniformly expanding maps}

Assume that $f$ is $C^{2}$ local diffeomorphism and $\Lambda\subset M$ be a transitive and $f$-invariant set
such that $f\mid_\Lambda$  uniformly expanding. In \cite{You90}, Young obtained a large deviations principle
for the unique SRB measure which in our setting generalizes as follows: if $\phi$ is a H\"older continuous
potential then for every $\psi : M \to\mathbb R$ continuous
\begin{align}\label{eq.LdP1}
 \limsup_{n\to \infty} \frac1n & \log \nu_{f,\phi} \left(x \in M : \frac1n S_n \psi(x) \in [a,b] \right)
         \leq - \inf_{s\in[a,b]} K_{f,\psi}(s)
\end{align}
and
\begin{align}\label{eq.LdP2}
 \liminf_{n\to \infty} \frac1n & \log \nu_{f,\phi} \left(x \in M : \frac1n S_n \psi(x) \in (a,b) \right)
         \geq - \inf_{s\in (a,b)} K_{f,\psi}(s)
\end{align}
where $K_{f,\psi}(s)=-\sup\left\{-\Ptop(f,\phi)+h_\eta(f) + \int \phi \,d\eta \colon \int \psi \, d\eta = s\right\}$.
We refer the reader to \cite{Va12} for a proof of the previous assertions and extension for weak Gibbs measures.
Moreover, if $\psi$ is H\"older continuous then Theorem~\ref{thm:differentiability.LDP} yields a large deviation
principle where the rate function $K_{f,\phi}$ in \eqref{eq.LdP1} and \eqref{eq.LdP1} is replaced by $I_{f,\phi,\psi}$,
where $I_{f,\phi,\psi}(s) = \sup_{-t_{\phi,\psi}\le t \le t_{\phi,\psi}} \; \big\{ st-\cE_{f,\phi,\psi}(t) \big\}$ is the Legendre transform of the
free energy function varies continuously.
In particular this proves that the two rate functions  above do coincide in the interval $(\cE'_{f,\psi}(-t_{\phi,\psi}), \cE'_{f,\psi}(t_{\phi,\psi}))$.

Now, take $T_-=\min \{ \int \psi \,d\eta\}$ and $T_+=\max \{ \int \psi \,d\eta\}$ where the minimum and maximum
are taken  over all $f$-invariant measures (we omit the dependence on $f$, $\phi$ and $\psi$ for notational
simplicity). Then for any fixed $t\in (T_-,T_+)$
$$
(f,\phi) \mapsto \sup \left\{ \Ptop(f,\phi)-h_\eta(f)-\int\phi\,d\eta : \eta\in\cM_1(f) \text{ and } \int \psi \,d\eta=t \right\}
$$
is continuous, provided that $\psi$ is H\"older continuous. This illustrates the space of invariant
probability measures is rich for uniformly expanding dynamical systems.

\begin{remark} Our large deviation results apply also to the robust class of multidimensional local diffeomorphisms
$\mathcal F^{2}$ obtained by bifurcations of expanding maps as in Subsections~\ref{ex.saddle} and ~\ref{ex.Hopf} above,
and it yields that for any H\"older continuous observable $\psi$ not cohomologous to a constant  there exists an interval
 $J\subset \mathbb R$ such that
$$
\limsup_{n\to\infty} \frac1n \log \nu_{f,\phi}
	\left(x\in M : \frac1n S_n\psi(x) \in [a,b] \right)
	\le-\inf_{s\in[a,b]} I_{f,\phi,\psi}(s)
$$
and
$$
\liminf_{n\to\infty} \frac1n \log \nu_{f,\phi}
	\left(x\in M : \frac1n S_n\psi(x) \in (a,b) \right)
	\ge-\inf_{s\in (a,b)} I_{f,\phi,\psi}(s).
$$
for all $f\in \cF^{2}$ and $[a,b]\subset J$. Previous to this local large deviations principle some upper and lower bounds were obtained in \cite{Va12}. In addition, for any injectively parametrized family $V \ni v\to f_v$, as in the Hopf bifurcation
construction, the rate function $(s,v) \mapsto I_{f_v,\psi}(s)$ varies continuously with the dynamics and the potential.
\end{remark}

\textbf{Acknowledgements.}
The authors are deeply grateful to the anonymous referee for many useful comments
that helped us to improve the manuscript. We are also grateful to Welington de Melo, Ian Melbourne and the
anonymous referee for indicating some important references related to this work. 
This work was partially supported by CNPq-Brazil, FAPESB and CAPES AEX 18517-12-9.

\bibliographystyle{alpha}

\end{document}